%% file: Flows_of_SU_n_-structures.tex
\documentclass[11pt,a4paper,reqno]{article}
\usepackage{graphicx} %,bbold,bbm,mathbbol,
\usepackage[english]{babel} %for biblatex
\usepackage{csquotes,amsmath} %for biblatex
\usepackage[dvipsnames]{xcolor}
\usepackage{color}
\usepackage{ulem,xcolor}

\usepackage{booktabs}

\newcounter{dummy}
\usepackage{enumitem}
\makeatletter
\newcommand\myitem[1][]{\item[#1]\refstepcounter{dummy}\def\@currentlabel{#1}}
\makeatother
\usepackage[
  hypertexnames = false,
  colorlinks    = true,
  citecolor     = blue,
  linkcolor     = blue,
  urlcolor      = blue,
  linktocpage = true,
  pagebackref
  ]{hyperref}
  
\usepackage[alphabetic,backrefs]{amsrefs}

\usepackage[margin=1.8cm,includehead]{geometry}
\usepackage{amsmath,stmaryrd, upgreek}
\usepackage{amsthm,amssymb}
\usepackage{latexsym}
\usepackage{amscd}
\usepackage{mathrsfs}
\usepackage{url}
\usepackage{mathtools}

\usepackage{cleveref}
\usepackage{doi}
\usepackage{tensor}
\usepackage{cancel}
\usepackage[all]{xy}
\usepackage{multicol}

\usepackage[T1]{fontenc}
\usepackage{tikz-cd}
\usepackage{titlesec}

\usepackage[titletoc,toc,title]{appendix}

\makeatletter %gets rid of Contents in toc
\renewcommand\tableofcontents{%
    \@starttoc{toc}%
}
\makeatother

\usepackage{enumitem}
\usepackage{slashed}
\graphicspath{}
\usepackage{fancyhdr} %gives heading and author name on alternate pages

\usepackage{changepage}

\newcommand\shorttitle{Flows of geometric structures II} %title which appear on alternate pages
\newcommand\authors{Fadel--Fowdar--Loubeau--Moreno--S\'a Earp} %author name appears on alternate pages

\newcounter{commentCounter}
\titleformat{\section}    
       {\normalfont\large\bfseries\center}{\thesection.}{1em}{}
\makeatletter
\newcommand*{\rom}[1]{\expandafter\@slowromancap\romannumeral #1@}
\makeatother

\input{preamble}
\usepackage{nicematrix}

\begin{document}

\title{\textbf{Flows of geometric structures II}}

\author{Daniel Fadel, Udhav Fowdar, Eric Loubeau, Andrés J. Moreno and Henrique N. S{\'{a}} Earp}
\date{\today}

\maketitle

\begin{abstract}
We advance the general theory of flows of tensorial $\H$-structures, focusing on non-isometric flows and on the case $\H=\SU(m)\subset\SO(2m)$. After developing the relevant $\SU(m)$ algebra, we compare two natural evolutions: the unrestricted negative gradient flow of the intrinsic-torsion energy and a Ricci-harmonic flow. We prove short-time existence and uniqueness for the Ricci-harmonic $\H$-flow, with arbitrary lower-order torsion-quadratic terms, for every closed subgroup 
$\H\subset\SO(n)$. For groups for which the projection to $\fh^\perp$ defines a $4$-form, including $\{1\}$, $\SU(2)$, $\mathrm G_2$, and $\mathrm{Spin}(7)$, we express the negative gradient flow in Ricci-harmonic form up to explicit lower-order torsion terms and prove short-time existence and uniqueness by a modified DeTurck argument. We treat the genuinely different $\SU(m)$ case by a separate principal-symbol computation, proving short-time existence and uniqueness for the unrestricted negative gradient flow of $\SU(m)$-structures. The same computation identifies the natural negative gradient flow of $\U(m)$-structures as a borderline case, which cannot be made strictly parabolic by first-order diffeomorphism gauges. For the modified Ricci-harmonic flow, we derive heat-type evolution equations for the intrinsic torsion, a doubling-time estimate and Shi-type derivative estimates for $(|\rR\rmm|^2+|\nabla T|^2+|T|^4)^{1/2}$, and a finite-time continuation criterion. In dimension six, we translate the formalism into the standard torsion forms of an $\SU(3)$-structure and describe, to highest order, the corresponding family of second-order quasilinear $\SU(3)$-flows.
\end{abstract}

%\newpage
\begin{adjustwidth}{0.95cm}{0.95cm}
    \tableofcontents
\end{adjustwidth}

%%%%%%%%

\section{Introduction}

Let $M^n$ be an oriented smooth manifold and let $\H\subset\SO(n)$ be a
closed subgroup.  An $\H$-structure on $M$ may be described either as an
$\H$-reduction of the oriented frame bundle or, whenever $\H$ is realised as the
stabiliser of a model tensor, as a tensor field  lying pointwise in a prescribed
$\GL(n,\mathbb R)$-orbit.  This language includes almost Hermitian structures,
special unitary structures, $\mathrm G_2$-structures, $\mathrm{Spin}(7)$-structures,
and several other geometries governed by reductions of the structure group.
Our preceding paper  \cite{Fadel2022} developed a general formalism for
flows of such tensorial $\H$-structures, expressing arbitrary infinitesimal
deformations in terms of the diamond action of $\mathfrak{gl}(n,\mathbb R)$,
and obtaining general formulae for the induced metric, the intrinsic torsion,
and their evolution under an arbitrary $\H$-flow.  That framework was then applied
to the harmonic $\H$-flow, namely the negative gradient flow of the intrinsic
torsion energy restricted to a fixed isometric class.

The present paper extends this programme to non-isometric flows.  It aims to systematise and extend our  analytical understanding of the short- and long-time behaviour of geometric structure flows, based on a number of recent developments in specific settings, such as the study of general flows of
$\mathrm G_2$-structures and their second-order differential invariants in
\cite{Dwivedi2023}, the negative gradient flow of $\mathrm{Spin}(7)$-structures in
\cite{Dwivedi2025gradient}, the Ricci-harmonic flow of $\mathrm G_2$- and
$\mathrm{Spin}(7)$-structures in \cite{dwivedi2026ricci}, the recent analytic theory of reasonable $\mathrm{Spin}(7)$-flows in \cite{Duthie2026reasonableSpin7}, and the classification
programme for flows of $\SU(2)$-structures in \cite{udhav2024}. Taken together, these works show that, once the induced metric is allowed to vary, the parabolicity problem is governed simultaneously by diffeomorphism invariance, curvature terms, torsion-divergence terms, and the representation theory of the structure group. We therefore  situate these phenomena within the general $\H$-structure formalism, distinguishing the identities and principal-symbol features that are uniform in $\H$ from those that reflect special algebraic properties of the homogeneous space $\SO(n)/\H$, and thereby clarifying the role of this homogeneous algebraic structure in the parabolicity of the corresponding (non-isometric) flows.

The special unitary case is our guiding example: we view $\SU(m)\subset\SO(n)$, where $n=2m$.  An $\SU(m)$-structure is an almost Hermitian structure together with a compatible complex volume form.  At the model level, let $J_{\circ}$ denote the standard complex structure on $\mathbb R^{2m}$.  The orthogonal complement
\[
    \fm=\mathfrak{su}(m)^\perp\subset\mathfrak{so}(2m)
\]
splits as the sum of two inequivalent $\SU(m)$-modules,
\[
    \fm
    =
    \fm_1\oplus\fm_2,
    \qquad
    \fm_1:=\langle J_{\circ}\rangle
    \quad\text{and}\quad
    \fm_2:=\mathfrak u(m)^\perp
    =
    \{A\in\mathfrak{so}(2m):AJ_{\circ}=-J_{\circ}A\}.
\]
Accordingly, the intrinsic torsion of an $\SU(m)$-structure decomposes into an almost Hermitian component, with values in $\fm_2$, and a determinant component, with values in $\fm_1$, measuring the further reduction from $\U(m)$ to $\SU(m)$. Since these summands are inequivalent, the inner product on $\fm$ induced by the diamond action may, in principle, weight them by independent constants. Thus, writing $\xi$ for the tensorial $\SU(m)$-structure, $T$ for its intrinsic torsion, and $\nabla$ for the Levi-Civita connection of the induced metric, the single-constant identity relating $|\nabla\xi|^2$ and $|T|^2$ in the main homogeneous examples of \cite{Fadel2022} is replaced here by a two-component comparison; see \cite{Fadel2022}*{Example 1.16}. This is not merely a technical complication in $\SU(m)$-geometry; it provides a useful test case for understanding general $\H$-flows when the complement
$\mathfrak h^\perp$ is not isotropy-irreducible. Indeed, we begin by developing the basic $\SU(m)$ identities, but the subsequent comparison between the negative gradient and Ricci-harmonic flows is carried out for a
large class of subgroups $\H\subset\SO(n)$. The former is the unrestricted negative gradient flow of
\[
    \mathcal E(\xi)=\frac12\int_M |T_\xi|^2\,\vol_{g_\xi},
\]
where both the tensor and its induced metric are allowed to vary.  The latter is the coupling of the Ricci flow of the induced metric with
the harmonic flow inside the evolving isometric class.  The unrestricted functional
$\mathcal E$ is scale-degenerate in real dimensions greater than two: as shown in 
\cite{FadelLoubeau2026}, its infimum is zero on every nonempty path component of the full space of $\H$-structures, and its critical points are precisely the torsion-free structures.  We reproduce the direct Euler--Lagrange proof of the critical-point statement in \S\ref{sec: modifying Ricci-harmonic}, because the trace identity obtained there indicates how to add lower-order terms to the Ricci-harmonic flow, so that its stationary points are also actually torsion-free in several important cases.

We now state the main results in the order in which they appear. Unless explicitly stated otherwise, $M$ is assumed to be closed, and the general flow of an $\H$-structure $\xi$ is denoted by 
\[
\partial_t\xi = A \diamond \xi, 
\qforq A=S+C\in\Sigma^2\oplus\Lambda^2_{\mathfrak m},
\]
where ${\mathfrak m}:=\fh^{\perp}\subset\Lambda^2$ is the reductive complement of $\fh$.

\begin{propx}[Proposition~\ref{prop: comparing negative gradient flows}]
\label{prop: intro comparison flows}
Let $\xi$ be an $\H$-structure whose orthogonal projection
$\pi^2_{\mathfrak m}:\Lambda^2\to\Lambda^2_{\mathfrak m}$ is written in the
form \eqref{equ: projection map pi2m}.  If the corresponding tensor $\Xi$ is a
$4$-form, then the negative gradient flow of $\mathcal E$ can be written as
\begin{equation*}
    \partial_t\xi
    =\Big(\mu\big(-2c_{\H}\Ric(g)-\mathcal L_{VT}g+\widehat Q\big)
      +\Div T\Big)\diamond\xi,
\end{equation*}
where $c_{\H}>0$, $VT_a:=T_{k,ka}$, and the term $\widehat Q$ is quadratic in the torsion. 

With the normalisations of Table~\ref{table_1}, this applies in particular to
$\H=\{1\}$, $\SU(2)$, $\mathrm G_2$, and $\mathrm{Spin}(7)$.  For
$\H=\U(m)$ and $\H=\SU(m)$, with $n=2m$, the corresponding formulae are
\begin{align*}
    \partial_t\xi
    &=\Big(\mu\big(-\Ric(g)+\frac12\sym(\Ric^*)
      -\mathcal L_{VT}g+\widehat Q\big)+\Div T\Big)\diamond\xi,\qquad \H=\U(m),\\
    \partial_t\xi
    &=\Big(\mu\big(-\Ric(g)+\frac{m-2}{2m}\sym(\Ric^*)
      -\mathcal L_{VT}g+\widehat Q\big)+\Div T\Big)\diamond\xi,\qquad \H=\SU(m),
\end{align*}
where $\Ric^*_{ab}:=R_{iajk}\omega_{ik}\omega_{bj}$.
\end{propx}

The proposition shows precisely where the almost Hermitian and
special unitary cases depart from the $\Xi$-skew cases, in which the
curvature contribution reduces to the Ricci tensor by the algebraic Bianchi
identity.  In the $\U(m)$ and $\SU(m)$ cases, the additional term
$\sym(\Ric^*)$ remains, and its presence must be taken into account in the
principal-symbol analysis.

\begin{thmx}[Theorem~\ref{thm: evolution of torsion for div T} and Corollary~\ref{cor: main result of section 3}]
\label{thm: intro torsion evolution}
Let $\xi(t)$ be an $\H$-flow for which the skew-symmetric component of the
velocity is $C=\Div T$.  Then the intrinsic torsion satisfies
\begin{equation*}
    \partial_tT_{l,ij}
    =\Delta T_{l,ij}-\pi^2_{\mathfrak m}(\Lambda\nabla\Ric_l)_{ij}
      -\pi^2_{\mathfrak m}(\Lambda\nabla S_l)_{ij}
      +(S\diamond T_l)_{ij}+R*T+\nabla T*T+T*T*T.
\end{equation*}
Consequently, for the Ricci-harmonic $\H$-flow, with arbitrary lower-order terms
quadratic in $T$, one has the heat-type equation
\begin{equation*}
    \partial_tT=\Delta T+R*T+\nabla T*T+T*T*T.
\end{equation*}
For the general unrestricted negative gradient flow, additional highest-order terms
remain, involving $\Lambda\nabla\Ric$, $d(VT)$ and, in the $\U(m)$ and $\SU(m)$ cases, $\Lambda\nabla\sym(\Ric^*)$.
\end{thmx}

This result is one of the main reasons for focusing on the Ricci-harmonic flow. Although the negative gradient flow is variationally natural, its torsion evolution is not uniformly heat-like. By contrast, the Ricci-harmonic flow is better adapted to derivative estimates, because the leading torsion term is the rough Laplacian.

\begin{thmx}[Theorem~\ref{thm: STE Ricci harmonic H}]
\label{thm: intro STE Ricci harmonic H}
Let $\xi_0$ be a smooth $\H$-structure for a closed  subgroup $\H\subset\SO(n)$.
Then the Ricci-harmonic $\H$-flow
\begin{equation*}
    \partial_t\xi=(-\Ric+\Div T)\diamond\xi
\end{equation*}
admits a unique smooth solution on some interval $[0,\varepsilon)$.
Furthermore, the assertion holds after adding arbitrary natural lower-order terms whose linearisation has trivial second-order principal symbol; in particular, it
holds for the modified Ricci-harmonic flow
\eqref{eq: modified Ricci-harmonic H flow}.
\end{thmx}

The proof is a direct DeTurck argument and does not require the algebraic assumption that $\Xi$ be a $4$-form.  Adding the Lie derivative term $\mathcal L_{2X_{DT}}\xi$, the Ricci and DeTurck contributions have scalar symbol $|\chi|^2S$ in the symmetric directions, while the torsion-divergence term, together with $dX_{DT}$, has scalar symbol $|\chi|^2C$ in the
$\Lambda^2_{\mathfrak m}$ directions.  Thus the modified operator has principal
symbol $|\chi|^2A$.

\begin{thmx}[Theorem~\ref{thm: STE skew case}]
\label{thm: intro STE skew}
    Let $\xi_0$ be a smooth $\H$-structure, and assume that the tensor $\Xi$ in \eqref{equ: projection map pi2m} is a $4$-form and that $c_{\H}\in(0,1]$.  Then there exists $\varepsilon>0$ such that the negative gradient $\H$-flow \eqref{equ: general gradient flow}, with weight $\mu=(2c_{\H})^{-1}$, admits a unique smooth solution on  $[0,\varepsilon)$.
    In particular, by Table~\ref{table_1}, this applies to
$\H=\{1\}$, $\SU(2)$, $\mathrm G_2$, and $\mathrm{Spin}(7)$.
\end{thmx}

The proof is again a DeTurck argument, now using the special algebraic form of
$\pi^2_{\mathfrak m}$.  The necessary principal symbols of the Ricci, DeTurck,
torsion, vector-torsion, and torsion-divergence terms are collected in
Proposition~\ref{prop: principal symbols}.  In the $\Xi$-skew case, choosing
$\mu=(2c_{\H})^{-1}$ and adding the appropriate DeTurck vector field, coupled with
$VT$, gives a strictly parabolic modified system.  Pulling back by the diffeomorphisms it generates then gives the solution to the original geometric flow.

\begin{thmx}[Theorem~\ref{thm: ste SUn}]
\label{thm: intro STE SUn}
Let $\xi_0=(g_0,J_0,\Upsilon_0)$ be an $\SU(m)$-structure on
$M^{2m}$.  Then there exist $\varepsilon>0$ and a unique family of smooth 
$\SU(m)$-structures $\{\xi(t)=(g(t),J(t),\Upsilon(t))\}_{t\in[0,\varepsilon)}$, 
with $\xi(0)=\xi_0$, solving the unrestricted negative gradient flow
\eqref{equ: negative grad flow U(m)_SU(m)} with $\lambda_{\H}=1$.
\end{thmx}

The proof requires a separate symbol computation, because the $\Ric^*$ term in
the negative gradient flow of $\SU(m)$-structures is not present in the $\Xi$-skew cases.
Lemma~\ref{lem: symbols_SU(m)} computes the symbols of $\Ric^*$, $\nabla VT$,
$\Div T$, and $\nabla J\eta$.  The DeTurck vector field is then chosen so that
the modified operator has a positive symbol.

\begin{propx}[Proposition~\ref{prop: U(m) endpoint degeneracy}]
\label{prop: intro Um endpoint degeneracy}
For the natural non-isometric negative gradient $\U(m)$-flow, the
principal symbol has non-zero null directions that are transverse to the principal symbol of the infinitesimal diffeomorphism action.  Moreover, along these directions the principal contribution of every first-order diffeomorphism
gauge term $\mathcal L_X\xi$ vanishes.  Consequently, no first-order DeTurck
modification can make this flow strictly parabolic in the principal
quadratic-form sense used in the proof of Theorem~\ref{thm: ste SUn}.
\end{propx}

Thus, the $\U(m)$ case is not merely omitted from Theorem~\ref{thm: ste SUn}.
It is a genuine endpoint obstruction singled out by the same symbol calculation
that proves the theorem for $\SU(m)$.  This also explains why the
(isometric) harmonic $\U(m)$-flow from \cite{Fadel2022} and the
unrestricted (non-isometric) negative gradient $\U(m)$-flow have different parabolicity behaviour.

\begin{thmx}[Theorem~\ref{thm: Shi-estimates}]
\label{thm: intro Shi}
Let $\{\xi(t)\}$ be a solution of the modified Ricci-harmonic flow
\eqref{eq: modified Ricci-harmonic H flow}, for
$t\in[0,1/K]$, and suppose
\begin{equation*}
    \Lambda(x,t) :=\big(|\rR\rmm|^2 +|\nabla T|^2 +|T|^4\big)^{1/2}\le K,
    \qquad\text{on }M\times[0,1/K].
\end{equation*}
Then, for every $k\in\mathbb N$, there is a constant $C_k$ such that
\begin{equation*}
    |\nabla^k\rR\rmm|+|\nabla^{k+1}T|
    \le C_k K t^{-k/2},
    \qquad\text{for }t\in(0,1/K].
\end{equation*}
\end{thmx}

The proof follows the strategy used in the Ricci-flow and Laplacian-flow literature: once the coupled
quantity $\Lambda$ is controlled, the evolution equations for curvature and
torsion yield parabolic inequalities for all higher derivatives.  The doubling-time
estimate (Proposition~\ref{prop: doub_time_est}) is a first consequence of the basic inequality for $\Lambda$.  In the $\mathrm{Spin}(7)$ case, the same curvature--torsion quantity is the organising object in Duthie's recent work on reasonable flows, where analogous Shi-type estimates are developed in a broader $\mathrm{Spin}(7)$-specific class and then used for compactness and finite-time singularity analysis~\cite{Duthie2026reasonableSpin7}.

\begin{thmx}[Theorem~\ref{thm: Lambda quantity for LTE}]
\label{thm: intro continuation}
Let $\{\xi(t)\}$ be a solution of the modified Ricci-harmonic flow
\eqref{eq: modified Ricci-harmonic H flow}, defined on a
maximal time interval $[0,T_0)$ with $T_0<\infty$.  If
\begin{equation*}
    \Lambda(t):=\sup_M\big(|\rR\rmm|^2+|\nabla T|^2+|T|^4\big)^{1/2},
\end{equation*}
then
\begin{equation*}
    \lim_{t\nearrow T_0}\Lambda(t)=\infty,
    \qquad
    \Lambda(t)\ge \frac{C}{T_0-t},
\end{equation*}
for some constant $C>0$.
\end{thmx}

Thus, finite-time singularity formation for the modified Ricci-harmonic flow is
detected by the same type of curvature-torsion quantity that appears in the
Shi estimates.  This is the analogue, in the $\H$-structure setting, of the standard continuation mechanism for Ricci-type flows.

We also obtain several auxiliary results that are not part of the main analytic
chain above but are useful both in applications and for comparison with other
geometries.  Section~\ref{sec: H-structures} contains the $\SU(m)$ algebra,
including the formula for the intrinsic torsion in terms of $\nabla J$ and
$\nabla\Upsilon$, the associated Bianchi and curvature identities, and the
compact L\"ust--Tsimpis dichotomy, Theorem~\ref{thm: compact LT}.  In real
dimension six, \S\ref{sec: SU3-torsion forms - appendix} translates the
endomorphism formalism into the standard torsion forms of an $\SU(3)$-structure
and gives explicit curvature and divergence formulae.  As an invariant-theoretic
application, Theorem~\ref{thm: general second order SU(3) flow} gives a
$27$-parameter family of second-order quasilinear $\SU(3)$-flows at highest
order.

The paper is organised as follows.  Section~\ref{sec: H-structures} develops
the $\SU(m)$ calculus and describes the L\"ust--Tsimpis example.  Section~\ref{sec: the general setup}
recalls the general $\H$-flow formalism, while
Section~\ref{sec: general evolution of torsion and curvature} derives the
curvature and torsion evolution formulae used later.  Section~\ref{sec: negative gradient flow versus Ricci-harmonic flow}
compares the unrestricted negative gradient and Ricci-harmonic flows. Together with the preceding principal-symbol computation, \S\ref{sec: ste for sun} proves the short-time existence
results.  Section~\ref{sec: Ricci-harmonic} proves the doubling-time estimate,
the Shi estimates, and the continuation criterion for the modified Ricci-harmonic flow.
Finally, \S\ref{sec: SU3-torsion forms - appendix} records the $\SU(3)$
torsion-form computations and the second-order invariant-theoretic
classification.

\paragraph{Notation and conventions.}

We follow the conventions of \cite{Fadel2022}, and $M^n$ denotes a real smooth manifold of dimension $n>2$. In the almost Hermitian and special unitary cases, we write $n=2m$ and regard $\mathrm{U}(m),\mathrm{SU}(m)\subset \mathrm{SO}(n)$; thus, $m\geq 2$ denotes the
complex rank, while $n=\dim_{\mathbb R}M$ remains the ambient real dimension.

We use $\Lambda^kT^*M$, $\Lambda^k$, $\Lambda^k_{\mathfrak m}$ and
$\Lambda_J^{p,q}$ for vector bundles or, when the discussion is pointwise, for their fibres. Their spaces of smooth sections are denoted by $\Omega^k(M)$,
$\Omega^k_{\mathfrak m}(M)$ and $\Omega_J^{p,q}(M)$, respectively.  For a vector
bundle or representation $E$, $S^2E$ denotes the symmetric square of $E$, while
$\Gamma(S^2E)$ denotes its space of smooth sections when needed.  In the special
case of symmetric $2$-tensors on $M$, we write $\Sigma^2$ for the bundle
$S^2T^*M$ and $\Sigma^2(M)$ for its space of smooth sections; likewise
$\Sigma^2_0$ and $\Sigma^2_0(M)$ denote the trace-free bundle and its sections.
When the base manifold or the underlying bundle is clear from the context, we
sometimes omit it.

For any $A\in \Gamma(\End(TM))$, we lower the output index according to $A_{ik}=A_i{}^j g_{jk}$. 
Hence, for $A,B\in \Gamma(\End(TM))$, composition is written as
\[
    (AB)_{il}=B_{ij}g^{jk}A_{kl}.
\]
The diamond operator $\diamond$ is the infinitesimal action induced by the
right $\GL(n,\mathbb R)$-action on tensors. In particular, for a one-form
$\eta$, one has
\[
    (A\diamond\eta)_l=A^k{}_l\eta_k.
\]
Whenever $J$ acts on a differential form without explicitly writing
$\diamond$, we use the standard pull-back convention
\[
    (J\alpha)(X_1,\ldots,X_k):=(J^*\alpha)(X_1,\ldots,X_k)
    =\alpha(JX_1,\ldots,JX_k).
\]
For one-forms this agrees with the diamond action, and throughout the
paper we therefore write
\begin{equation}\label{eq: convention J eta covector}
    J\eta:=J\diamond\eta=J^*\eta,
    \qquad
    (J\eta)_l=J^k{}_l\eta_k=\eta_k\omega_{lk},
\end{equation}
where $\omega_{ij}=g(J\partial_i,\partial_j)=J_i{}^k g_{kj}$ in local
coordinates. Notice that for higher-degree forms, $J\alpha=J^*\alpha$
denotes the pull-back action on all entries, whereas $J\diamond\alpha$ denotes
the infinitesimal diamond action.

We also keep the Laplacian convention of \cite{Fadel2022}: $\Delta$
denotes the negative rough Laplacian on tensors.

\paragraph{Acknowledgements.}

The authors acknowledge the support of the Agence Nationale de la Recherche (ANR)--São Paulo Research Foundation (FAPESP) collaboration \textit{BRIDGES---Brazil--France Interplays in Gauge Theory, Extremal Structures and Stability} (ANR grant ANR-21-CE40-0017 and FAPESP grant 2021/04065-6), and of the Brazilian Federal Agency for Support and Evaluation of Graduate Education (CAPES) MATH-AmSud collaboration \textit{Symmetries in Geometry and Physics} (SGP 24-MATH-12).

\textbf{DF} was supported by SGP; by the Brazilian National Council for Scientific and Technological Development (CNPq), under Universal grant 406666/2023-7; by BRIDGES; and by the University of São Paulo through its \textit{Programa de Apoio aos Novos Docentes USP}.

\textbf{EL} was supported by BRIDGES and SGP.

\textbf{UF} was partially supported by FAPESP grant 2023/12372-1, associated with BRIDGES.

\textbf{AM} was supported by FAPESP grant 2021/08026-5, and by BRIDGES.

\textbf{HSE} was supported by FAPESP  grants 2020/09838-0, associated with BI0S---Brazilian Institute of Data Science; 2021/04065-6, associated with BRIDGES; and 2024/00923-6, associated with CBG---Brazilian Centre for Geometry. HSE was also supported by CNPq Research Productivity Fellowship 307145/2025-5, level PQ-A.

%\newpage
\section{Geometry of \texorpdfstring{$\SU(m)$}{}-structures}
\label{sec: H-structures}

We recall some preliminaries on $\SU(m)$-geometry, recalling that 
$n=2m$ is the underlying real dimension. Some of the results
presented here were previously obtained in \cites{bor2001,cabrera2005,martin2006}.
We revisit them following the ansatz proposed in \cite{Fadel2022}, with a view to investigating
general flows of $\mathrm{SU}(m)$-structures.

\subsection{Algebraic properties of \texorpdfstring{$\SU(m)$}{}-structures}

Let $(g,J)$ be an almost Hermitian structure on $M^n$, i.e., $g$ is a Riemannian metric and $J$ is a $g$-compatible almost complex structure satisfying $g(J\cdot,J\cdot)=g(\cdot,\cdot)$. Denote by $TM^{\bC}=TM\otimes_\bR \bC$ the complexified tangent vector bundle. Since $J^2=-I$, the almost complex structure $J$ has eigenvalues $i$ and $-i$; specifically, at each $p\in M$,
\begin{align*}
    T^{1,0}:=T_p^{1,0}M=\{X-iJX; \quad X\in T_pM\} \qandq  T^{0,1}:=T_p^{0,1}M=\{X+iJX; \quad X\in T_pM\}
\end{align*}
are the respective eigenspaces. With respect to the decomposition $T_pM^\bC=T_p^{1,0}M+T_p^{0,1}M$, we have $$\Lambda^1(T_pM^\bC)^\ast=\Lambda_J^{1,0}\oplus\Lambda_J^{0,1},$$ 
where $\Lambda_J^{1,0}$ is the annihilator of $T_p^{0,1}M$ and $\Lambda_J^{0,1}=\bar{\Lambda_J^{1,0}}$. Following the notation of \cites{falcitelli1994, Salamon1989}, the space of $(p+q)$-forms contains the subspace $\Lambda_J^{p,q}\simeq \Lambda^p(T^{1,0})^\ast\otimes \Lambda^q(T^{0,1}) $ of forms of $(p,q)$-type. The spaces $\Lambda_J^{p,q}\oplus \Lambda_J^{q,p}$ and $\Lambda_J^{p,p}$ are complexifications of the real vector spaces denoted by $[\![\Lambda_J^{p,q}]\!]$ and $[\Lambda_J^{p,p}]$, such that $\dim_\bR[\![\Lambda_J^{p,q}]\!]=2 \dim_\bC\Lambda_J^{p,q}$ and $\dim_\bR [\Lambda_J^{p,p}]=\dim_\bC\Lambda_J^{p,p}$, respectively. Thus, there is an isomorphism of real vector spaces
\begin{align}
    \Lambda^{2k}(T_pM)^\ast=\bigoplus_{p=0}^{k-1}[\![\Lambda_J^{2k-p,p}]\!]\oplus [\Lambda_J^{k,k}] \qandq \Lambda^{2k+1}(T_pM)^\ast=\bigoplus_{p=0}^k[\![\Lambda_J^{2k+1-p,p}]\!],
\end{align}
where each summand is irreducible under the action of $\GL(m,\bC)\subseteq\GL(n,\bR)$.

\begin{definition}\label{def: SUn-structure}
Let $m\geq2$, and let $(M^{2m},J,g)$ be an oriented almost Hermitian manifold with a complex form $\Upsilon\in \Omega_J^{m,0}(M)$. The triple $(g,J,\Upsilon)$ is called an \emph{$\SU(m)$-structure} if $|\Upsilon|_g=2^{m/2}$, where $|\cdot |_g$ denotes the norm induced by the $\bC$-linear extension of $g$.
\end{definition}
Alternatively,  Definition \ref{def: SUn-structure} can be reformulated in terms of the fundamental 2-form $\omega(X,Y)=g(JX,Y)$ associated with the almost Hermitian structure $(g,J)$. 
\begin{definition}
   Let $m\geq 2$.
   An \emph{$\SU(m)$-structure} on $M^{2m}$ is a triple $(g,J,\Upsilon)$ satisfying
   \begin{equation}\label{eq: Volume SUn identity}
       \Upsilon\wedge \bar{\Upsilon}=2^m(-1)^{\frac{m(m+1)}{2}}i^m\frac{\omega^m}{m!},
   \end{equation}
   where $\Upsilon\in  \Omega_J^{m,0}(M)$ and $\omega$ is the fundamental $2$-form associated with the almost Hermitian structure $(g,J)$, i.e., $\omega(X,Y)=g(JX,Y)$ for any $X,Y\in \sX(M)$.
\end{definition}

The complex volume form can be written as $\Upsilon=\Upsilon^++i\Upsilon^-$, where $\Upsilon^\pm$ are smooth sections of $[\![\Lambda_J^{m,0}]\!]$ satisfying
  \begin{equation}\label{eq: relation Upsilon+_and_Upsilon-}
      X\lrcorner \Upsilon^+=(JX)\lrcorner \Upsilon^- \qandq  X\lrcorner \Upsilon^-=-(JX)\lrcorner \Upsilon^+ \qforq X\in \sX(M).
  \end{equation}
  The vanishing $\Upsilon\wedge \Upsilon=0$ implies
    \begin{equation*}
        \Upsilon^+\wedge \Upsilon^+-\Upsilon^-\wedge \Upsilon^-=0 \qandq \Upsilon^+\wedge \Upsilon^-+\Upsilon^-\wedge \Upsilon^+=0.
    \end{equation*}
 Combining these identities with \eqref{eq: Volume SUn identity}, we obtain the identities (see \cite{martin2006}*{Lemma 2.1}):
    \begin{align}\label{eq: volume_m_odd}
        \Upsilon^+\wedge \Upsilon^-=&2^{m-1}\vol \qandq \Upsilon^\pm\wedge\Upsilon^\pm=0 \qforq m \quad \text{odd},\\ \label{eq: volume_m_even}
        \Upsilon^\pm\wedge \Upsilon^\pm=&2^{m-1}\vol \qandq \Upsilon^+\wedge\Upsilon^-=0 \qforq m \quad \text{even}.
    \end{align}
Notice that the normalisation $|\Upsilon|_g^2=2^m$ implies  $|\Upsilon^\pm|_g^2=2^{m-1}$; this is readily checked using  $\Upsilon^+=\frac{\Upsilon+\bar{\Upsilon}}{2}$ and $\Upsilon^-=\frac{\Upsilon-\bar{\Upsilon}}{2i}$. 
Thus, from \eqref{eq: volume_m_odd} and \eqref{eq: volume_m_even}, 
\begin{align}
\label{eq: star_Upsilon_m_odd}
    *\Upsilon^+=\Upsilon^- \qandq *\Upsilon^-&=-\Upsilon^+, \qforq m \quad \text{odd},\\ 
\label{eq: star_Upsilon_m_even}
     *\Upsilon^+=\Upsilon^+ \qandq *\Upsilon^-&=+\Upsilon^-, \qforq m \quad \text{even}.
\end{align}
The case $m=3$ is rather special since the $\mathrm{GL}(6,\mathbb{R})$ orbit of $\Upsilon^+$ in $\Lambda^3(M^6)$ is open. Using this, Hitchin showed   that one can recover the entire $\mathrm{SU}(3)$-structure from $\omega$ and $\Upsilon^+$ alone \cite{Hitchin2000}. 

\subsection{Infinitesimal deformations of \texorpdfstring{$\SU(m)$}{}-structures}

Let $\xi$ be any $(p,q)$-tensor field on $M$. 
In local coordinates, we write
\[
    \xi = \xi^{i_1 \dots i_p}_{j_1 \dots j_q} \frac{\partial}{\partial x^{i_1}}\otimes \dots\otimes\frac{\partial}{\partial x^{i_p}}\otimes dx^{j_1}\otimes\dots \otimes dx^{j_q},
\]
where $\xi^{i_1 \dots i_p}_{j_1 \dots j_q}:=\xi(dx^{i_1},\dots,dx^{i_p};\frac{\partial}{\partial x^{j_1}},\dots,\frac{\partial}{\partial x^{j_q}})$ are smooth local functions. The canonical right $\GL(n,\bR)$-action on tensors on $\mathbb{R}^{n}$ extends naturally pointwise to tensors on $M$.
This induces the following \emph{infinitesimal
action} of endomorphisms $A\in\Gamma(\mathrm{End}(TM))$ on $(p,q)$-tensor fields given by 
\begin{align}
\label{eq: diamond_operator}
    A\diamond \xi &:=\left.\ddt\right\vert_{t=0} e^{tA}.\xi \nonumber\\
    &= \xi^{i_1 \dots i_p}_{j_1 \dots j_q} 
    \sum_{r,s=1}^{p,q} \Big\{
    - \frac{\partial}{\partial x^{i_1}}\otimes \dots \otimes A\frac{\partial}{\partial x^{i_r}}\otimes \dots\otimes \frac{\partial}{\partial x^{i_p}}\otimes dx^{j_1}\otimes\dots \otimes dx^{j_q}
    \\
    &\quad\quad+
    \frac{\partial}{\partial x^{i_1}}\otimes \dots\otimes \frac{\partial}{\partial x^{i_p}}\otimes dx^{j_1}\otimes\dots \otimes A^{\ast}dx^{j_s}\otimes\dots\otimes dx^{j_q}
    \Big\}\nonumber.
\end{align} 
Writing $A=(A^i_j)\in \mathfrak{gl}(n, \mathbb{R})$ (pointwise) in these coordinates and rearranging the terms in \eqref{eq: diamond_operator}, one has
\begin{equation*}
    (A\diamond\xi)^{i_1\dots i_p}_{j_1\dots j_q}
    =-\sum_{r=1}^p A^{i_r}{}_{\mu}\xi^{i_1\dots \mu \dots i_p}_{j_1 \dots j_q}
    +\sum_{s=1}^q A^{\mu}{}_{j_s}\xi_{j_1 \dots \mu \dots j_q}^{i_1 \dots i_p}.
\end{equation*} 
More generally, if $\xi=(\xi_1,\ldots,\xi_k)$ is a multi-tensor, we define the \emph{diamond operator} by the component-wise infinitesimal action of $A\in\Gamma(\End(TM))$:
\[
A\diamond\xi := (A\diamond\xi_1,\ldots,A\diamond\xi_k).
\]
\begin{remark}\label{rmk: Complex form of A}
If $\xi$ is a complex $(p,q)$-tensor field, written as $\xi=\xi^++i\xi^-$ with $\xi^\pm$ real $(p,q)$-tensor fields, then for any $A\in \Gamma(\End(TM))$ we define
$$
  A\diamond \xi=A\diamond\xi^++i(A\diamond\xi^-).
$$
\end{remark}

From the $\SU(m)$-structure $(g,J,\Upsilon)$, the metric induces the following decomposition at the level of sections:
$$
\Gamma(\End(TM))\simeq\Omega^0(M)g\oplus\Sigma^2_0(M)\oplus\Omega^2(M),
$$
where $\Omega^0(M)g$ is the trivial summand spanned by $g$, and $\Sigma_0^2(M)$ denotes the space of traceless symmetric bilinear forms. In addition, the compatible almost complex structure $J$ induces the bundle decompositions 
$$
\Sigma_0^2= U\oplus V \qandq \Lambda^2=\Lambda^2_{\fu(m)}\oplus \Lambda^2_{\fu(m)^\perp},
$$ 
where  the subbundles $U$  and $V$ consist of traceless symmetric endomorphisms commuting and anti-commuting with $J$, respectively. In particular, contracting with $J$ yields an isomorphism $U \cong \Lambda^2_{\mathfrak{su}(m)}$. 

From the reductive decomposition $\fso(n)=\fso(2m)=\fsu(m)\oplus\fm_1\oplus\fm_2$, where $\fm_1$ is the trivial $\SU(m)$-submodule generated by $J$ and $\fm_2=\fu(m)^\perp$ is  the $\SU(m)$-submodule of skew-symmetric matrices anti-commuting with $J$, the bundle of $2$-forms satisfies $\Lambda^2=\Lambda^2_{\fsu(m)}\oplus\Lambda^2_{\fm_1}\oplus \Lambda^2_{\fm_2}$. Altogether, at the level of smooth sections, we have the decomposition 
\begin{equation}\label{eq: decomposition End(TM)}
    \Gamma(\End(TM))\simeq\Omega^0(M)g\oplus U\oplus V \oplus \Omega^2_{\fsu(m)}(M)\oplus\Omega^2_{\fm_1}(M)\oplus \Omega^2_{\fm_2}(M).
\end{equation}
Thus, for any $A\in \Gamma(\End(TM))$, we have $A=A_0+S_U+S_V+C_{\fsu(m)}+C_{\fm_1}+C_{\fm_2}$, where
\begin{align}
\label{eq: su(n)_projections} \nonumber
        A_0&=\frac{1}{2m}(\tr A)g=\frac{1}{2m}\langle A,g\rangle g,\\ \nonumber
        S_U&=\frac{1}{4}(A+A^t)-\frac{1}{4}J(A+A^t)J-\frac{1}{2m}(\tr A)g,\\ 
        S_V&=\frac{1}{4}(A+A^t)+\frac{1}{4}J(A+A^t)J,\\ \nonumber
        C_{\fsu(m)}
        &=\frac{1}{4}(A-A^t)-\frac{1}{4}J(A-A^t)J+\frac{1}{2m}\tr(JA)J,\\ \nonumber
        C_{\fm_1}
        &=-\frac{1}{2m}\tr(JA)J=\frac{1}{2m}\langle A, J\rangle J,\\ \nonumber
        C_{\fm_2}
        &=\frac{1}{4}(A-A^t)+\frac{1}{4}J(A-A^t)J.
    \end{align}
 In the next lemma, we collect some properties of the action \eqref{eq: diamond_operator}, restricted to the action of $\fm=\fm_1\oplus\fm_2$ on the $\GL(m,\bC)$-irreducible components of $\Lambda^k(T_pM^\bC)^*$:

 \begin{lemma}\label{lm: diamond_m}
     Let $(g,J,\Upsilon)$ be an $\SU(m)$-structure. Then the operator $\diamond$ defined in \eqref{eq: diamond_operator} satisfies the following:
     \begin{itemize}
         \myitem[(i)]\label{item: m2_vectors} $\fm_2(T^{1,0})\subset T^{0,1}$ and $\fm_2( T^{0,1})\subset T^{1,0}$. 

         \myitem[(ii)] \label{item: m1_forms} $\fm_1\diamond\Lambda_J^{1,0}\subset \Lambda_J^{1,0}$ and $\fm_1\diamond\Lambda_J^{0,1}\subset \Lambda_J^{0,1}$. In particular, if $\alpha\in \Lambda_J^{1,0}$ then $J\diamond\alpha=i\alpha$ and 
         $J\diamond\bar{\alpha}=-i\bar{\alpha}$.
         
         \myitem[(iii)]\label{item: m2_forms} If $A\in \fm_2$ and $\alpha\in\Lambda_J^{1,0}$ then $A\diamond \alpha\in \Lambda_J^{0,1}$ and $A\diamond \bar{\alpha}\in \Lambda_J^{1,0}$. 

         \myitem[(iv)]\label{item: m2_pq_forms} For $A\in \fm_2$ and $\beta\in \Lambda_J^{p,q}$ we have $A\diamond\beta\in \Lambda_J^{p-1,q+1}\oplus\Lambda_J^{p+1,q-1}$.
     \end{itemize}
 \end{lemma}

 \begin{proof}\ 
     \begin{itemize} 
         \item[(i)] Let $A\in \fm_2$ and $Z\in T^{1,0}$, then
         $$
         J(A(Z))=-A(J(Z))=-A(iZ)=-iA(Z),
         $$
         and similarly $J(A(\bar{Z}))=iA(\bar{Z})$.

         \item[(ii)] Since $\Lambda^{1,0}_J$ and $\Lambda^{0,1}_J$ are $\GL(m,\bC)$-invariant and $\fm_1$ is generated by $J$, it follows that $\fm_1\diamond\Lambda_J^{1,0}\subset \Lambda_J^{1,0}$ and $\fm_1\diamond\Lambda_J^{0,1}\subset \Lambda_J^{0,1}$. Now, consider $\alpha\in \Lambda_J^{1,0}$ and $Z\in T^{1,0}$,  then
         $$
         (J\diamond\alpha)(Z)=\alpha(J(Z))=\alpha(iZ)=i\alpha(Z).
         $$
         Similarly, we obtain $J\diamond\bar{\alpha}=-i\bar{\alpha}$.

         \item[(iii)] Let $A\in \fm_2$ and $\alpha\in \Lambda_J^{1,0}$, by \ref{item: m2_vectors} we have that $A(Z)\in T^{0,1}$ for any $Z\in T^{1,0}$ then $(A\diamond\alpha)(Z)=0$. Hence $A\diamond\alpha\in \Lambda_J^{0,1}$.

         \item[(iv)] For $\beta\in \Lambda_J^{p,q}$, we can write it as $\beta=\alpha^{i_1}\wedge\cdots\wedge\alpha^{i_p}\wedge\bar{\alpha}^{j_1}\wedge\cdots\wedge\bar{\alpha}^{j_q}$ where each unbarred factor lies in $\Lambda_J^{1,0}$, hence for any $A\in\fm_2$ we have
         \begin{align*}
          A\diamond\beta=&(A\diamond\alpha^{i_1})\wedge\cdots\wedge\alpha^{i_p}\wedge\bar{\alpha}^{j_1}\wedge\cdots\wedge\bar{\alpha}^{j_q}+\dots+\alpha^{i_1}\wedge\cdots\wedge(A\diamond\alpha^{i_p})\wedge\bar{\alpha}^{j_1}\wedge\cdots\wedge\bar{\alpha}^{j_q}\\
          &+\alpha_{i_1}\wedge\cdots\wedge\alpha_{i_p}\wedge(A\diamond\bar{\alpha}_{j_1})\wedge\cdots\wedge\bar{\alpha}_{j_q}+...+\alpha_{i_1}\wedge\cdots\wedge\alpha_{i_p}\wedge\bar{\alpha}_{j_1}\wedge\cdots\wedge(A\diamond\bar{\alpha}_{j_q}).
         \end{align*}
         Thus, by \ref{item: m2_forms} the first $p$ terms have type $(p-1,q+1)$, whereas the last $q$ terms have type $(p+1,q-1)$. Therefore, $A\diamond\beta\in \Lambda_J^{p-1,q+1}\oplus\Lambda_J^{p+1,q-1} $.\qedhere
     \end{itemize}
 \end{proof}

Using the decomposition \eqref{eq: decomposition End(TM)}, we now describe the kernel of $\diamond$ acting on the $\SU(m)$-structure.

\begin{lemma}\label{lm: SUn_decomposition of End}
    Let $(g,J,\Upsilon)$ be an $\SU(m)$-structure on $M$. According to the decomposition \eqref{eq: decomposition End(TM)}, for any $A\in \Gamma(\End(TM))$ we have:
\begin{align}\label{eq: A_diamond_g}
    A\diamond g =&0 \quad \Leftrightarrow \quad A=C_{\fsu(m)}+C_{\fm_1}+C_{\fm_2}\\ \label{eq: A_diamond_J}
    A\diamond J =&0 \quad \Leftrightarrow \quad  A=A_0+S_U+C_{\fsu(m)}+C_{\fm_1}\\ \label{eq: A_diamond_Upsilon}
    A\diamond \Upsilon =&0 \quad  \Leftarrow \quad  A=S_U+C_{\fsu(m)}.
\end{align}
\end{lemma}

\begin{proof}
    For \eqref{eq: A_diamond_g} and \eqref{eq: A_diamond_J}, see \cite{Fadel2022}*{Lemma 1.4 (v) \& Example 1.9}. For \eqref{eq: A_diamond_Upsilon}, the matrix $S_U+C_{\fsu(m)}$ preserves the subspaces $T^{1,0}$ and $T^{0,1}$. It follows that
    \begin{gather*}
        (S_U+C_{\fsu(m)})\diamond \Upsilon=\tr(S_U)\Upsilon-i\tr(JC_{\fsu(m)})\Upsilon=0.
        \qedhere
    \end{gather*}
\end{proof}

\subsection{The torsion of a \texorpdfstring{$\SU(m)$}{}-structure}\label{sec: torsion SU(m)}

Given an $\SU(m)$-structure $(g,J,\Upsilon)$, we denote by $\nabla^{\SU(m)}$ the canonical $\SU(m)$-connection  with torsion $T\in \Omega^1(M,\Lambda^2_{\fm})$ \cite{Fadel2022}*{Lemma 1.17}. In other words, for any vector field $X\in\sX(M)$, one has  $\nabla_X^{\SU(m)}=\nabla_X+T_X$ and 
$$
  \nabla^{\SU(m)}_Xg =0, \quad \nabla^{\SU(m)}_XJ =0 \qandq \nabla^{\SU(m)}_X\Upsilon=0,
$$
where $\nabla$ is the Levi-Civita connection of $(M,g)$, as well as
    \begin{equation}\label{eq: nabla J_ nabla Upsilon}
        \nabla_XJ=T_X\diamond J \qandq \nabla_X\Upsilon=T_X\diamond\Upsilon.
    \end{equation}
By the metric compatibility of $\nabla$ and the fact that  $|\Upsilon|^2=2^m$ is constant, we have
\begin{equation}
\label{eq:  derivative_norm_Upsilon}
g(\nabla_X\Upsilon,\bar{\Upsilon})+g(\Upsilon,\overline{\nabla_X\Upsilon})=0.
\end{equation}
Since $\nabla_X\bar{\Upsilon}=\overline{\nabla_X\Upsilon}$, relation  \eqref{eq:  derivative_norm_Upsilon} implies $\overline{g(\nabla^h_X\Upsilon,\bar{\Upsilon} )}=-g(\nabla^h_X\Upsilon,\bar{\Upsilon})$. Hence   $ig(\nabla_X\Upsilon,\bar{\Upsilon})$ is a real-valued function on $M$ and we can define the real $1$-form $\eta$ by
\begin{align}\label{eq: eta_form}
    \eta(X)=ig(\nabla_X\Upsilon,\bar{\Upsilon}).
\end{align}

An expression for the intrinsic torsion $T$ of an $\SU(m)$-structure can be found in \cites{cabrera2005,martin2006}, but we rederive it to illustrate the formalism:

\begin{proposition}\label{prop: intrinsic torsion of su(n)}
    The intrinsic torsion of the $\SU(m)$-structure $(g,J,\Upsilon)$ is
    \begin{equation}\label{eq: torsion_of_SUn}
        T_X=-\frac{1}{m2^m}\eta(X)J+\frac{1}{2}(\nabla_XJ)J,
    \end{equation}
    with $\eta$ defined in \eqref{eq: eta_form}.
\end{proposition}

\begin{proof}
    For a  coordinate vector field $X=\partial_l$, the intrinsic torsion in \eqref{eq: nabla J_ nabla Upsilon} has components $T_{\partial_l}=:T_l=(T_l)_{\fm_1}+(T_l)_{\fm_2}$, where
    \begin{align*}
        \nabla_lJ=T_l\diamond J=-[T_l,J]=-2(T_l)_{\fm_2}J,
    \end{align*}
    and therefore 
    \begin{equation*}
        (T_l)_{\fm_2}=\frac{1}{2}(\nabla_lJ)J.
    \end{equation*}
    By Lemma \ref{lm: diamond_m} \ref{item: m1_forms} and \ref{item: m2_pq_forms}, we have $\nabla_l\Upsilon\in \Lambda_J^{m,0}\oplus \Lambda_J^{m-1,1}$, thus
    \begin{align*}
        g(\nabla_l\Upsilon,\bar{\Upsilon})=&g(T_l\diamond \Upsilon,\bar{\Upsilon})=-\frac{1}{2m}\tr(JT_l)g(J\diamond \Upsilon,\bar{\Upsilon})\\
        =&-\frac{i}{2}\tr(JT_l)g(\Upsilon,\bar{\Upsilon})=-i2^{m-1}\tr(JT_l),
    \end{align*}
    hence the remaining component is
    \begin{gather*}
        (T_l)_{\fm_1}
        =-\frac{i}{m2^{m}}g(\nabla_l\Upsilon,\bar{\Upsilon})J.
        \qedhere
    \end{gather*}
\end{proof}
    
The next lemma provides an alternative description of \eqref{eq: eta_form} in terms of real tensors.

\begin{lemma}
    Let $(g,J,\Upsilon)$  be an $\SU(m)$-structure on $M$, where $\Upsilon=\Upsilon^++i\Upsilon^-$ and  $\Upsilon^\pm\in [\![\Lambda^{m,0}_J]\!]$ satisfy \eqref{eq: relation Upsilon+_and_Upsilon-}. Then, the $1$-form $\eta$ defined by \eqref{eq: eta_form} can be written as
    \begin{align}\label{eq: eta_real and imag part of Upsilon}
        \eta(X)=&2g(\nabla_X\Upsilon^+, \Upsilon^-)=-2g(\nabla_X\Upsilon^-,\Upsilon^+).
    \end{align}
\end{lemma}

    \begin{proof}
    Since $\nabla_X\Upsilon^\pm\in [\![\Lambda_J^{m,0}]\!]\oplus [\![\Lambda_J^{m-1,1}]\!]$, for $m$ odd and using \eqref{eq: volume_m_odd}, we have
    \begin{align*}
        g(\nabla_X\Upsilon^+,\Upsilon^-)=*(\nabla_X\Upsilon^+\wedge *\Upsilon^-)=\frac{\eta(X)}{2^m}*(\Upsilon^+\wedge\Upsilon^-)=\frac{\eta(X)}{2}.
    \end{align*}
    For $m$ even and using \eqref{eq: volume_m_even}, we have 
    \begin{align*}
        g(\nabla_X\Upsilon^+,\Upsilon^-)=*(\nabla_X\Upsilon^+\wedge *\Upsilon^-)=\frac{\eta(X)}{2^m}*(\Upsilon^-\wedge\Upsilon^-)=\frac{\eta(X)}{2}.
    \end{align*}
Similarly, we obtain $\eta(X)=-2g(\nabla_X\Upsilon^-,\Upsilon^+)$.
\end{proof}

\begin{example}[Torsion of an $\SU(3)$-structure]
\label{ex: Torsion SU3}
    Consider the case $m=3$, using \eqref{eq: eta_real and imag part of Upsilon} to express the intrinsic torsion \eqref{eq: torsion_of_SUn}, we have 
\begin{equation}
\label{eq: torsion su3 endomorphism}
       T_l=\frac{1}{12}g(\nabla_l\Upsilon^-,\Upsilon^+)J+\frac12(\nabla_lJ)J.
\end{equation}
Since $2(\nabla_lJ)J=[\nabla_lJ,J]=J\diamond \nabla_lJ$, the full expression with explicit endomorphism indices is
\begin{equation}
\label{eq: torsion su3}
    T_{l,mn}=-\frac{1}{24}\eta(\partial_l)J_{mn}+\frac{1}{4}(J\diamond\nabla_lJ)_{mn}.
\end{equation}
Notice that \eqref{eq: torsion su3} agrees with the spinorial expression  found in \cite{agricola2015}*{Proposition 3.3}, since
\begin{align*}
    (J\diamond\nabla_XJ)(Y,Z) =2\nabla_XJ(JY,Z) =4\Upsilon^ -(\cT(X),JY,Z) =4\Upsilon^ +(\cT(X),Y,Z),
\end{align*}
where $\cT\in \Gamma(\End(TM))$ corresponds to the $\Omega^1(M,\Lambda^2_{\fm_2})$-component of the intrinsic torsion. In particular, the identity $\Upsilon^\pm_{jmn}\Upsilon^\pm_{kmn}=4\delta_{jk}$ (see, e.g., \cite{bedulli2007}*{Sec. 2.2}) allows us to write the endomorphism $\cT$ as
\begin{align}\label{eq: endomorphism torsion SU3}
    \cT_{jk}=\frac{1}{16}(J\diamond\nabla_jJ)_{mn}\Upsilon^+_{kmn}=\frac{1}{8}\nabla_jJ_{mn}\Upsilon_{kmn}^-.
\end{align}
Thus
\begin{equation}
\label{eq: torsion_eta+S}
   T_{l,mn} =-\frac{1}{24}\eta(\partial_l)J_{mn}+\cT_l^p\Upsilon^+_{pmn}
\end{equation}
and we have the alternative expressions
\begin{align}
\label{eq: nabla_J_SU3}
    \nabla_aJ_{ij} &=2\cT_a^b\Upsilon^-_{bij}=-2(J\cT)_a^b\Upsilon^+_{bij},\\ \label{eq: nabla_Upsilon+SU3}
    \nabla_a\Upsilon^+_{ijk}&=\frac{1}{8}\eta_a\Upsilon^-_{ijk}+2((J\cT)_{ai}J_{jk}+(J\cT)_{aj}J_{ki}+(J\cT)_{ak}J_{ij})\\ \nonumber
    &=\frac{1}{8}\eta_a\Upsilon^-_{ijk}+2((\partial_a\lrcorner J\cT)\wedge J)_{ijk}\\
    \label{eq: nabla_Upsilon-SU3}
    \nabla_a\Upsilon^-_{ijk}&=-\frac{1}{8}\eta_a\Upsilon^+_{ijk}-2(\cT_{ai}J_{jk}+\cT_{aj}J_{ki}+\cT_{ak}J_{ij})\\ \nonumber
    &=-\frac{1}{8}\eta_a\Upsilon^+_{ijk}-2((\partial_a\lrcorner\cT)\wedge J)_{ijk}.
\end{align}
\end{example}

Using Proposition \ref{prop: intrinsic torsion of su(n)}, we can easily compute the divergence of the intrinsic torsion.

\begin{proposition}
    The divergence of the $\SU(m)$ intrinsic torsion \eqref{eq: torsion_of_SUn} is
    \begin{equation}
    \label{eq: divergence_of_T_SUn}
    \Div T=-\frac{1}{m2^m}\left(\Div\eta\right) J 
        -\frac{1}{m2^m}\nabla_{\eta^\sharp}J+\frac{1}{4}[\Delta J,J],
    \end{equation}
    where $\eta^\sharp=\eta(\partial_a)g^{ab}\partial_b$ and $\Delta \beta:=\tr_{g}(\nabla^2\beta)=(\nabla_a\nabla_b\beta-\nabla_{\nabla_a\partial_b}\beta)g^{ab}$, for any tensor field $\beta$. 
    
    Moreover, the $\SU(m)$-structure $(g,J,\Upsilon)$ is \emph{harmonic} (i.e. $\Div T=0$) if and only if,
\begin{align}
    \label{eq: SUn-harmonic condition}
        \Div \eta=2g(\Delta \Upsilon^+,\Upsilon^-)=0 \qandq [\Delta J,J]=\frac{1}{m2^{m-2}}\nabla_{\eta^\sharp}J.
    \end{align}
\end{proposition}

\begin{proof}
Taking the divergence directly on \eqref{eq: torsion_of_SUn}, we have
\begin{align*}
    \Div T=&\left(\nabla_a(T_b)-T(\nabla_a\partial_b)\right)g^{ab}\\
    =&\left(-\frac{1}{m2^{m}}\left(\partial_a(\eta(\partial_b))J+\eta(\partial_b)\nabla_aJ\right)+\frac{1}{2}(\nabla_a(\nabla_bJ) J+\nabla_bJ\nabla_aJ)\right.\\
    &\left.+\frac{1}{m2^m}\eta(\nabla_a\partial_b)J-\frac{1}{2}\nabla_{\nabla_a\partial_b}J \right)g^{ab}\\
    =&-\frac{1}{m2^{m}}\left(\Div\eta J+\nabla_{\eta^\sharp}J\right)+\frac{1}{2}\left(\nabla_a(\nabla_bJ) J+\nabla_bJ\nabla_aJ-\nabla_{\nabla_a\partial_b}J \right)g^{ab}.
\end{align*}
      Notice that by applying  $\nabla_b\nabla_a$ to $J^2=-I$, we get the identity $$-2\nabla_bJ\nabla_aJg^{ab} =(\nabla_a\nabla_bJ\circ J+J\circ\nabla_a\nabla_bJ)g^{ab},
      $$ 
      and so    \begin{align*}
        \frac{1}{2}\left(\nabla_a(\nabla_bJ) J+\nabla_bJ\nabla_aJ-\nabla_{\nabla_a\partial_b}J \right)g^{ab}=&\frac{1}{2}\left(\nabla_a\nabla_bJ\circ J+\nabla_bJ\nabla_aJ\right)g^{ab}\\
        =&\frac{1}{4}(\nabla_a\nabla_bJ\circ J-J\circ \nabla_a\nabla_bJ)g^{ab}=\frac{1}{4}[\Delta J,J].
    \end{align*}
    Finally, using \eqref{eq: eta_real and imag part of Upsilon}, the harmonicity condition $\Div T=0$ clearly decomposes as \eqref{eq: SUn-harmonic condition}.
\end{proof}

\begin{remark}
    If $\eta=0$, then harmonic $\SU(m)$-structures are the harmonic $\U(m)$-structures. In particular, the examples of harmonic $\SU(3)$-structures given in \cite{niedzialomski2020}*{Theorem 1} are in fact harmonic $\U(3)$-structures with vanishing $\eta$. 
\end{remark}

\begin{example}[Harmonic $\SU(2)$-structure]\label{ex: almost_Abelian SU2_harmonic}
    Following \cite{andrada2023}*{Example 5.4}, consider the almost Abelian Lie algebra $\fg=\bR\ltimes_L \bR^3$, with Lie bracket induced by $L$ and almost complex structure $J$ as follows:
    \begin{equation*}
        L=\left(\begin{array}{c|cc}
           0 & 0 & 0 \\ \hline
            1 & 0 & -1 \\
            0 & 1 & 0
        \end{array}\right) \qandq J=\left(\begin{array}{cc}
            0 & -1  \\
            1 & 0 
        \end{array}\right)^{\oplus 2}.    \end{equation*}
        If $g$ is the canonical inner product on $\bR^4$, the $\U(2)$-structure $(g,J)$ on $\fg$ is not harmonic \cite{andrada2023}*{Corollary 3.4}, since
        $$
          [\Delta J,J]=\left(\begin{array}{cc|cc}
               & & -1 & 0\\
               & & 0 & 1 \\ \hline
               1 & 0 & & \\
                0 & -1 & &
          \end{array}\right).
        $$
        Now, consider $\Upsilon=\Upsilon^++i\Upsilon^-=e^{13}-e^{24}+i(e^{14}+e^{23})$, where $\{e^1,e^2,e^3,e^4\}\subset (\bR^4)^*$ is the canonical dual basis. We claim that $(g,J,\Upsilon)$ is a harmonic $\SU(2)$-structure on $\fg$. Indeed, the Levi-Civita connection on $(\fg, g)$ is
    \begin{equation}
        \label{eq: LC-connection_example}
        \nabla_{e_1}e_1=0, \quad \nabla_{e_1}e_j=\frac{1}{2}(L-L^t)e_j \quad \nabla_{e_j}e_1=-\frac{1}{2}(L+L^t)e_j, \qandq \nabla_{e_i}e_j=\frac{1}{2}g((L+L^t)e_i,e_j)e_1,
        \end{equation}
        for $e_i,e_j\in \bR^3=\Spa(e_2,e_3,e_4)$. Thus, from \eqref{eq: eta_real and imag part of Upsilon} we have $\eta=4e^1-2e^3$, and  \eqref{eq: LC-connection_example} yields
    \begin{equation*}
        \Div\eta =0 
        \qandq \nabla_{\eta^\sharp}J 
        =2[\Delta J,J],
    \end{equation*}
        which is precisely the harmonic condition \eqref{eq: SUn-harmonic condition} for $m=2$.
\end{example}

\begin{example}[Divergence of an $\SU(3)$-structure] 
Using \eqref{eq: torsion su3 endomorphism},  the divergence of the torsion \eqref{eq: divergence_of_T_SUn} is
\begin{align}
\label{eq: div torsion su3}
    (\Div T) =-\frac{1}{24}(\Div\eta)J-\frac{1}{24}\nabla_{\eta^\sharp}J +\frac{1}{4}[\Delta J, J].
\end{align}
    Alternatively, we can compute $\Div T$ from \eqref{eq: torsion_eta+S}:
    \begin{align*}
        \Div T_{ab}=-\frac{1}{24}(\Div\eta)J_{ab}-\frac{1}{24}\nabla_{\eta^\sharp}J_{ab}+(\Div \cT\lrcorner\Upsilon^+)_{ab}+(\nabla_p\Upsilon^+)(\cT(\partial_q),\partial_a,\partial_b)g^{pq}.
    \end{align*}
    Using the definition $\nabla_p\Upsilon^+=T_p\diamond\Upsilon^+$ and \eqref{eq: torsion_eta+S}, we have
    \begin{align*}
        (\nabla_p\Upsilon^+)(\cT(\partial_q),\partial_a,\partial_b)g^{pq}=&\left(\frac18\eta(\partial_p)\Upsilon^-(\cT(\partial_q),\partial_a,\partial_b)+2\left(\Upsilon(\partial_b,\cT(\partial_q),\partial_s)\Upsilon^+(\cT(\partial_p),\partial_a,\partial_r)\right.\right.\\
        &\left.\left.+\Upsilon(\partial_a,\partial_b,\partial_s)\Upsilon^+(\cT(\partial_p),\cT(\partial_q),\partial_r)+\Upsilon(\cT(\partial_q),\partial_a,\partial_s)\Upsilon^+(\cT(\partial_p),\partial_b,\partial_r)\right)g^{rs}\right)g^{pq}\\
        =&\frac{1}{8}\Upsilon^-(\cT(\eta^\sharp),\partial_a,\partial_b),
    \end{align*}
    hence
    \begin{equation*}
        \Div T= -\frac{1}{24}(\Div \eta)J-\frac{1}{24}\nabla_{\eta^\sharp}J+\frac18 \cT(\eta^\sharp)\lrcorner \Upsilon^-+\Div \cT\lrcorner \Upsilon^+.
    \end{equation*}
    Thus, the harmonicity condition is equivalent to
    $$
    \Div\eta=0 \qandq \Div\cT=\frac{1}{24}J\cT(\eta^\sharp).
    $$
In particular, if  $\eta=0$ then $\Div T=0$ is equivalent to $\Div \cT=0$  \cite{niedzialomski2020}*{Theorem 1}.
\end{example}

\subsection{Bianchi-type identity and curvature}\label{sec: bianchi identity and curvature}

As an application of \eqref{eq: torsion_of_SUn}, we can explicitly write  the \emph{Bianchi-type identity} from  \cite{Fadel2022}*{Corollary 1.38},
\begin{equation}
\label{eq: Bianchi-type-identity}
    \nabla_aT_b-\nabla_bT_a =-2[T_a,T_b]+[T_a,T_b]_{\fm}+(R_{ba})_{\fm},
\end{equation}
in the $\SU(m)$ case, using the projections \eqref{eq: su(n)_projections}.

\begin{corollary}
    The torsion $T_l$ of an $\SU(m)$-structure $(g,J,\Upsilon)$ satisfies the Bianchi-type identity
\begin{equation}
\label{eq: su(n) Bianchi ident}
    \frac{1}{m2^{m-1}}d\eta_{ab}J-\left(\nabla_a\nabla_bJ-\nabla_b\nabla_aJ\right)J =\frac{1}{2m}\left(\tr(J\nabla_aJ\nabla_bJ) +2\tr(JR_{ba})\right)J-\left(R_{ba}+JR_{ba}J\right).
\end{equation}
    Equivalently,   
\begin{gather}
\label{eq: m1 Bianchi identity}
    d\eta_{ab} 
    =2^{m-2}\tr(J\nabla_aJ\nabla_bJ) 
    +2^{m-1}\tr(JR_{ba})\\
\label{eq: m2 Bianchi identity}
    (\nabla_a\nabla_bJ -\nabla_b\nabla_aJ)_{\fm_2} =R_{ab}J-JR_{ab}.
    \end{gather}
    By \eqref{eq: m1 Bianchi identity}, $d\eta$ represents the first Chern class.
\end{corollary}

\begin{proof}
    From \eqref{eq: torsion_of_SUn}, we have
    \begin{align*}
        [T_a,T_b]=&-\frac{1}{m2^{m+1}}\left(\eta(\partial_a)[J,(\nabla_bJ)J]+\eta(\partial_b)[(\nabla_aJ)J,J]\right)+\frac14[(\nabla_aJ)J,(\nabla_bJ)J]\\
        =&-\frac{1}{m2^{m}}\left(\eta(\partial_a)\nabla_bJ-\eta(\partial_b)\nabla_aJ\right)+\frac14\left(\nabla_aJ\nabla_bJ-\nabla_bJ\nabla_aJ\right).
    \end{align*}
    Using the projections \eqref{eq: su(n)_projections}, we get
    \begin{align*}
        [T_a,T_b]_{\fm_1} 
        =-\frac{1}{2m}\tr(J[T_a,T_b])J 
        =-\frac{1}{4m}\tr(J\nabla_aJ\nabla_bJ)J
    \end{align*}
     and 
    \begin{align*}
        [T_a,T_b]_{\fm_2} =\frac{1}{2}([T_a,T_b]+J[T_a,T_b]J)
        =-\frac{1}{m2^{m}}\left(\eta(\partial_a)\nabla_bJ-\eta(\partial_b)\nabla_aJ\right),
    \end{align*}
    since  $J\nabla_aJ\in \Lambda^2_{\fm_2}$ and $[\nabla_aJ,\nabla_bJ]\in \Lambda^2_{\fsu(m)}\oplus\Lambda^2_{\fm_1}$. Therefore 
    \begin{align*}
        -2[T_a,T_b]+[T_a,T_b]_{\fm}+(R_{ba})_{\fm}
        =&\;\frac{1}{m2^{m}}\left(\eta(\partial_a)\nabla_bJ-\eta(\partial_b)\nabla_aJ\right)-\frac12[\nabla_aJ,\nabla_bJ]\\
        &-\frac{1}{4m}\tr(J\nabla_aJ\nabla_bJ)J-\frac{1}{2m}\tr(JR_{ba})J+\frac12\left(R_{ba}+JR_{ba}J\right).
    \end{align*}
    The left-hand side of \eqref{eq: Bianchi-type-identity} expands as 
    \begin{align*}
        \nabla_aT_b-\nabla_bT_a
        =&-\frac{1}{m2^{m}}\left((\nabla_a\eta(\partial_b)-\nabla_b\eta(\partial_a))J+\eta(\partial_b)\nabla_aJ-\eta(\partial_a)\nabla_bJ\right)\\
        &+\frac12\left(\nabla_a\nabla_bJ-\nabla_b\nabla_aJ\right)J+\frac12[\nabla_bJ,\nabla_aJ],
    \end{align*}
    so that \eqref{eq: Bianchi-type-identity} becomes
    \begin{multline*}
        -\frac{1}{m2^{m}}\left(\nabla_a\eta(\partial_b)-\nabla_b\eta(\partial_a)\right)J+\frac{1}{2}\left(\nabla_a\nabla_bJ-\nabla_b\nabla_aJ\right)J=-\frac{1}{4m}\tr(J\nabla_aJ\nabla_bJ)J\\-\frac{1}{2m}\tr(JR_{ba})J+\frac{1}{2}\left(R_{ba}+JR_{ba}J\right).
    \end{multline*}
    Thus, the $\Omega_{\fm_1}^2$-part of \eqref{eq: su(n) Bianchi ident} is 
    \begin{align*}
        \left(\nabla_a\eta(\partial_b)-\nabla_b\eta(\partial_a)\right)J=2^{m-2}\tr(J\nabla_aJ\nabla_bJ)J+2^{m-1}\tr(JR_{ba})J.
    \end{align*}
    Since $((\nabla_a\nabla_bJ-\nabla_b\nabla_aJ)J)_{\fm_2} =\frac12[\nabla_a\nabla_bJ-\nabla_b\nabla_aJ,J]$,
    the remaining  $\Omega_{\fm_2}^2$-part of \eqref{eq: su(n) Bianchi ident} is
\begin{gather*}
    (\nabla_a\nabla_bJ -\nabla_b\nabla_aJ)_{\fm_2} =R_{ab}J-JR_{ab}.
    \qedhere
\end{gather*}
\end{proof}

We now consider the corresponding Riemannian geometry. For fixed biplane indices $a$ and $b$, the Riemann curvature endomorphism $R_{ab}$ is skew-symmetric, and its projections \eqref{eq: su(n)_projections} are
\begin{align}\label{eq:  curvature_projections}
  (R_{ab})_{\fsu(m)}
        &=\frac{1}{2}R_{ab}-\frac{1}{2}JR_{ab}J+\frac{1}{2m}\tr(JR_{ab})J,\\ \nonumber
        (R_{ab})_{\fm}
        &=(R_{ab})_{\fm_1}+(R_{ab})_{\fm_2}\\ \nonumber
        &=-\frac{1}{2m}\tr(JR_{ab})J+\frac{1}{2}R_{ab}+\frac{1}{2}JR_{ab}J.
\end{align}
\begin{definition}[{\cite{tachibana1959}}]
   The $*$-Ricci operator of an almost Hermitian metric is given by
\begin{equation}\label{eq:Ricci*}
\Ric^*_{ab}:=\frac{1}{2}R_{ijka}J^{ij}J^k_b =R_{iajk}J^{ik}J^j_b,
\end{equation}
or, more invariantly, by $\Ric^*(X,Y)=\tr(Z\mapsto JR(X,Z)JY)$.
\end{definition}
We can express the Ricci and $*$-Ricci curvatures in terms of the intrinsic torsion as follows. 
\begin{proposition}[Ricci and scalar curvatures]
\label{prop: Ricci tensor}
    Given an $\SU(m)$-structure $(J,\Upsilon)$ with almost Hermitian metric $g$, its Ricci tensor is given by
\begin{align}
\label{eq: su(n) Ricci}
    \Ric_{bc} =&\frac12[J,\nabla_a\nabla_bJ-\nabla_b\nabla_aJ]_c^a+\frac{1}{2^m}d\eta_{ab}J_c^a-\frac14\tr(J\nabla_aJ\nabla_bJ)J_c^a\\ 
        \label{eq: su(n) * Ricci}
        \Ric^*_{bc} =&\frac{1}{2^{m}}d\eta_{ab}J_c^a-\frac14\tr(J\nabla_aJ\nabla_bJ)J_c^a
    \end{align}
   In particular, the scalar curvature $s$ and the $*$-scalar curvature $s^*$ are
    \begin{align}\label{eq: su(n) scalar}   
    s=&|\Div J|^2-2\Div(J\Div J)+(\nabla_aJ\nabla_bJ)_{ab}+\frac14\langle J\nabla_{J}J,\nabla J\rangle-\frac{1}{2^m}\langle d\eta,J\rangle,\\ \label{eq: su(n) * scalar}   
    s^*=&\frac14\langle \nabla J,J\nabla_J J\rangle-\frac{1}{2^m}\langle d\eta,J\rangle.
    \end{align}
\end{proposition}

\begin{proof}
To derive the Ricci tensor, we write
    \begin{align*}
        \Ric_{bc}=R_{abcd}g^{ad}={(R_{ab})_{\fsu(m)}}_{c}^{a}+{(R_{ab})_{\fm}}_{c}^{a}.
    \end{align*}
    Applying the identity \cite{Fadel2022}*{(1.65)}
    \begin{equation*}
        (JR_{ab}J)_c^a=\frac12\tr(JR_{ab})J_c^a=-n{(R_{ab})_{\fm_1}}_c^a,
    \end{equation*}
    to \eqref{eq:  curvature_projections}, we obtain
    \begin{align*}
        {(R_{ab})_{\fsu(m)}}_c^a={(R_{ab})_{\fm_2}}_c^a+(m-1){(R_{ab})_{\fm_1}}_c^a.
    \end{align*}
   Using the Bianchi-type identities \eqref{eq: m1 Bianchi identity} and \eqref{eq: m2 Bianchi identity}, we have 
    \begin{align*}
        \Ric_{bc} &= 2{(R_{ab})_{\fm_2}}_{c}^{a} +n{(R_{ab})_{\fm_1}}_{c}^{a}\\
        &=\frac12[J,\nabla_a\nabla_bJ-\nabla_b\nabla_aJ]_c^a +\frac{1}{2^m}d\eta_{ab}J_c^a-\frac14\tr(J\nabla_aJ\nabla_bJ)J_c^a.
    \end{align*}
    Similarly, for the $*$-Ricci curvature, we have
    \begin{align*}
        \Ric^*_{jk}=&R_{ijlp}J_{ip}J_{kl}
        =-(JR_{ij}J)_{ki}=n{(R_{ij})_{\fm_1}}_{ki},
    \end{align*}
    and \eqref{eq: su(n) * Ricci} follows from \eqref{eq: m1 Bianchi identity}. Finally, we compute $s=\tr(\Ric)$ and $s^*=\tr(\Ric^*)$:
\begin{align*}
    s =&\;\frac12[J,\nabla_a\nabla_bJ-\nabla_b\nabla_aJ]_{ba} -\frac14\left(\tr(J\nabla_aJ\nabla_bJ) -\frac{1}{2^{m-2}}d\eta_{ab}\right)J_{ba}\\
    =&-(\nabla_aJ\nabla_bJ)_{ba}+2\nabla_a(J\nabla_bJ)_{ba}+(\nabla_bJ\nabla_aJ)_{ba} \\
    &-\frac14\tr(J\nabla_{J\partial_b}J\nabla_bJ)-\frac{1}{2^m}\langle d\eta,J\rangle\\
    =&\;|\Div J|^2-2\Div(J\Div J)+(\nabla_aJ\nabla_bJ)_{ab}+\frac14\langle J\nabla_{J}J,\nabla J\rangle-\frac{1}{2^m}\langle d\eta,J\rangle.
    \qedhere
\end{align*}
\end{proof}

In the $\SU(3)$ case, the reader might find it interesting to compare Proposition \ref{prop: Ricci tensor} with the description in terms of torsion forms in Proposition \ref{prop: ricci torsion forms 0}, \ref{prop: ricci torsion forms 1} and \ref{prop: ricci torsion forms 2}. 

\begin{remark}\quad
\begin{itemize}
    \item     In \cite{martin2006}*{Lemma 3.3}, the author finds the formula for the Ricci curvature in terms of the tensors
    \begin{equation*}
        \hat{\eta}=\frac{1}{m2^m}\eta \qandq \xi=-\frac{1}{2}J\nabla J.
    \end{equation*}
    Thus, for any vector fields $X,Y\in \cX(M)$, the Ricci tensor \eqref{eq: su(n) Ricci} is
    \begin{align*}
        \Ric(X,Y)=&g((\nabla_X\nabla_aJ-\nabla_a\nabla_XJ)JY,\partial_b)g^{ab}-\frac14\tr(J\nabla_XJ\nabla_{JY}J)-\frac{1}{2^m}(\nabla_X\eta(JY)-\nabla_{JY}\eta(X))\\
        =&2g((\nabla_X(J\xi_a)-\nabla_a(J\xi_X)))JY,\partial_b)g^{ab}+\tr(\xi_XJ\xi_{JY})-m\,d\hat{\eta}(X,JY).
    \end{align*}
    Using the $\U(m)$ connection  $\Tilde{\nabla}=\nabla+\xi$ (i.e. $\Tilde{\nabla}J=0$), we obtain
    \begin{align*}
        \Ric(X,Y)
        &= 2g((\nabla_X(\xi_a)-\nabla_a(\xi_X)-2[\xi_a,\xi_X]))Y,\partial_b)g^{ab} -g(J\xi_{JY},\xi_X)-m\,d\hat{\eta}(X,JY)\\
        &= 2g((\Tilde{\nabla}_X\xi)_a-(\Tilde{\nabla}_a\xi)_X+\xi_{\xi_X\partial_a}-\xi_{\xi_{\partial_a}X}))Y,\partial_b)g^{ab} -g(J\xi_{JY},\xi_X) -m\,d\hat{\eta}(X,JY).
    \end{align*}
    \item \cite{martin2006}*{Theorem 3.4} If $(M,g,J)$ is Kähler, from \eqref{eq: su(n) Ricci} and \eqref{eq: su(n) * Ricci} we have  $\Ric=\Ric^*$. Moreover, if $d\eta=\lambda\omega$ for some $\lambda\in \bR$, then $M$ is Kähler-Einstein, and if $\eta$ is closed then  $M$ is Ricci-flat.
\end{itemize}
\end{remark}

\subsection{Torsion classes of \texorpdfstring{$\SU(m)$}{}-structures}

Let $(M,g,J)$ be an almost Hermitian manifold. 
The vector bundle $\tilde{W}=T^*M\otimes \Lambda^2_{\fu(m)^\perp}$ is identified with
\begin{equation*}
    \tilde{W}=\{\alpha:=(\alpha_{i,jk})\in T^*M\otimes\Lambda^2T^*M : \quad \alpha_{i,jk}=-\alpha_{i,kj}=-\alpha_{i,ab}J^a_jJ^b_k \},
\end{equation*}
and it admits the following decomposition into irreducible $\U(m)$-modules $\Tilde{W}=W_1\oplus W_2\oplus W_3\oplus W_4$ \cite{gray1980}*{Theorem 2.1}, where 
\begin{align}\nonumber
    W_1:=&\{\alpha\in \Tilde{W}: \quad \alpha_{i,jk}=-\alpha_{j,ik}\}, \quad W_2:=\{\alpha\in \Tilde{W}: \quad \alpha_{i,jk}+\alpha_{k,ij}+\alpha_{j,ki}=0 \}\\ \label{eq: Gray Hervella decomposition}
    W_3:=&\{\alpha\in \Tilde{W}: \quad \alpha_{i,jk}-J_i^aJ_j^b\alpha_{a,bk}=0 \qandq\bar{\alpha}_k=\alpha_{i,jk}g^{ij}=0 \}\\ \nonumber
    W_4:=&\{\alpha\in \Tilde{W}: \quad \alpha_{i,jk}=\frac{1}{2(m-1)}\left(g_{ik}g^{ab}\alpha_{a,bj}-g_{ij}g^{ab}\alpha_{a,bk}-g_{ia}J_j^bJ_k^ag^{pq}\alpha_{p,qb}+g_{ia}J_j^aJ_k^bg^{pq}\alpha_{p,qb}\right)\}.
\end{align}
Notice that the spaces $W_1\oplus W_2$ and $W_3\oplus W_4 $ are characterised, respectively, by the relations:
\begin{equation}
\label{eq: eigenspace_condition W}
     \alpha_{i,jk} +J_i^aJ_j^b\alpha_{a,bk}=0 \qandq \alpha_{i,jk}-J_i^aJ_j^b\alpha_{a,bk}=0.
\end{equation}
Moreover, any element of $W_1\oplus W_2\oplus W_3$ satisfies $\alpha_{i,jk}g^{ij}=0$ for every $k$, and 
\begin{align*}
    \dim W_1=&\frac{n}{3}(m-1)(m-2), \quad \dim W_2=\frac{n}{3}(m-1)(m+1),\\ \nonumber \dim W_3=&m(m+1)(m-2) \qandq \dim W_4=n.
\end{align*}

\begin{lemma}
[\cite{gray1980}*{Theorem 3.1}]
\label{lm: U(m) projections}
    For an almost Hermitian manifold $(M^{2m},g,J)$, the projections $\pi_l:\Tilde{W}\rightarrow W_l$ for $l=1,2,3,4$, are given by
    \begin{align*}
        \pi_1(\alpha)_{i,jk}=&\frac{1}{6}\left(\alpha_{i,jk}+\alpha_{j,ki}+\alpha_{k,ij}-J_i^aJ_j^b\alpha_{a,bk}-J_j^aJ_k^b\alpha_{a,bi}-J_k^aJ_i^b\alpha_{a,bj}\right)\\
        \pi_2(\alpha)_{i,jk}=&\frac{1}{6}\left(2\alpha_{i,jk}-\alpha_{j,ki}-\alpha_{k,ij}-2J_i^aJ_j^b\alpha_{a,bk}+J_j^aJ_k^b\alpha_{a,bi}+J_k^aJ_i^b\alpha_{a,bj}\right)\\
        \pi_3(\alpha)_{i,jk}=&\frac{1}{2}\left(\alpha_{i,jk}+J_i^aJ_j^b\alpha_{a,bk}\right)-\frac{1}{2(m-1)}\left(g_{ij}g^{ab}\alpha_{a,bk}-g_{ik}g^{ab}\alpha_{a,bj}+g_{ia}J_j^bJ_k^ag^{pq}\alpha_{p,qb}-g_{ia}J_j^aJ_k^bg^{pq}\alpha_{p,qb}\right)\\
        \pi_4(\alpha)_{i,jk}=&\frac{1}{2(m-1)}\left(g_{ij}g^{ab}\alpha_{a,bk}-g_{ik}g^{ab}\alpha_{a,bj}+g_{ia}J_j^bJ_k^ag^{pq}\alpha_{p,qb}-g_{ia}J_j^aJ_k^bg^{pq}\alpha_{p,qb}\right)
    \end{align*}
\end{lemma}

For an $\SU(m)$-structure $(g,J,\Upsilon)$ on $M$, we have
\begin{align}\label{eq: torsion_classes_SUn}
    T^*M\otimes\Lambda^2_{\fm}\simeq W= W_1\oplus W_2\oplus W_3\oplus W_4\oplus W_5,
\end{align}
where $\Tilde{W}=W_1\oplus W_2\oplus W_3\oplus W_4$ are the $\U(m)$-submodules \eqref{eq: Gray Hervella decomposition} and $W_5=T^*M\otimes \Lambda^2_{\fm_1}\simeq T^*M\otimes\langle J\rangle$. Thus, for $\alpha\in W$ we have:
\begin{equation*}
    \alpha_{i,jk}=\nu_iJ_{jk}+\beta_{i,jk} \qforq \nu \in \Omega^1 \qandq \beta\in \Tilde{W}.
\end{equation*}
Using \eqref{eq: su(n)_projections}, the last projection is given by
\begin{equation*}
    \pi_5(\alpha)_{i,jk} =\frac{1}{2m}\alpha_{i,ab}g^{ac}g^{bd}J_{cd}J_{jk}.
\end{equation*}

\begin{lemma}
[\cite{martin2006}]
    Let $(g,J,\Upsilon)$ be an $\SU(m)$-structure on $M^{2m}$:
    \begin{itemize}
        \item For $m=2$, the decomposition \eqref{eq: torsion_classes_SUn} simplifies to $$W=W_2\oplus W_4\oplus W_5\simeq \left(T^*M\otimes \Upsilon^+\right)\oplus\left(T^*M\otimes \Upsilon^-\right)\oplus\left(T^*M\otimes J\right).$$
        \item For $m=3$, the decomposition \eqref{eq: torsion_classes_SUn} has $W_3\oplus W_4\oplus W_5$ as $\SU(3)$-irreducible summands and $W_l=W_l^+\oplus W_l^-$ for $l=1,2$, where  
        \begin{align*}
             W_1^\pm=\Omega^0\Upsilon^\pm \qandq 
            W_2^\pm=\{\alpha\in W_2: \quad \alpha_{a,jk}\Upsilon^+_{bjk}=\pm\alpha_{b,jk}\Upsilon^+_{ajk}
            \},
        \end{align*}
        are $\SU(3)$-irreducible.
        \item For $m\geq 4$, the decomposition \eqref{eq: torsion_classes_SUn} is $\SU(m)$-irreducible.
    \end{itemize}
\end{lemma}

In particular, the intrinsic torsion \eqref{eq: torsion_of_SUn} decomposes according to  \eqref{eq: torsion_classes_SUn},
\begin{align*}
    T=T_1+T_2+T_3+T_4+T_5,
    \qwithq
    T_l\in W_l, 
    \qforq
    l=1,\dots,5.
\end{align*} Explicitly, 
\begin{equation*}
    T_l=\frac{1}{2}\pi_l(\nabla J\circ J), \qforq l=1,2,3,4, \qandq 
    T_5=-\frac{1}{m2^m}\eta\otimes J,
\end{equation*}
for $\pi_1,\dots,\pi_4$  given in Lemma \ref{lm: U(m) projections}. Hence
\begin{align*}\nonumber
    (T_1)_{i,jk}=&\;-\frac{1}{12}\left((J\nabla_iJ)_{jk}+(J\nabla_jJ)_{ki}+(J\nabla_kJ)_{ij}-(\nabla_{J\partial_i}J)_{jk}-(\nabla_{J\partial_j}J)_{ki}-(\nabla_{J\partial_k}J)_{ij}\right)\\ \nonumber
    (T_2)_{i,jk}=&\;-\frac{1}{12}\left(2(J\nabla_iJ)_{jk}-(J\nabla_jJ)_{ki}-(J\nabla_kJ)_{ij}-2(\nabla_{J\partial_i}J)_{jk}+(\nabla_{J\partial_j}J)_{ki}+(\nabla_{J\partial_k}J)_{ij}\right)\\ 
    (T_3)_{i,jk}=&\;-\frac{1}{4}\left((J\nabla_iJ)_{jk}+(\nabla_{J\partial_i}J)_{jk}\right)\\ \nonumber
    &-\frac{1}{4(m-1)}\left(g_{ij}(J\Div J)_k-g_{ik}(J\Div J)_j+J_{ij}(\Div J)_k-J_{ik}(\Div J)_j\right)\\ \nonumber
    (T_4)_{i,jk}=&\;\frac{1}{4(m-1)}\left(g_{ij}(J\Div J)_k-g_{ik}(J\Div J)_j+J_{ij}(\Div J)_k-J_{ik}(\Div J)_j\right)\\ \nonumber
    (T_5)_{i,jk}=&\;-\frac{1}{m2^m}\eta_iJ_{jk}.
\end{align*}
Furthermore, the norms of these torsion components are given by
\begin{align} \nonumber
    |T_1|^2=&\frac{1}{24}\left(|\nabla J|^2+\langle \nabla J,J\nabla_J J\rangle+2(\nabla_iJ\nabla_jJ)^{ij}\right)\\ \nonumber
    |T_2|^2=&\frac{1}{12}\left(|\nabla J|^2+\langle\nabla J,J\nabla_JJ\rangle-(\nabla_iJ\nabla_jJ)^{ij}\right)\\ \label{eq: norm torsion components}
    |T_3|^2=&\frac18\left(|\nabla J|^2-\langle \nabla J,J\nabla_J J\rangle\right)-\frac{1}{2(m-1)}|\Div J|^2\\ \nonumber
    |T_4|^2=&\frac{1}{2(m-1)}|\Div J|^2\\ \nonumber
    |T_5|^2=&\frac{1}{m2^{2m-1}}|\eta|^2.
\end{align}
Alternatively, we have:
\begin{align}\nonumber
    |\nabla J|^2  &= 4(|T_1|^2+|T_2|^2+|T_3|^2+|T_4|^2)\\ \nonumber
    \langle \nabla J, J\nabla_J J\rangle &= 4(|T_1|^2+|T_2|^2-|T_3|^2-|T_4|^2)\\ \label{eq: quadratic nabla J}
    (\nabla_iJ\nabla_jJ)^{ij} &= 4(2|T_1|^2-|T_2|^2)\\ \nonumber
    |\Div J|^2 &= 2(m-1)|T_4|^2\\ \nonumber
    |\eta|^2 &= m2^{2m-1}|T_5|^2
\end{align}

In particular, the scalar curvature \eqref{eq: su(n) scalar} and the $*$-scalar curvature \eqref{eq: su(n) * scalar} can be rewritten in terms of the quadratic torsion components \eqref{eq: quadratic nabla J}. These formulae were obtained previously in \cite{bor2001}:
\begin{proposition}
    Let $(g,J,\Upsilon)$ be an $\SU(m)$-structure on $M^n$. Then, the scalar and $*$-scalar curvature are:
    \begin{align}
    \label{eq: scalar in function of the torsion components}
        s &= 9|T_1|^2-3|T_2|^2-|T_3|^2+(2m-3)|T_4|^2-\frac{m}{m-1}\Div(\bar{T}_4)-\frac{1}{2^m}\langle d\eta,J\rangle\\ \nonumber
        s^* &= |T_1|^2+|T_2|^2-|T_3|^2-|T_4|^2-\frac{1}{2^m}\langle d\eta,J\rangle,
    \end{align}
    where $(\bar{T}_4)_j=(T_4)_{i,ij}$.
\end{proposition}
\begin{proof}
    Formula \eqref{eq: scalar in function of the torsion components} follows immediately from  \eqref{eq: su(n) scalar}, \eqref{eq: su(n) * scalar} and \eqref{eq: quadratic nabla J}.
\end{proof}

\subsection{An example: L\"ust-Tsimpis \texorpdfstring{$\SU(3)$}{}-structures}\label{sec: LT Su3}

We now illustrate the preceding torsion-class and curvature formulae with a concrete class of $\SU(3)$-structures, proposed by L\"ust--Tsimpis (LT), thereby obtaining structural information about LT geometries. Such  $\SU(3)$-structures arise in $\mathcal N=1$ massive type IIA compactifications with non-zero flux \cite{LustTsimpis}, where the internal six-manifold carries a half-flat $\SU(3)$-structure with intrinsic torsion  in $W_1\oplus W_2$, i.e., the torsion components in $W_3$, $W_4$, and $W_5$ vanish. Related constructions and examples include \cites{strominger86,Larfors13,Fidanza05,LKT08,T-T-JHEP17,QP25}.
In the mathematical literature these structures are often called ``coupled'' or ``closed''.  In this section we prove that, on compact manifolds, they split into two classes; we compute the corresponding scalar curvature, and we give two proofs that they are  $\SU(3)$-harmonic.

\begin{definition}
[\cites{chiossi2002,LustTsimpis, bedulli2007}]
\label{defLT}
An $\SU(3)$-structure determined by $(\omega,\Upsilon^\pm)$ is called an \emph{LT structure} if it satisfies the following three conditions:
    \begin{align*}
    d\om &= -\frac{3}{2}\sigma_0 \Upsilon^+ + \frac{3}{2}\pi_0\Upsilon^-;  \, d\Upsilon^+ = \pi_0 \om^2  - \pi_2 \w \om, \mbox{ and }
    d\Upsilon^- = \sigma_0 \om^2  - \sigma_2 \w \om.
\end{align*}
\end{definition}

\begin{remark}
Equivalently, LT structures correspond to $\SU(3)$-structures with torsion in $W_1\oplus W_2$. In Definition~\ref{defLT}, $\pi_0,\sigma_0 \in W_1$ and $\pi_2,\sigma_2 \in W_2$.

Looking at the  torsion component norms $|T_3|^2$, $|T_4|^2$ and $|T_5|^2$  in \eqref{eq: norm torsion components}, we see that $T\in W_1\oplus W_2$ implies $\eta =0$, $\Div J =0$ and 
\begin{align} \label{eq*}
    |\nabla J|^2 - \langle \nabla J , J \nabla_{J} J \rangle = 0.
\end{align}
On the other hand, decomposing (orthogonally) the 2-tensor $\nabla J$ into its $J$-symmetric and $J$-anti-symmetric parts
\begin{align*}
   2 \nabla J = \left( \nabla J (X,Y) - \nabla J (JX,JY)\right) + \left( \nabla J (X,Y) + \nabla J (JX,JY)\right),
\end{align*}
and rewriting \eqref{eq*} as
\begin{align}
    |\nabla J + \nabla_{J} J \circ J |^2 = 0,
\end{align}
yields the well-known  $(1,2)$-symplectic almost Hermitian property  $d\omega^{(1,2)} =0$, also sometimes called ``quasi-K\"ahler''.
Since this implies $\Div J=0$, LT structures can be characterised by $\eta =0$ and $J$ being $(1,2)$-symplectic.
\end{remark}

Our first result is a dichotomy of compact LT structures:
\begin{theorem}
\label{thm: compact LT}
    Let $(M^6,\om,\Upsilon^\pm)$ denote a compact LT $\SU(3)$-structure; then, without loss of generality, either
    \begin{align*}
        d\omega &=0, \, d\Upsilon^+ =-\pi_2 \wedge \omega, 
        \qandq d\Upsilon^- =-\sigma_2 \wedge \omega, 
        \end{align*}
    or
    \begin{align*}
        d\omega &= -\frac{3}{2}\sigma_0 \Upsilon^+,\, d\Upsilon^+ =0, \,  \qandq 
        d\Upsilon^-=\sigma_0\omega^2-\sigma_2 \wedge \omega, 
    \end{align*}
    where $\sigma_0$ is a non-zero constant. We shall call these LT-structures of \emph{type I} and \emph{type II}, respectively.
\end{theorem}
\begin{proof}
Taking the exterior derivative to the defining equations of LT structures, we obtain the relations:
\begin{align*}
    d\pi_0 &= -J(d\sigma_0),\quad \sigma_0 \pi_2 = \pi_0 \sigma_2,\quad
    d\pi_0 \w \om^2 = d\pi_2 \w \om,\quad d\sigma_0 \w \om^2 = d\sigma_2 \w \om.
\end{align*}
In particular, $dd^c\sigma_0=dd^c\pi_0=0$. Contracting these identities with $\om$, we obtain
\[
    d^*d\sigma_0=d^*d\pi_0=0.
\]
Here we used the almost Hermitian identity
\[
    g(dd^cf,\omega)=d^*df+2g(df,\nu_1),
\]
together with the vanishing of the Lee form $\nu_1$. Since $M^6$ is compact, both $\pi_0$ and $\sigma_0$ are constant. If these both are zero then the LT structure is of type I. If these are not both zero then one can perform a rotation of $\Upsilon^+ + i \Upsilon^-$ such that, without loss of generality,  $d\Upsilon^+=0$ i.e., the LT structure is of type II. 
\end{proof}
\begin{remark}
    Observe that compact LT structures of type I are symplectic and hence $b_2(M^6)=b_4(M^6) \geq 1$. However, we do not have any such topological obstructions for (compact) LT structures of type II, such as the nearly K\"ahler examples $S^6$, $\mathbb{C}\mathbb{P}^3$, $S^3 \times S^3$ and $\mathrm{SU}(3)/\mathbb{T}^2$.
\end{remark}
\begin{corollary}
    For an LT $\SU(3)$-structure, as in the previous Theorem \ref{thm: compact LT},
    \[
    \mathrm{Scal}(g)= \frac{15}{2}\sigma_0^2+\frac{15}{2}\pi_0^2 - \frac{1}{2}|\sigma_2|^2-\frac{1}{2}|\pi_2|^2.
    \]
    In particular, an LT structure of type I has non-positive scalar curvature, which vanishes  if and only if, it is torsion-free.
\end{corollary}
\begin{proof}
This follows from the scalar-curvature formula in \S\ref{sec: appendix curvature formulae}.
\end{proof}
\begin{corollary}
     LT $\SU(3)$-structures are harmonic.
\end{corollary}
\begin{proof}
    One quick proof is to combine Corollary \ref{cor: harmonic su3 cases}  and the fact that for LT structures $\pi_1=\nu_1=0$ and  $d\pi_0+J(d\sigma_0)=0$.

    An alternative method is to start with the characterisation of LT structures by $\eta =0$ and $J$ being $(1,2)$-symplectic, so the  formula~\eqref{eq: divergence_of_T_SUn} for the divergence of an $\SU(3)$-structure $(J,\Upsilon)$ reduces to
$$\Div T = \frac{1}{4} [\Delta J , J].
$$
Therefore $(J,\Upsilon)$ is harmonic as an $\SU(3)$-structure if and only if, $J$ is harmonic as a 
$\U(3)$-structure.
By \cite{Wood1995}*{Th 2.8}, a $(1,2)$-symplectic $J$ is harmonic if and only if, it commutes with the $(1,1)$-Ricci form, which in turn is equivalent to $\Ric^*$ being symmetric \cite{Wood1995}*{(R1)-(R5)}.
Now, by \cite{martin2006}*{Lemma 3.3} or indeed by the Bianchi identity,
\begin{align*}
    \Ric^* (X,Y) 
    &= - \sum_{i} \langle -\frac{1}{2} (J\nabla_{X}J)(e_i) ,  -\frac{1}{2} (J\nabla_{JY}J)(Je_i) \rangle \\
    &= - \frac{1}{4} \sum_{i} \langle (J\nabla_{X}J)(e_i) ,  (J\nabla_{JY}J)(Je_i) \rangle \\
     &= - \frac{1}{4} \sum_{i} \langle (J\nabla_{JX}J)(Je_i) ,  (J\nabla_{Y}J)(e_i) \rangle \\
     &= \Ric^* (Y,X),
\end{align*}
so $(J,\Upsilon)$ is $\SU(3)$-harmonic.
\end{proof}

\section{Negative gradient and Ricci-harmonic flows for \texorpdfstring{$\H$}{H}-structures}

The main goal of this section is to compare the negative gradient flow \eqref{equ: general gradient flow} with the Ricci-harmonic flow \eqref{equ: Ricci-harmonic}. While our primary focus in this paper is the case $\H=\mathrm{SU}(m)$, it is more illuminating to consider the general case. In \S \ref{sec: the general setup} we collect the relevant basic facts about general $\H$-flows; in \S \ref{sec: general evolution of torsion and curvature} we derive several evolution equations and useful identities; and in \S \ref{sec: negative gradient flow versus Ricci-harmonic flow} we formulate the comparison. We summarise our main findings:
\begin{itemize}
    \item Theorem \ref{thm: critical points are torsion-free} shows that the stationary points of the negative gradient $\H$-flow are always torsion-free, regardless of the choice of $\H$. Subsequently in \S \ref{sec: modifying Ricci-harmonic}, we show that a  suitable lower-order modification (depending on $\H$) of the Ricci-harmonic $\H$-flow ensures that its stationary points are torsion-free as well.
    \item Proposition \ref{prop: comparing negative gradient flows} gives a general comparison of the negative gradient flow for different $\H$-structures. In particular, the $\mathrm{U}(m)$ and $\mathrm{SU}(m)$ cases differ from the $\mathrm{G}_2$ and $\mathrm{Spin}(7)$ cases by the appearance of the $*$-Ricci term. This identifies the highest-order terms which distinguish the Ricci-harmonic flow from the negative gradient flow.
    \item Corollary \ref{cor: main result of section 3} shows that the torsion tensor has heat-type leading term along the Ricci-harmonic $\H$-flow, whereas additional highest-order terms remain for the negative gradient $\H$-flow. This distinction is the analytic reason for treating the two flows separately.
\end{itemize}

\subsection{The general setup of \texorpdfstring{$\H$}{H}-flows}
\label{sec: the general setup}

We briefly recall the general framework of $\H$-flows from  \cite{Fadel2022}. The reader might also find it helpful to compare this discussion with \S\ref{sec: H-structures}, which is specialised to the $\H=\mathrm{SU}(m)$ case.

An $\H$-structure on  $M^n$ is a reduction of the structure group of its principal frame bundle from $\mathrm{GL}(n,\mathbb{R})$ to a subgroup $\H$. We shall assume that this $\H$-structure is determined by a (multi-)tensor $\xi$, i.e., that $\H$ is the stabiliser of $\xi$. For instance, on an oriented Riemannian manifold we have $\xi=(g,\mathrm{vol})$, i.e., $\H=\mathrm{SO}(n)$. In what follows, we will only consider reductions to  $\H\subseteq \mathrm{SO}(n)$, which provides a natural Lie algebra embedding $\mathfrak{h}\hookrightarrow \mathfrak{so}(n)\cong \Lambda^2$, where the isomorphism is determined by the metric $g$. Hence there is an orthogonal decomposition of the space of $2$-forms: 
\begin{equation}
    \Lambda^2 = \Lambda_\mathfrak{h}^2 \oplus \Lambda_\mathfrak{m}^2.
\end{equation}

The intrinsic torsion of the $\H$-structure is a tensor $T\in \Lambda^1 \otimes \Lambda^2_{\mathfrak{m}}$ defined implicitly by the relation
\begin{equation}
    \nabla_l \xi = T_l \diamond \xi, \label{equ: def of torsion}
\end{equation}
where $\nabla$ denotes the Levi-Civita connection. In local coordinates, the torsion is denoted by $T_{l,ij}$ and is the first-order invariant of the $\H$-structure. Observe that the torsion vanishes when $\mathfrak{h}=\mathfrak{so}(n)$; this is essentially  the fundamental theorem of Riemannian geometry asserting the existence of a unique torsion-free metric connection. In order to obtain an explicit expression for $T$ from \eqref{equ: def of torsion}, one needs to invert the operator $(\cdot\diamond \xi)|_{\Lambda_{\fm}^2}$. Although there does not seem to be a general method for doing so without specifying $\H$, we always have the following  decomposition as $\mathrm{SO}(n)$-modules: 
\begin{equation}
\Lambda^1 \otimes \Lambda^2_{\mathfrak{m}}
\subset
    \Lambda^1 \otimes \Lambda^2 \cong \Lambda^1 \oplus \Lambda^3 \oplus U,\label{equ: split torsion}
\end{equation}
where $U$ denotes an irreducible $\mathrm{SO}(n)$-module. 
The corresponding  decomposition of the torsion is:
\begin{align*}
    VT_i
    &:=
    T_{k,ki} \in \Lambda^1,\\
    (T_{SK})_{ijk}
    &:= \frac{1}{2}T_{[i,jk]} = T_{i,jk}+ T_{j,ki} +T_{k,ij} \in \Lambda^3,\\
    (T_U)_{ijk}
    &:= T_{i,jk}-\frac{1}{n-1}(g_{ij}VT_k-g_{ik}VT_j)-\frac{1}{3}(T_{sk})_{ijk},
\end{align*}
where $T_U$ is a section of the subbundle of $\Lambda^1\otimes \Lambda^2$
so that
\begin{equation*}
    T_{i,jk}=\frac{1}{n-1}(g_{ij}VT_k-g_{ik}VT_j)+\frac{1}{3}(T_{SK})_{ijk}+ (T_U)_{ijk}
\end{equation*}
corresponds to the splitting of $T$ into three orthogonal components, in accordance with \eqref{equ: split torsion}. Of course, one can in general further decompose the above torsion components into irreducible $\H$-modules, as already illustrated in \S \ref{sec: torsion SU(m)}. The vector field $VT$ will play an important role later; in the $\SU(3)$ torsion-form notation of Appendix~\ref{sec: SU3-torsion forms - appendix}, it is given explicitly by
\[
    VT=-\frac{1}{3}(\pi_1+2\nu_1).
\]

A general flow of an $\H$-structure, determined by $\xi$, can be expressed as 
\begin{equation}
    \frac{\partial}{\partial t} \xi = A \diamond \xi,\label{equ: GF}
\end{equation}
where 
\begin{equation*}
A=S+C \in \mathrm{End}_{\mathfrak{h}}:= \Sigma^2\oplus \Lambda^2_{\mathfrak{m}} \cong \mathfrak{gl}(n,\mathbb{R})/\mathfrak{h}.
\end{equation*}
Note that  $C \in \Lambda^2_\mathfrak{m}\subset
\Lambda^2$  because $\Lambda^2_\mathfrak{h}\diamond \xi=0$ i.e., endomorphisms in  $\Lambda^2_\mathfrak{h}$ leave the $\H$-structure invariant, by definition of $\xi$. Thus,  $\mathrm{End}_{\mathfrak{h}}$ is the space of effective infinitesimal deformations of the $\H$-structure as such, and the $\H$-flow is completely determined by the choice of the symmetric and skew-symmetric tensors $S$ and $C$, respectively. We recall the two simplest, yet fundamental, examples:
\begin{example}\label{ex: ricci flow}
    The Ricci flow, with $\H=\mathrm{O}(n)$, $\xi=g$, $B=-\mathrm{Ric}$ and $C=0$:   
\begin{equation}
    \frac{\partial}{\partial t} g = (-\mathrm{Ric})\diamond g 
    = -2\mathrm{Ric}.
\end{equation}
This is arguably the most well-studied $\H$-flow; cf. \cites{HamiltonRic,topping2006,chow2011} and references therein.
\end{example}
\begin{example}
\label{ex: harmonc h-flow}
    The \emph{harmonic $\H$-flow}, also known as isometric $\H$-flow, with $B=0$ (so that the metric is stationary) and $C=\Div T$, where $(\Div T)_{jk}:=\nabla_iT_{i,jk}$:
\begin{equation}
    \frac{\partial}{\partial t} \xi 
    = \Div T\diamond \xi.
\end{equation}
    This type of flow has been studied extensively in recent years, for various structure groups $\H$, see e.g.  \cites{Grigorian2019,dgk-isometric,Dwivedi-Loubeau-SaEarp2021, Fadel2022, udhav2022quaternionic}. The nomenclature `harmonic' stems from the fact that this can be viewed as a harmonic map heat-type flow, as argued in  \cite{loubeau-saearp}.
\end{example}
In order to apply elliptic theory to obtain short-time existence of an $\H$-flow, among other desirable  analytic properties, one typically wants $A$ in \eqref{equ: GF} to be a quasilinear second-order differential invariant of $\xi$ to highest order  -- as illustrated by the  two examples above. In addition, one may need  `good' lower-order terms to enforce special properties of its stationary points, as well as reasonable long-time behaviour. We are thus naturally led to the problem of classifying the first- and second-order invariants of the $\H$-structure. 

\subsubsection{Differential invariants and classification of \texorpdfstring{$\H$}{H}-flows}
Following \cite{Bryant2006},  the space of $k$th-order invariants of an $\H$-structure, denoted by $V_k(\mathfrak{h})$, is implicitly defined by
\begin{equation}
    V_k(\mathfrak{h})\oplus (\Lambda^1 \otimes \Sigma^{k+1}) = \Sigma^k \otimes (\Sigma^2 \oplus \Lambda^2_\mathfrak{m}). \label{equ: bryant invariants}
\end{equation}
Indeed, we have already encountered the space of first-order invariants, namely the intrinsic torsion:
\begin{equation*}
    V_1(\mathfrak{h}) 
    = \Lambda^1 \otimes \Lambda^2_\mathfrak{m}.
\end{equation*}
Since $\H\subseteq \mathrm{SO}(n)$, observe also that
\begin{equation*}
V_k(\mathfrak{so}(n)) \subseteq     V_k(\mathfrak{h}).
\end{equation*}
For the next order $k=2$, $V_2(\mathfrak{so}(n))$ is just the space of Riemann curvature tensors and we have the well-known decomposition:
\begin{equation*}
V_2(\mathfrak{so}(n)) = \mathbb{R} \oplus \Sigma^2_0 \oplus {W},
\end{equation*}
where $\mathbb{R}$ is spanned by the scalar curvature, $\Sigma^2_0$ by the traceless Ricci tensor, and ${W}$ by the Weyl curvature. 
When $\H\subsetneq \mathrm{SO}(n)$, additional second-order invariants emerge from $\nabla T$, and these are related to the curvature tensor $R$ by the (torsion) Bianchi-type identity: 
\begin{proposition}\label{prop: general bianchi identity}
The torsion $T_{l,ij}$ satisfies the following Bianchi-type identity
\begin{equation*}
    (\nabla_a T_b-\nabla_bT_a-[T_b,T_a]-R_{ba})\diamond \xi=0.
\end{equation*}
Equivalently,
\begin{equation}
    \nabla_aT_{b,cd} - \nabla_bT_{a,cd} = \pi^2_{\fm}(R_{ba})_{cd}+  Q_{abcd}  ,\label{equ: general bianchi identity}
\end{equation}
where $\pi^2_\fm: \Lambda^2 \to \Lambda^2_{\fm}$ denotes the projection map, $\pi^2_{\fm}(R_{ba})$ is the $\fm$-projection of the curvature endomorphism $R_{ba}$ on the last two indices, and
\begin{equation}
\label{eq: def_Q}
  Q_{abcd}  :=-2[T_{a},T_{b}]_{cd}+\pi^2_{\fm}([T_{a},T_{b}])_{cd} \in  \Lambda^2 \otimes \Lambda^2
\end{equation}
denotes the lower-order terms quadratic in the torsion.
\end{proposition}
\begin{proof}
This is a direct computation, by applying $\nabla$ to \eqref{equ: def of torsion} and skew-symmetrising over the first two indices, see \cite{Fadel2022}*{Corollary 1.39}.
\end{proof}
We emphasise  that the spaces $V_{k}(\mathfrak{h})$  incorporate all the relations provided by the latter Bianchi-type identity and its derivatives. Since skew-symmetrisation in the first two indices of $\nabla T$ corresponds to curvature terms via \eqref{equ: general bianchi identity}, it follows that the additional second-order invariants lie in the space
$\Sigma^2 \otimes \Lambda^2_{\mathfrak{m}}$. 
Indeed one can compute directly from the definition of $V_2(\mathfrak{h})$ in \eqref{equ: bryant invariants} that
\begin{align*}
    \mathrm{dim}(V_2(\mathfrak{h})) &= \frac{1}{12}n^2(n^2-1)+\frac{1}{2}n(n+1)\cdot \mathrm{dim}(\Lambda^2_{\mathfrak{m}})\\
 &= \frac{1}{3}n^2(n^2-1)-\frac{1}{2}n(n+1)\cdot \mathrm{dim}({\mathfrak{h}}).
\end{align*}
We are now in a position to explain the procedure to classify general second-order quasilinear $\H$-flows.
\begin{itemize}
    \item Step 1: Decompose the spaces $V_2(\mathfrak{h})$ and $\mathrm{End}_{\mathfrak{h}}$ into irreducible $\H$-modules, say $$V_2(\mathfrak{h})= \bigoplus_i k_i V_i
    \qandq    \mathrm{End}_{\mathfrak{h}}  =\bigoplus_i l_i V_i, 
    \qforq
    k_i,l_i\in \mathbb{Z}^+_0,
    $$ where  $V_i$ are irreducible $\H$-modules. There is then a $(\sum_{i}k_il_i\varepsilon_i)$-parameter family of possible second-order quasilinear $\H$-flows (to highest order), where $\varepsilon_i$ corresponds to the multiplicity of any $\H$-invariant automorphism of $V_i$. An example of such automorphism is provided by the almost complex structure $J$ when  $\H=\mathrm{SU}(m)$, see Appendix \ref{sec: SU3-torsion forms - appendix} for the detailed $m=3$ case.
    \item Step 2: Find a generator for each space $V_i$ occurring in $V_2(\mathfrak{h})$ which also arises in $\mathrm{End}_\mathfrak{h}$. The deformation tensor $A$ is then just a linear combination of these elements, possibly coupled  with the $\H$-invariant automorphism. 
    This automorphism matters for instance when $\alpha$ is a $1$-form corresponding to a second-order invariant of an $\SU(3)$-structure. Then $\alpha$ and $J\alpha$ generate the same underlying irreducible $\SU(3)$-module, but the associated deformation terms $\alpha\ip\Upsilon^+$ and $(J\alpha)\ip\Upsilon^+$ define distinct $\SU(3)$-flows.
    \item Step 3: Choose the lower-order terms which are quadratic in torsion. It suffices to decompose $S^2(V_1(\mathfrak{h}))$ into $\H$-modules, then choose a set of generators which also occur in $\mathrm{End}_{\fh}$ and add them to $A$ as before.
\end{itemize}
We give a few examples to illustrate this procedure.
\begin{example}
    In the Riemannian case  $\H=\mathrm{O}(n)$, 
    \begin{align*}
	V_2(\mathfrak{so}(n)) &= \mathbb{R}\oplus \Sigma^2_0 \oplus W\\
	\mathrm{End}_{\mathfrak{so}(n)} &= \mathbb{R} \oplus \Sigma^2_0.
\end{align*}
    Thus, there is a $(1\times1+1\times 1)=2$-parameter family of flows:  \begin{equation}        \frac{\partial}{\partial t }g =(\lambda_1\mathrm{Scal} \ g+\lambda_2\mathrm{Ric}) \diamond g.
    \end{equation}
    This is the so-called Ricci-Bourguignon flow introduced in \cite{Bourguignon}, which includes Hamilton's Ricci flow \eqref{ex: ricci flow}. Since $V_1(\mathfrak{so}(n))=\{0\}$, there are no lower-order terms.
\end{example}

\begin{example}
    In the $\H=\mathrm{G}_2$ case, 
    \begin{align*}
	V_2(\mathfrak{g}_2) &= \mathbb{R} \oplus  2\Lambda^1 \oplus 3\Sigma^2_0 \oplus V_{0,1} \oplus 2V_{1,1} \oplus V_{0,2} \oplus V_{3,0}\\
	\mathrm{End}_{\mathfrak{g}_2} &= \mathbb{R} \oplus \Lambda^1 \oplus \Sigma^2_0
\end{align*}
    Thus, there is a $(1\times1+2\times 1+3 \times 1)=6$-parameter family of $\mathrm{G}_2$-flows (to second-order). The $6$ second-order invariants are  scalar curvature, traceless Ricci curvature, $\mathcal{L}_{VT}g$, a Weyl curvature term, $\Div T$ and $\Div(T^t)$, see \cite{Dwivedi2023}*{Theorem 6.2}. The space $S^2(V_1(\mathfrak{g}_2))$ is computed in \cite{fino2025}*{Theorem 9.2} and  there are 16 possible quadratic terms involving the torsion that can be added to $A$. All $\mathrm{G}_2$-flows studied so far in the literature occur as special cases in this family, e.g.   Laplacian flow, (modified) Laplacian co-flow, and harmonic flow.
\end{example}

\begin{example}
    In the $\H=\mathrm{Spin}(7)$ case, 
\begin{align*}
    V_2(\mathfrak{spin}(7)) 
    &= \mathbb{R} 
	\oplus  \Lambda^2_7 
	\oplus {2}\Sigma^2_0 
	\oplus V_{0,1,0} 
	\oplus V_{1,1,0} 
	\oplus V_{2,0,0} 
	\oplus V_{0,2,0}
	\oplus V_{1,0,2}\\
	\mathrm{End}_{\mathfrak{spin}(7)} 
    &= \mathbb{R} \oplus \Lambda^2_7  \oplus \Sigma^2_0.
\end{align*}
    Thus, there is a $(1\times1+1\times 1+2 \times 1)=4$-parameter family of 
    $\mathrm{Spin}(7)$-flows. The $4$ second-order invariants are scalar curvature, traceless Ricci curvature, $\mathcal{L}_{VT}g$ and $\Div T$; this agrees with  \cite{Dwivedi2025gradient}. The space $S^2(V_1(\mathfrak{spin}(7))$ is computed in \cite{FowdarSpin7quotient}*{\S 4} and  there are 7 possible quadratic terms involving the torsion that can be added to $A$. 
\end{example}
\begin{example}
    In the $\H=\mathrm{SU}(2)=\mathrm{Sp}(1)$ case, we have 
    \begin{align*}
		V_2(\mathfrak{su}(2)) &= 9\mathbb{R} 
		\oplus  12\mathbb{R}^3 
        \oplus \mathbb{R}^5\\
        \mathrm{End}_{\mathfrak{su}(2)} &= 4\mathbb{R} 
		\oplus  3\mathbb{R}^3 
\end{align*}
    Thus, there is a $(9\times4+12\times 3)=72$-parameter family of $\mathrm{SU}(2)$-flows, see \cite{udhav2024} for the generators of the corresponding  differential invariants, as well as those of $S^2(V_1(\mathfrak{su}(2)))$. 
\end{example}
The case $\H=\mathrm{SU}(3)$ is described in detail in Appendix \ref{sec: SU3-torsion forms - appendix}. In particular, the above procedure leads to the classification in Theorem \ref{thm: general second order SU(3) flow}. We next investigate how the torsion and curvature evolve under the general $\H$-flow  \eqref{equ: GF}. Intuitively, one would like such quantities to evolve by a `heat-type flow', for better regularity.

\subsection{General evolution of torsion and curvature}\label{sec: general evolution of torsion and curvature}

We derive the evolution formula for the torsion and its derivatives under the general $\H$-flow \eqref{equ: GF}. These will be specialised to the Ricci-harmonic and negative gradient flow in \S \ref{sec: comparing Ricci-harmonic and -ve gradient}.
First we recall the standard formulae for the evolution of the associated Riemannian geometry.  Throughout this subsection, the metric evolves by $\partial_t g=2S$, and hence
\begin{equation}\label{eq: standard metric variation}
    \frac{\partial}{\partial t}g^{ij}=-2S^{ij},\qquad
    \frac{\partial}{\partial t} \Gamma_{ij}^k
    =g^{kl}(\nabla_i S_{jl}+\nabla_jS_{il}-\nabla_lS_{ij}).
\end{equation}
\begin{lemma}
\label{lemma: evolution curvature}
Under the general $\H$-flow \eqref{equ: GF}, the curvatures of the compatible metric evolve by:
\begin{align*}
    \frac{\partial}{\partial t}{R_{ijk}}^l
    &=g^{lp}(\nabla_i\nabla_jS_{kp} +\nabla_i\nabla_kS_{jp}-\nabla_i\nabla_pS_{jk}-\nabla_j\nabla_iS_{kp}-\nabla_j\nabla_kS_{ip} +\nabla_j\nabla_pS_{ik}),\\
    \frac{\partial}{\partial t}\mathrm{Ric}_{jk}
    &=-\nabla_p\nabla_pS_{jk} +\nabla_p\nabla_jS_{kp} +\nabla_p\nabla_kS_{jp}-\nabla_j\nabla_kS_{pp},\\
    \frac{\partial}{\partial t}\mathrm{Scal}&= 2(-\nabla_i\nabla_iS_{jj} +\nabla_i\nabla_jS_{ij}-S_{ij}R_{ij}).
\end{align*}
\end{lemma}
\begin{proof}
    Since $\partial_t g = 2S$ under \eqref{equ: GF}, these follow easily from \eqref{eq: standard metric variation}, see for instance \cites{topping2006, chow2011}.
\end{proof}

\begin{proposition}\label{prop: evolution of nabla alpha}
Under the general $\H$-flow \eqref{equ: GF}, if $\alpha$ is an arbitrary tensor on $M$ determined by $\xi$, then
\begin{equation}
    \frac{\partial}{\partial t}\nabla_l\alpha 
    =\nabla_l \frac{\partial}{\partial t}\alpha-(\nabla_lS+\Lambda\nabla S_l)\diamond \alpha,
\end{equation}
    where $(\Lambda\nabla S_l)_{ab} 
    := \nabla_a S_{bl} - \nabla_b S_{al} \in \Lambda^2 \otimes \Lambda^1$. 
    In particular, if $\alpha=\xi$,
 \begin{equation}
        \frac{\partial}{\partial t}\nabla_l\xi= A\diamond \nabla_l\xi+(\nabla_lC-\Lambda\nabla S_l)\diamond \xi. \label{equ: evolution of nabla xi}
    \end{equation}
\end{proposition}
\begin{proof}
Let $\alpha=\alpha^I_J\in \Gamma(\otimes^p TM \otimes^q T^*M)$, where $I=\{i_1,...,i_p\}$ and $J=\{j_1,\dots,j_q\}$ are index sets. We shall write $I_\mu(i)$ and $J_\mu(j)$ for the sets obtained by replacing indices $i\in I$ and $j\in J$ with the artificial index $\mu$, respectively. Using the evolution of Christoffel symbols \eqref{eq: standard metric variation}, we now compute:
\begin{align*}
    \frac{\pt}{\pt t}(\nabla_l\alpha)^I_J=&\ \nabla_l\left(\frac{\pt}{\pt t}(\alpha_l)^I_J\right)-\sum_{j\in J}\left(\frac{\pt}{\pt t}\Gamma_{lj}^\mu\right)\alpha^I_{J_\mu(j)}+\sum_{i\in I}\left(\frac{\pt}{\pt t}\Gamma_{l\mu}^{i}\right)\alpha^{I_\mu(i)}_J \\
    =&\ \nabla_l\left(\frac{\pt}{\pt t}(\alpha_l)^I_J\right)-\sum_{j\in J} g^{\mu k}(\nabla_l S_{jk}+\nabla_j S_{lk}-\nabla_k S_{lj})\alpha^I_{J_\mu(j)}\\
    &\ +\sum_{i\in I}g^{ik}(\nabla_l S_{\mu k}+\nabla_\mu S_{lk}-\nabla_k S_{l\mu})\alpha^{I_\mu(i)}_J\\    
    =&\ \nabla_l\left(\frac{\pt}{\pt t}(\alpha_l)^I_J\right)+\big(g^{ik}\nabla_l S_{\mu k}-g^{ik}(\nabla_k S_{\mu l}-\nabla_\mu S_{kl})\big)\alpha^{I_\mu(i)}_J\\
    &\ -\big(g^{\mu k}\nabla_lS_{jk}-g^{\mu k}(\nabla_k S_{jl}-\nabla_j S_{kl})\big)\alpha^I_{J_\mu(j)}.   
\end{align*}
The evolution equation for $\nabla_l\alpha$ follows from the definition of $\diamond$ and setting $$(\Lambda\nabla S_l)^i_j=(\Lambda\nabla S_l)_{jk}g^{ki}=(\nabla_j S_{kl}-\nabla_k S_{jl})g^{ki}.$$ The evolution for $\nabla_l\xi$ follows by setting $\alpha=\xi$ and using $\partial_t\xi=A\diamond\xi=(S+C)\diamond\xi$.
\end{proof}
The above proposition yields the evolution of the covariant derivative of a tensor in terms of the evolution of the tensor itself. In particular, this can be iterated for the evolution of higher derivatives. Before we derive the evolution of torsion, let us make explicit the algebraic input used to track the motion of the splitting $\Lambda^2=\Lambda^2_{\fh}\oplus\Lambda^2_{\fm}$.  In the cases considered below, the projection onto $\Lambda^2_{\fh}$ has the form
\begin{equation}\label{eq: projection pi h algebraic}
    \pi^2_{\fh}(\alpha)_{ij}=a_{\H}\alpha_{ij}+b_{\H}\alpha_{ab}\Xi_{pqij}g^{ap}g^{bq},
\end{equation}
where $a_{\H},b_{\H}$ are constants depending on $\H$ and $\Xi$ is an algebraic tensor determined by the $\H$-structure.  This hypothesis holds for all groups listed in Table~\ref{table_1}.  The first projected identity in the proof below is formal for an arbitrary closed subgroup $\H\subset\SO(n)$; the full pointwise formula for $\partial_tT$ uses \eqref{eq: projection pi h algebraic} to compute the $\Lambda^2_{\fh}$-component.

\begin{proposition}
\label{prop: general evolution of torsion}
Under a general $\H$-flow \eqref{equ: GF}, the torsion tensor evolves by
\begin{equation}
    \frac{\partial}{\partial t} T_{l,ij} = \pi^2_{\fm}(\nabla_l C )_{ij} - \pi^2_{\fm}(\Lambda\nabla S_l)_{ij} + (A \diamond T_l)_{ij}, \label{equ: general evolution of torsion}
\end{equation}
where $(\Lambda\nabla S_l)_{ab} := \nabla_a S_{bl} - \nabla_b S_{al}\in \Omega^2(M,T^*M)$. 
\end{proposition}
\begin{proof}
The proof will be carried out in two parts, first  deriving  $\pi^2_{\fm}(\partial_tT)$ and then $\pi^2_{\fh}(\partial_tT)$. 
Applying $\partial_t$ to \eqref{equ: def of torsion} and using \eqref{equ: GF}, we obtain
\begin{equation}
    \partial_t(\nabla_l \xi) = \partial_t T_l \diamond \xi + T_l \diamond(A\diamond \xi).
\end{equation}
Hence, we compute
\begin{align*}
  \partial_t T_l \diamond \xi &= (\nabla_lC-\Lambda\nabla S_l)\diamond \xi+  A\diamond \nabla_l\xi - T_l \diamond(A\diamond \xi)\\
  &= (\nabla_lC-\Lambda\nabla S_l)\diamond \xi +A\diamond (T_l\diamond \xi)- T_l \diamond(A\diamond \xi)\\
  &= (\nabla_lC-\Lambda\nabla S_l+ A \diamond T_l)\diamond \xi,
\end{align*}
where for the first equality we used  \eqref{equ: evolution of nabla xi}, for the second we used \eqref{equ: def of torsion}, and the last one is a standard property of the $\diamond$ operator, see \cite{Fadel2022}*{Lemma 1.4 (ii)}. Since $\Lambda^2_{\fh}:=\mathrm{ker}(\Lambda^2 \diamond \xi)$, it follows that
\begin{equation*}
    \pi^2_{\fm}(  \partial_t T_l ) = \pi^2_{\fm}(\nabla_lC-\Lambda\nabla S_l+ A \diamond T_l).
\end{equation*}
We now compute the $\Lambda^2_{\fh}$-component under the algebraic hypothesis \eqref{eq: projection pi h algebraic}.  Thus we write
\begin{equation*}
    \pi^2_{\fh}(\alpha)_{ij}  = c_1 \alpha_{ij}+c_2\alpha_{ab}\Xi_{pqij}g^{ap}g^{bq},
\end{equation*}
with constants $c_1,c_2$ depending on $\H$.  For simplicity, one typically writes  $\alpha_{ab}\Xi_{abij}=\alpha_{ab}\Xi_{pqij}g^{ap}g^{bq}$, but it will be useful to display the contractions explicitly when taking a time-derivative.

Given a general $2$-form $\hat{C}\in \Lambda^2_{\mathfrak{m}}$,  we have $\pi^2_{\fh}(\hat{C})=0$. As a consequence we see
\begin{align*}
0= \partial_t(\pi^2_{\fh}(\hat{C})) =    (\partial_t \pi^2_{\fh})(\hat{C})+\pi^2_{\fh}(\partial_t \hat{C}).
\end{align*}
In particular, this shows that the evolution of the $\Lambda^2_{\fh}$-component of $\hat{C}$ is a lower-order term, i.e., one involving only the torsion but not its derivatives. More explicitly, using the above expression for $\pi^2_{\fh}$, we have   
\begin{align*}
\pi^2_{\fh}(\partial_t \hat{C})_{ij} &= - c_2\hat{C}_{ab}\partial_t( \Xi_{pqij}g^{ap}g^{bq})\\
&= - c_2 \hat{C}_{ab}\Big((A\diamond \Xi)_{abij}-2\Xi_{pbij}S_{ap}-2\Xi_{aqij}S_{bq}\Big),
\end{align*}
where for the last line we used that $\Xi$ is algebraically determined by the $\H$-structure, so $\partial_t\Xi=A\diamond\Xi$ in the notation of \cite[Lemma~1.4]{Fadel2022}, together with the inverse-metric variation \eqref{eq: standard metric variation}; the remaining contractions are taken with the evolving metric. We now expand
\begin{align*}
    (A\diamond \Xi)_{abij}-2\Xi_{pbij}S_{ap}-2\Xi_{aqij}S_{bq} =\ &A_{ak}\Xi_{kbij}+ A_{bk}\Xi_{akij}+ A_{ik}\Xi_{abkj}+A_{jk}\Xi_{abik}-2\Xi_{pbij}S_{ap}-2\Xi_{aqij}S_{bq}\\
    =\ &+C_{ak}\Xi_{kbij}+ C_{bk}\Xi_{akij}+ C_{ik}\Xi_{abkj}+C_{jk}\Xi_{abik}\\
    &-S_{ak}\Xi_{kbij}- S_{bk}\Xi_{akij}+ S_{ik}\Xi_{abkj}+S_{jk}\Xi_{abik}.
\end{align*}
On the other hand, we have
\begin{align*}
    \pi^2_{\fh}(A\diamond \hat{C})_{ij} &=c_1(A_{ik}\hat{C}_{kj}+A_{jk}\hat{C}_{ik})+c_2(A_{ak}\hat{C}_{kb}+A_{bk}\hat{C}_{ak})\Xi_{abij}\\
    &=-c_2(A_{ik}\hat{C}_{ab}\Xi_{abkj}+A_{jk}\hat{C}_{ab}\Xi_{abik})+c_2(A_{ak}\hat{C}_{kb}+A_{bk}\hat{C}_{ak})\Xi_{abij}\\
    &= c_2\hat{C}_{ab} \Big(-A_{ik}\Xi_{abkj}-A_{jk}\Xi_{abik}+A_{ka}\Xi_{kbij}+A_{kb}\Xi_{akij}\Big)\\
    &= -c_2 \hat{C}_{ab}\Big( C_{ak}\Xi_{kbij}+ C_{bk}\Xi_{akij}+ C_{ik}\Xi_{abkj}+C_{jk}\Xi_{abik}\\
    &\qquad-S_{ak}\Xi_{kbij}- S_{bk}\Xi_{akij}+ S_{ik}\Xi_{abkj}+S_{jk}\Xi_{abik}\Big)
\end{align*}
where for the second equality we used that $\pi^2_{\fh}(\hat{C})_{ij}:=c_1\hat{C}_{ij}+c_2\hat{C}_{ab}\Xi_{abij}=0$.
Thus, 
\begin{equation*}
    \pi^2_{\fh}(\partial_t\hat{C})=\pi^2_{\fh}(A\diamond \hat{C})
\end{equation*}
and
the assertion follows by setting $\hat{C}_{ij}=T_{l,ij}$.
\end{proof}
In each context, one needs an explicit expression for the projection map $\pi_{\mathfrak{m}}:\Lambda^2\to \Lambda^2_{\mathfrak{m}}$ occurring in \eqref{equ: general evolution of torsion}. Similarly, we obtain corresponding expressions for $\pi^2_{\fh}=(\mathrm{Id}-\pi^2_{\fm})$. We describe the cases most relevant to our purposes:
\begin{enumerate}
    \item  When $\H=\mathrm{SU}(2)$, we have $\pi^2_{\fm}=\pi^2_{+}:\Lambda^2 \to \Lambda^2_+$. Explicitly,
    \[ (\pi^2_{+}(\alpha))_{ij} = \frac{1}{2}(\alpha_{ij}+\frac{1}{2}\alpha_{kl}\varepsilon_{klij}),
    \]
    where $\varepsilon_{klij}$ denotes the Levi-Civita symbol for total skew-symmetrisation of the indices. Here we follow the conventions in \cite{udhav2024} and identify $\mathfrak{h}=\mathfrak{su}(2)$ with the space of anti-self-dual $2$-forms; other conventions in the literature use self-dual forms instead.
    \item  When $\H=\mathrm{G}_2$, we have $\pi^2_{\fm}=\pi_7^2:\Lambda^2 \to \Lambda^2_7$. Explicitly,
    \[ (\pi_7^2(\alpha))_{ij} = \frac{1}{3}(\alpha_{ij}-\frac{1}{2}\alpha_{kl}\psi_{klij}),
    \]
    following the convention in \cite{Dwivedi2023}*{(2.18)} and denoting the $\mathrm{G}_2$ dual $4$-form by $\psi=*\varphi$.
    \item  When $\H=\mathrm{Spin}(7)$, we have $\pi^2_{\fm}=\pi_7^2:\Lambda^2 \to \Lambda^2_7$. Explicitly,
    \[ (\pi_7^2(\alpha))_{ij} = \frac{1}{4}(\alpha_{ij}-\frac{1}{2}\alpha_{kl}\Phi_{klij}),
    \]
    following the convention in \cite{Dwivedi2025gradient}*{(2.10)} and denoting the $\mathrm{Spin}(7)$ $4$-form by $\Phi$.
    \item When $\H=\mathrm{U}(m)$, with $n=2m$, we have $\pi^2_\fm=\pi^{[2,0]}:\Lambda^2 \to [\![\Lambda^{2,0}]\!]$, where $[\![\Lambda^{2,0}]\!]$ denotes the real vector space underlying $\Lambda^{2,0}\oplus\Lambda^{0,2}\cong \mathfrak{u}(m)^{\perp}\otimes \mathbb{C}$. Explicitly, 
    \begin{equation*}
        (\pi^{[2,0]}(\alpha))_{ij}=\frac{1}{2}(\alpha_{ij}-\alpha_{kl}\omega_{ki}\omega_{lj})
    \end{equation*}
    If instead $\H=\mathrm{SU}(m)$, then we have $\pi^2_{\fm}=\pi^{2}_{1}\oplus\pi^{[2,0]}$, where $\pi^2_{1}$ is  the projection on the $\omega$-component: 
    \begin{equation*}
        (\pi^{2}_{1}(\alpha))_{ij}=\frac{1}{2m}\alpha_{kl}\omega_{kl}\omega_{ij}.
    \end{equation*}
\end{enumerate}

\begin{remark}
It is sometimes convenient to write the projection map $\pi^2_{\fm}$ in the form
\begin{equation}\label{equ: projection map pi2m}
(\pi^2_{\fm}(\alpha))_{ij}=c_{\H}\alpha_{ij}+\tilde{c}_{\H}\alpha_{kl}\Xi_{klij},
\end{equation}
where the constants $c_{\H},\tilde{c}_{\H}$ and the $4$-tensor $\Xi$ depend only on the structure group $\H$. In the cases $\H=\{1\},\mathrm{SU}(2),\mathrm{G}_2,\mathrm{Spin}(7)$, the $4$-tensor $\Xi$ is in fact a $4$-form. However, this is not always the case; for instance when $\H=\mathrm{U}(m)$, $n=2m$, and $\Xi_{ijkl}=\omega_{ik}\omega_{jl}$. The extra symmetries in the former cases do lead to some nicer formulae, see for instance Proposition \ref{prop: comparing negative gradient flows} below.

The same projection ansatz also covers some further examples. For instance, an $\mathrm{Sp}(q)$-structure on $M^{4q}$ determined by $(g,I,J,K)$ satisfies
\[
   \pi^2_{\mathfrak{sp}(q)}(\alpha) = \frac{1}{4}\Big( \alpha(\cdot,\cdot)+
   \alpha(I\cdot,I\cdot)+
   \alpha(J\cdot,J\cdot)+
   \alpha(K\cdot,K\cdot)\Big),
   \]
   for $\alpha\in\Lambda^2(M^{4q})$.  Equivalently,
   \[
   \pi^2_{\mathfrak{m}}(\alpha)_{ij}= \frac{3}{4}\alpha_{ij}
   -\frac{1}{4}\alpha_{kl}(I_{ki}I_{lj}+J_{ki}J_{lj}+K_{ki}K_{lj}),
   \]
   which is of the form \eqref{equ: projection map pi2m} with
   $c_{\H}=\frac{3}{4}$, $\tilde{c}_{\H}=-\frac{1}{4}$ and
   $\Xi_{klij}=I_{ki}I_{lj}+J_{ki}J_{lj}+K_{ki}K_{lj}$.
\end{remark}
\begin{table}[ht!]
\centering
\renewcommand{\arraystretch}{1.3} % Increases row height for better spacing
\setlength{\tabcolsep}{12pt} % Adjusts horizontal spacing between columns
\begin{tabular}{@{}lcccccc@{}}
\toprule
 $\H$ & $\{1\}$ & $\SU(2)$ & $\mathrm{G}_2$ & $\mathrm{Spin}(7)$ & $\mathrm{U}(m)$ & $\mathrm{SU}(m)$ \\
\midrule
$\mathrm{dim}(M)$ & $n$ & $4$ & $7$ & $8$ & $n=2m$ & $n=2m$ \\
$c_{\H}$ & $1$ & $\tfrac{1}{2}$ & $\tfrac{1}{3}$ & $\tfrac{1}{4}$ &  $\tfrac{1}{2}$ & $\tfrac{1}{2}$\\
$\tilde{c}_{\H}$ & $0$ & $\tfrac{1}{4}$ & $-\tfrac{1}{6}$ & $-\tfrac{1}{8}$ &  $-\tfrac{1}{2}$ & $-\tfrac{1}{2}$\\
$\Xi_{ijk\ell}$ & $0$ & $\varepsilon_{ijk\ell}$ & $\psi_{ijk\ell}$ & $\Phi_{ijk\ell}$ &  $\omega_{ik}\omega_{jl}$ & $\omega_{ik}\omega_{jl}-\frac{1}{m}\omega_{ij}\omega_{kl}$\\
\bottomrule
\end{tabular}
\caption{The projection map \eqref{equ: projection map pi2m} for some groups $\H$. Here $\varepsilon_{ijk\ell}$ is the Levi--Civita symbol,  $\psi$ and $\Phi$ denote, respectively, the $\mathrm{G}_2$ and $\mathrm{Spin}(7)$ structure $4$-forms, and $\omega$ is the K\"ahler form in the almost Hermitian cases.}\label{table_1}
\end{table}

We record the following useful identity:
\begin{lemma}\label{lem: projection of nabla C}
For any $\hat{C}\in \Lambda^2_{\fm}$, one has
    \begin{equation}
        \pi^2_{\fm}(\nabla_l \hat{C})=\nabla_l \hat{C} - \pi^2_{\fh}(T_l\diamond \hat{C}).
    \end{equation}    
\end{lemma}
\begin{proof}
     By definition of the intrinsic torsion, we have $\nabla_l=\nabla^{\H}_l-T_l$, where $\nabla^{\H}$ denotes an $\H$-connection. Hence we compute:
    \begin{align*}
        \pi^2_{\fm}(\nabla_l\hat{C})
        &=\pi^2_{\fm}(\nabla_l^{\H}\hat{C}+T_l\diamond \hat{C})\\
        &=\nabla_l^{\H}\pi^2_{\fm}(\hat{C})+\pi^2_{\fm}(T_l\diamond \hat{C})\\
        &= (\nabla_l+T_l)(\hat{C})+\pi^2_{\fm}(T_l\diamond \hat{C})\\
        &= \nabla_l\hat{C} + (-1+\pi^2_{\fm})(T_l\diamond \hat{C})\\
        &= \nabla_l\hat{C} -\pi^2_{\fh}(T_l\diamond \hat{C}),
    \end{align*}
    where we used that $\nabla^{\H}\pi^2_{\fm}(\hat{C})=\pi^2_{\fm}(\nabla^{\H}\hat{C})$ since $\nabla^{\H}$ preserves $\pi^2_{\fm}$, and $T_l\hat{C}=-T_l\diamond \hat{C}$ (this follows from the definition of $\diamond$).
\end{proof}
\begin{corollary}
    Under the general $\H$-flow \eqref{equ: GF}, the torsion $1$-form  $VT_i$ evolves by:
\begin{equation}
    \frac{\partial}{\partial t} VT_i = \nabla_kC_{ki}-\pi^2_{\fm}(\Lambda\nabla S_k)_{ki} -\pi^2_{\fh}(T_k\diamond C)_{ki}+(A\diamond T_k)_{ki}-2T_{p,qi}S_{pq}
\end{equation}
\end{corollary}
\begin{proof}
    Taking the time derivative of $VT_i=T_{p,qi}g^{pq}$, we use Proposition \ref{prop: general evolution of torsion} and Lemma \ref{lem: projection of nabla C}.
\end{proof}
\begin{corollary}
\label{cor: evolution of derivatives of torsion}
Under the general $\H$-flow \eqref{equ: GF}, the derivatives of the torsion evolve by:
\begin{equation*}
     \frac{\partial}{\partial t}\nabla^k T_{l,ij} = \nabla^k\nabla_{l}C_{ij}-\pi^2_{\fm}(\nabla^k\Lambda \nabla S_l)_{ij} + \mathcal{Q},
\end{equation*}   
where $\nabla^k:=\nabla\circ\overset{k}{\cdots}\circ\nabla$  and $\mathcal{Q}$ is a polynomial in the lower-order terms $\nabla^{r}A$ and $\nabla^{r}T$, for $r\leq k$.
\end{corollary}
\begin{proof}
    This follows directly from Proposition \ref{prop: evolution of nabla alpha}, Proposition \ref{prop: general evolution of torsion} and Lemma \ref{lem: projection of nabla C}.
\end{proof}

\subsection{Negative gradient \texorpdfstring{$\H$}{}-flow versus Ricci-harmonic \texorpdfstring{$\H$}{}-flow}\label{sec: negative gradient flow versus Ricci-harmonic flow}

Among all possible $\H$-flows, two natural choices stand out independently of the choice of the group $\H$. The first one is the coupling of the Ricci and harmonic flows in Example \ref{ex: ricci flow} and \ref{ex: harmonc h-flow}, called the \emph{Ricci-harmonic $\H$-flow}:
\begin{equation}\label{equ: Ricci-harmonic}
    \frac{\partial}{\partial t} \xi = \big(-\mathrm{Ric}+\Div T\big)\diamond \xi.
\end{equation}
The second one is provided by the negative gradient flow of the so-called Dirichlet functional:
\begin{equation}
    \mathcal{E}(\xi) := \frac{1}{2}\int_M |T|^2 \mathrm{vol}. \label{equ: general energy functional}
\end{equation}
From \cite{Fadel2022}*{Proposition 1.45}, under the general evolution \eqref{equ: GF}, the first variational formula is
\begin{align}
    \frac{d}{dt}\Bigr|_{t=0}\cE(\xi)=&\int_M\Big(-\big(-\mathrm{sym}(\Div T^t)+(T\star T)_{ia}-\frac{1}{2}|T|^2g_{ia}\big)S_{pq}-\Div T_{ia}C_{pq}\Big)g^{ip}g^{aq}\vol_g,
\end{align}
where we recall the definitions:
\begin{align*}
\mathrm{sym}(\Div T^t)_{ij}
    &:= \nabla_p T_{i,pj}+\nabla_p T_{j,pi} \in \Sigma^2(M),\\
(\Div T)_{ij}
    &:= \nabla_p T_{p,ij} \in \Omega^2_{\mathfrak m}(M)\subset\Omega^2(M),\\
(T\star T)_{ij}
    &:= T_{i,pq}T_{j,pq} \in \Sigma^2(M),\\
|T|^2
    &:= T_{i,jk}T_{i,jk} \in C^\infty(M).
\end{align*}

It follows that the \emph{negative gradient flow} of \eqref{equ: general energy functional} is given by
\begin{equation}
\frac{\partial}{\partial t}\xi = \Big(\mu\big(  
-\mathrm{sym}(\Div T^t)+
   T\star T -\frac{1}{2}|T|^2g \big)
   +\Div T
\Big) \diamond \xi,   \label{equ: general gradient flow}
\end{equation}
where $\mu \in \mathbb{R}^+$ denotes a free parameter scaling the speed of flow in the symmetric component. One could of course scale the skew-symmetric component as well, but we can always fix one of these parameters by simply rescaling $\xi$, so we avoid this redundancy. At this point, the reader might wonder why we did not consider a similar scaling factor in \eqref{equ: Ricci-harmonic}, but it will become apparent when we derive the evolution equation for the torsion that this flow is already optimally scaled.

\subsubsection{Stationary points of the negative gradient flow}
We obtain the direct Euler--Lagrange form for one consequence of
\cite{FadelLoubeau2026}*{Theorem 4.1(3)}.  That result proves, in the
unrestricted space of $\H$-structures, that the critical points of the intrinsic
torsion energy are precisely the torsion-free structures.  We nevertheless keep
the short proof below, since it extracts from the first variation the trace
identity that will be used immediately afterwards to motivate the modified
Ricci-harmonic flow.

\begin{theorem}
\label{thm: critical points are torsion-free}
Let $M$ be compact of real dimension greater than $2$.  Then the stationary
points of the negative gradient flow~\eqref{equ: general gradient flow} are
precisely the torsion-free $\H$-structures, i.e., they satisfy $T=0$.
\end{theorem}
\begin{proof}
    From \eqref{equ: general gradient flow}, by separating the symmetric and skew-symmetric terms, a stationary point of the Dirichlet functional
\eqref{equ: general energy functional} satisfies
\begin{align}
    \mathrm{sym}(\Div T^t)-
    T\star T +\frac{1}{2}|T|^2g  
    &= 0, \label{equ: cpt 1}\\
    \Div T &=0.
    \label{equ: cpt 2}
\end{align}
Taking the trace of \eqref{equ: cpt 1}, we obtain
\begin{align}
0 &= 2\nabla_p T_{i,pi}- T_{i,pq}T_{i,pq}+\frac{1}{2} T_{i,pq}T_{i,pq} \dim(M) \nonumber\\
&=  \pm 2 \delta VT+\Big(\frac{1}{2}\dim(M)-1 \Big)|T|^2,\label{eq: trace_of_S}
\end{align}
where the first term of \eqref{eq: trace_of_S} is the co-differential of the
$1$-form $VT_p:=T_{i,ip}$; in the $\SU(3)$ case this one-form is equal
to $-\frac{1}{3}(\pi_1+2\nu_1)$, see Appendix~\ref{sec: SU3-torsion forms - appendix}.
Since $M$ is compact and $\dim(M)>2$, we conclude by integration of \eqref{eq: trace_of_S}:
\[
\int_M |T|^2\,\mathrm{vol}=0.
\]
Conversely, if $T=0$, then the right-hand side of
\eqref{equ: general gradient flow} vanishes, so the structure is stationary.
\end{proof}

\begin{remark}
The conclusion of Theorem~\ref{thm: critical points are torsion-free} is a
local Euler--Lagrange proof of the critical-point assertion in
\cite{FadelLoubeau2026}*{Theorem 4.1(3)}.  The latter result gives the stronger
scale-degeneracy statement for arbitrary closed and connected
$\H\leqslant \SO(n)$, $n>2$: on the unrestricted space of $\H$-structures,
homothetic collapse provides a path
\[
    \xi_t=e^{-t}\mathrm{Id}_{TM}\cdot \xi_0,
    \qquad
    g_{\xi_t}=e^{-2t}g_{\xi_0},
    \qquad
    \mathcal E(\xi_t)=e^{-(n-2)t}\mathcal E(\xi_0).
\]
Consequently, the infimum of the unrestricted torsion energy is zero on every
nonempty path component; its critical points are exactly the torsion-free
structures; and any component containing no torsion-free structure has no
minimizer and no critical point.  The fixed-isometric-class restriction is
therefore not merely a convenient analytic choice: it removes this artificial
homothetic escape direction and leaves a genuinely variational problem.

The proof above should be understood as the infinitesimal counterpart of that
scale-degeneracy theorem.  Instead of using the explicit homothetic path, it
uses the symmetric Euler--Lagrange equation and traces it to obtain
\eqref{eq: trace_of_S}.  This trace identity is useful in the present paper
because it identifies the lower-order term that distinguishes the negative
gradient flow from the Ricci-harmonic flow and motivates the modifications
introduced in \S\ref{sec: modifying Ricci-harmonic}.  In the special
$\mathrm{G}_2$ case, the same conclusion was previously obtained by
Weiss--Witt \cite{WeissWitt2012}*{Corollary 4.3}; in the $\mathrm{Spin}(7)$
case, see also Dwivedi's recent article \cite{Dwivedi2025gradient}.
\end{remark}

\subsubsection{Modifying stationary points of the Ricci-harmonic flow}
\label{sec: modifying Ricci-harmonic}

Unlike the unrestricted negative gradient flow, the Ricci-harmonic flow \eqref{equ: Ricci-harmonic} has stationary points given by Ricci-flat harmonic $\H$-structures, and these need not be torsion-free. However, the trace identity in the proof of Theorem \ref{thm: critical points are torsion-free} suggests a natural lower-order modification to address this. Consider the \emph{modified Ricci-harmonic flow}:
\begin{equation}
%\tag{mRH}
    \frac{\partial}{\partial t} \xi 
    = \Big(-{\rm{Ric}}(g)+ (\lambda-n\hat{\lambda}) (T\star T)+ \hat{\lambda}|T|^2g + {\rm{div}}\ T\Big)\diamond \xi, \label{eq: modified Ricci-harmonic H flow}
\end{equation}
where $\lambda,\hat{\lambda}$ are arbitrary constants and $n=\dim M$. 
Considering the trace of the evolution tensor $A$, we find that a stationary point of \eqref{eq: modified Ricci-harmonic H flow} satisfies:
\begin{align*}
 -\mathrm{Scal}(g)+\lambda |T|^2 = 0.
\end{align*}
It turns out that in many cases, by choosing the constant $\lambda$ to lie in a suitable range, we can guarantee that the stationary points are always torsion-free on a compact manifold. Here are some concrete examples:
\begin{enumerate}
\item In the $\mathrm{G}_2$ case, from \cite{Bryant2006}*{(4.28)}, we have
\[
\mathrm{Scal}(g)= 12 \delta \tau_1 + \frac{21}{8}\tau_0^2+30|\tau_1|^2-\frac{1}{2}|\tau_2|^2-\frac{1}{2}|\tau_3|^2.
\]
One can also check that
\[
|T|^2= \frac{7}{24}\tau_0^2+4|\tau_1|^2+\frac{1}{3}|\tau_2|^2+\frac{1}{3}|\tau_3|^2.
\]
and hence
\[
-\mathrm{Scal}(g)+\lambda |T|^2 = -12\delta \tau_1 + \big(-\frac{21}{8}+\frac{7}{24}\lambda \big)\tau_0^2 
+ \big(-30+4\lambda \big)|\tau_1|^2 
+ \big(\frac{1}{2}+\frac{1}{3}\lambda \big)|\tau_2|^2 
+ \big(\frac{1}{2}+\frac{1}{3}\lambda \big)|\tau_3|^2.
\]
Notice that the first term is a divergence and hence integrates to zero. Thus, it suffices to choose $\lambda$ so as to ensure a definite sign for the remaining terms.
By a simple computation, one may choose any $\lambda$ such that either $-\frac{21}{8}+\frac{7}{24}\lambda>0\iff\lambda> 9$, or $\frac{1}{2}+\frac{1}{3}\lambda<0\iff\lambda<-\frac{3}{2}$, to ensure that the critical points of \eqref{eq: modified Ricci-harmonic H flow} for $\xi=\varphi$ are torsion-free $\mathrm{G}_2$-structures.
\item In the $\mathrm{Spin}(7)$ case, from \cite{FowdarSpin7quotient}*{Proposition 4.1}, we have
\[
\mathrm{Scal}(g)= \frac{7}{2} \delta T^1_8 + \frac{21}{8}|T^1_8|^2-\frac{1}{2}|T^5_{48}|^2.
\]
One can also check that
\[
|T|^2= \frac{7}{32}|T^1_8|^2+\frac{1}{4}|T^5_{48}|^2.
\]
As above, one may choose any $\lambda$ such that either $-\frac{21}{8}+\frac{7}{32}\lambda>0\iff\lambda> 12$, or $\frac{1}{2}+\frac{1}{4}\lambda<0\iff\lambda <-2$, to ensure that the critical points of \eqref{eq: modified Ricci-harmonic H flow} for $\xi=\Phi$ are torsion-free $\mathrm{Spin}(7)$-structures.

\item In the $\mathrm{SU}(3)$ case, from \cite{bedulli2007}*{(3.11)}, we have 
\[
\mathrm{Scal}(g)=  2\delta(\pi_1)+2\delta(\nu_1) +\frac{15}{2} \pi_0^2 + \frac{15}{2} \sigma_0^2 - |\nu_1|^2- \frac{1}{2}|\sigma_2|^2 - \frac{1}{2}|\pi_2|^2 -\frac{1}{2}|\nu_3|^2+4g(\pi_1,\nu_1).
\]
The definitions of the $\mathrm{SU}(3)$ torsion forms $\pi_i,\sigma_i,\nu_i$ are provided in Appendix \ref{sec: SU3-torsion forms - appendix}, and one can check that
\[
|T|^2= \frac{3}{2} \pi_0^2 + \frac{3}{2} \sigma_0^2 +  \frac{5}{3}|\nu_1|^2+ \frac{2}{3}|\pi_1|^2+\frac{1}{2}|\sigma_2|^2 + \frac{1}{2}|\pi_2|^2 +\frac{1}{2}|\nu_3|^2-\frac{4}{3}g(\pi_1,\nu_1).
\]
Hence, we obtain
\begin{align*}
    -\mathrm{Scal}(g)+\lambda |T|^2 
    =&-2\delta(\pi_1+ \nu_1)+ \Big(-\frac{15}{2}+\frac{3}{2}\lambda \Big)(\pi_0^2+\sigma_0^2) + (\frac{1}{2}+\frac{1}{2}\lambda)(|\sigma_2|^2+|\pi_2|^2+|\nu_3|^2)\\
    &+ (1+\frac{5}{3}\lambda)|\nu_1|^2+ \frac{2}{3}\lambda |\pi_1|^2  + (-4-\frac{4}{3}\lambda)g(\pi_1,\nu_1).
\end{align*}
In this case, since the torsion terms $\nu_1$ and $\pi_1$ lie in isomorphic $\mathrm{SU}(3)$ modules, we need to be a little more careful in determining the range for $\lambda$. First we rewrite
\begin{align*}
    (1+\frac{5}{3}\lambda)|\nu_1|^2+
\frac{2}{3}\lambda |\pi_1|^2 
+ (-4-\frac{4}{3}\lambda)g(\pi_1,\nu_1)
&= |c_1\pi_1+c_2\nu_1|^2 + (\frac{2}{3}\lambda-c_1^2)|\pi_1|^2+(1+\frac{5}{3}\lambda-c_2^2)|\nu_1|^2,
\end{align*}
where $c_1^2c_2^2=4(3+\lambda)^2/9$.
Some algebra now shows that, for $\lambda>5$, one can choose $c_i$ so that the critical points are always torsion-free, since all the terms in $-\mathrm{Scal}(g)+\lambda |T|^2$ will be strictly positive up to a divergence term. 
Similarly, writing
\begin{align*}
    (1+\frac{5}{3}\lambda)|\nu_1|^2+
\frac{2}{3}\lambda |\pi_1|^2 
+ (-4-\frac{4}{3}\lambda)g(\pi_1,\nu_1)
&= -|c_1\pi_1+c_2\nu_1|^2 + (\frac{2}{3}\lambda+c_1^2)|\pi_1|^2+(1+\frac{5}{3}\lambda+c_2^2)|\nu_1|^2,
\end{align*}
where again $c_1^2c_2^2=4(3+\lambda)^2/9$, a computation shows that if $\lambda<(3-\sqrt{33})/2\approx-1.37$, then again the critical points are torsion-free, although now the terms are strictly negative.
\end{enumerate}
Thus, in the three cases discussed above, adding suitable lower-order terms to the Ricci-harmonic flow \eqref{equ: Ricci-harmonic} ensures that stationary points are torsion-free. A similar result can also be shown for the $\SU(2)$ case using the formulae in \cite{udhav2024}.

\subsubsection{Comparing the Ricci-harmonic and negative gradient flow}
\label{sec: comparing Ricci-harmonic and -ve gradient}

In order to compare the (modified) Ricci-harmonic flow and the negative gradient flow, we need to rewrite \eqref{equ: general gradient flow}. To this end, we first show:
\begin{proposition}
\label{prop: comparing negative gradient flows}
    For $\H=\{1\}$, $\mathrm{SU}(2)$, $\mathrm{G}_2$ or $\mathrm{Spin}(7)$, the negative gradient flow \eqref{equ: general gradient flow} can be expressed as:
\begin{equation}
    \frac{\partial}{\partial t} \xi = \Big(\mu\big(-2c_{\H}\mathrm{Ric}(g)-\mathcal{L}_{VT}g+\hat{Q}\big)+\Div T\Big)\diamond \xi,\label{equ: negative grad flow most cases}
\end{equation}
where $c_{\H}>0$ is the constant listed in Table~\ref{table_1}, $VT_a:=T_{k,ka}$ is a vector field and $$\hat{Q}_{bd}:=-Q_{abad}-Q_{adab}+(T\star T)_{bd}-\frac{1}{2}|T|^2g_{bd}\in \Sigma^2$$ is a symmetric lower-order term which is quadratic in torsion. More generally, \eqref{equ: negative grad flow most cases} holds whenever the $4$-tensor $\Xi_{ijkl}$ in \eqref{equ: projection map pi2m} is totally skew-symmetric. 

On the other hand, for $\H=\mathrm{U}(m)$ or $\mathrm{SU}(m)$, with $n=2m$, the negative gradient flow \eqref{equ: general gradient flow} can be expressed as:
\begin{equation}
    \frac{\partial}{\partial t} \xi = \Big(\mu\big(-\mathrm{Ric}(g)+\frac{1}{2}\mathrm{sym}(\mathrm{Ric}^*)-\mathcal{L}_{VT}g+\hat{Q}\big)+\Div T\Big)\diamond \xi,\label{equ: negative grad flow U(m)}
\end{equation}
and 
    \begin{equation}
        \frac{\partial}{\partial t} \xi = \Big(\mu\big(-\mathrm{Ric}(g)+\frac{m-2}{2m}\mathrm{sym}(\mathrm{Ric}^*)-\mathcal{L}_{VT}g+\hat{Q}\big)+\Div T\Big)\diamond \xi.\label{equ: negative grad flow SU(m)}
    \end{equation}
respectively, where we recall $\mathrm{Ric}^*_{ab}:=R_{iajk}\omega_{ik}\omega_{bj}$.
\end{proposition}
\begin{proof}
Consider the contraction of the indices $a$ and $c$ in \eqref{equ: general bianchi identity}:
\begin{align}
    \nabla_aT_{b,ad} = \nabla_bT_{a,ad} + \pi^2_{\fm}(R_{ba})_{ad}+ Q_{abad},
\end{align}
and symmetrise on the remaining indices $b$ and $d$ to get
\begin{align}\label{eq: sym_div_Tt}
    \mathrm{sym}(\Div T^t)_{bd} = 
    (\mathcal{L}_{VT} g)_{bd} +
    \pi^2_{\fm}(R_{ba})_{ad}+\pi^2_{\fm}(R_{da})_{ab}+ 
    Q_{abad}+Q_{adab}.
\end{align}
Using expression \eqref{equ: projection map pi2m} for the projection $\pi^2_{\fm}$, we obtain
\begin{align*}
    \pi^2_{\fm}({R}_{ba})_{ad}+\pi^2_{\fm}({R}_{da})_{ab} 
    &= c_{\H} ({R}_{baad} + R_{daab})+ \tilde{c}_{\H}({R}_{baij}\Xi_{ijad}+{R}_{daij}\Xi_{ijab})\\
    &= 2c_{\H} \mathrm{Ric}_{bd} + \tilde{c}_{\H}({R}_{baij}\Xi_{ijad}+{R}_{daij}\Xi_{ijab}).
\end{align*}
The key observation is that if $\Xi_{ijkl}$ is a $4$-form then, by the (algebraic) Bianchi identity $$R_{ijkl}+R_{iklj}+R_{iljk}=0,$$ 
the last term involving $\tilde{c}_{\H}$ vanishes. In particular, we check in Table \ref{table_1}  that this applies to $\H=\{1\}$, $\mathrm{SU}(2)$, $\mathrm{G}_2$ or $\mathrm{Spin}(7)$. 

On the other hand, for the $\mathrm{U}(m)$ case,
\begin{align*}
R_{baij}\Xi_{ijad}+R_{daij}\Xi_{ijab} &= 
R_{baij}\omega_{ia}\omega_{jd}+R_{daij}\omega_{ia}\omega_{jb}\\
&=\mathrm{Ric}^*_{bd}+\mathrm{Ric}^*_{db}.
\end{align*}
For the $\mathrm{SU}(m)$ case, 
\begin{align*}
R_{baij}\Xi_{ijad}+R_{daij}\Xi_{ijab} 
&=\mathrm{Ric}^*_{bd}+\mathrm{Ric}^*_{db} - \frac{1}{m}R_{baij}\omega_{ij}\omega_{ad}- \frac{1}{m}R_{daij}\omega_{ij}\omega_{ab}\\
&=\mathrm{Ric}^*_{bd}+\mathrm{Ric}^*_{db}-\frac{2}{m}\mathrm{Ric}^*_{bd}-\frac{2}{m}\mathrm{Ric}^*_{db}\\
&= (1-\frac{2}{m})(\mathrm{Ric}^*_{bd}+\mathrm{Ric}^*_{db}),
\end{align*}
where we again used the (algebraic) Bianchi identity for the second equality. 
\end{proof}
\begin{remark}
    Unlike  the $\mathrm{SU}(2)$, $\mathrm{G}_2$ and $\mathrm{Spin}(7)$ cases, the corresponding curvature term for $\mathrm{U}(m)$ and $\mathrm{SU}(m)$, $m\geq 3$, does not coincide with the Ricci tensor. Indeed, a component of the Weyl curvature tensor appears in the negative gradient flows of both $\mathrm{U}(m)$- and $\mathrm{SU}(m)$-structures, cf. Remark \ref{remaark: riccistar and weyl}. Note that this additional curvature term $\mathrm{Ric}^*$ has both symmetric and skew-symmetric components. 
    It is worth pointing out that if the underlying $\mathrm{U}(m)$ or $\mathrm{SU}(m)$-structure is K\"ahler, then $\nabla \omega=0$ and we have the simplification:
    \begin{align*}   \mathrm{Ric}^*_{ab}=R_{iajk}\omega_{ik}\omega_{bj}= R(e_i,e_a ,Je_b ,Je_i )
        =\mathrm{Ric}_{ab},
    \end{align*}
where we used that the K\"ahler curvature endomorphism is $J$-invariant. 
In general, Proposition \ref{prop: comparing negative gradient flows} shows that one cannot expect to have a uniform theory for the negative gradient flow of $\mathcal{E}(\xi)$ for arbitrary $\H$-structures, unlike the harmonic case investigated in \cite{Fadel2022}. The strategy for proving short-time existence of \eqref{equ: negative grad flow most cases} involves setting $\mu=(2c_{\H})^{-1}$ and modifying the flow using the DeTurck vector field coupled with the vector field $VT$ to break the diffeomorphism invariance as in the Ricci flow case, see \cites{Dwivedi2025gradient, Dwivedi2023, udhav2024}. The $\mathrm{Ric}^*$ term in \eqref{equ: negative grad flow SU(m)} requires  a more involved argument, see \S \ref{sec: ste for sun}.
\end{remark}

We shall now use Proposition \ref{prop: general evolution of torsion} to compare how $T$ evolves under the Ricci-harmonic flow \eqref{equ: Ricci-harmonic} and the negative gradient flow \eqref{equ: general gradient flow}. To this end, we begin by deriving explicit expressions for the terms occurring in \eqref{equ: general evolution of torsion}.

\begin{lemma}
\label{lemma pim divT}
    In the terms of \eqref{equ: general evolution of torsion}, the following holds:
    \begin{equation}
        \pi^2_{\fm}(\nabla_l\Div T)_{ij} = \Delta T_{l,ij} + \nabla_k\pi^2_{\fm}(R_{kl})_{ij} + \nabla_k Q_{lkij} + (R_{kl}\diamond T_k)_{ij}-\pi^2_{\fh}(T_l\diamond \Div T)_{ij},
    \end{equation}
    where $\Delta:=\nabla_k\nabla_k$.
\end{lemma}
\begin{proof}
    The result follows by setting  $\hat{C}=\Div T$ in Lemma \ref{lem: projection of nabla C} and simplifying:
    \begin{align*}
        \nabla_l \Div T_{ij} = \nabla_l\nabla_k T_{k,ij} &= \nabla_k\nabla_l T_{k,ij}+(R_{kl}\diamond T_k)_{ij}\\
        &=\nabla_k \nabla_kT_{l,ij}+\nabla_k(Q_{lkij}+\pi^2_{\fm}(R_{kl})_{ij})+(R_{kl}\diamond T_k)_{ij},
    \end{align*}
    where we used the Ricci identity for the second equality and the (torsion) Bianchi identity \eqref{equ: general bianchi identity} for the third one.
\end{proof}

\begin{theorem}\label{thm: evolution of torsion for div T}
  If $C=\Div T$, then 
    \begin{equation}
        \frac{\partial}{\partial t} T_{l,ij} = \Delta T_{l,ij} - \pi^2_{\fm}(\Lambda\nabla \Ric_l)_{ij} - \pi^2_{\fm}(\Lambda\nabla S_l)_{ij}+(S\diamond T_l)_{ij} + R*T+\nabla T*T+T*T*T.
    \end{equation}
    In particular, for an $\H$-flow \eqref{equ: GF} with $C=\Div T$, the torsion tensor $T$ evolves by a heat-type equation precisely if $S=-\Ric$ (note we cannot scale this by a constant factor).
\end{theorem}
\begin{proof}
    The proof follows almost immediately by combining the expressions in Proposition \ref{prop: general evolution of torsion} and Lemma \ref{lemma pim divT}. The only term that requires explaining is the one involving the Ricci curvature. This arises as follows:
    \begin{align*}
        \nabla_k\pi^2_{\fm}(R_{kl})_{ij}  &= c_{\H} \nabla_kR_{klij}+\tilde{c}_{\H}\nabla_kR_{klab}\Xi_{abij}+\tilde{c}_{\H}R_{klab}\nabla_k\Xi_{abij}\\
        &= c_{\H}(-\nabla_j R_{klki}-\nabla_i R_{kljk})+\tilde{c}_{\H}(-\nabla_b R_{klka}-\nabla_a R_{klbk})\Xi_{abij}+R*T\\
        &= c_{\H}(\nabla_j\Ric_{li}-\nabla_i\Ric_{lj})+\tilde{c}_{\H}(\nabla_b\Ric_{la}-\nabla_a\Ric_{lb})\Xi_{abij} + R*T\\
        &= -\pi^2_{\fm}(\Lambda \nabla \Ric_l )_{ij}+R*T,
    \end{align*}
    where we used the (differential) Bianchi identity for the second equality.
\end{proof}

\begin{lemma}\label{lem: evolution for lie derivative term}
When $S=\mathcal{L}_{VT}g$, we have
 \begin{align*}
        \pi_{\mathfrak{m}}(\Lambda\nabla S_l)_{ij} 
        &= \nabla_l(\pi^2_{\mathfrak{m}}(d VT)_{ij}) + ((c_{\H}R_{ji}+\tilde{c}_{\H}R_{ba}\Xi_{abij}) \diamond VT)_{l} - \tilde{c}_{\H} (dVT)_{ab}\nabla_l \Xi_{abij} \\
         &\quad+c_{\H}(
         (R_{li}\diamond VT)_j - (R_{lj}\diamond VT)_i )+\tilde{c}_{\H}( (R_{la}\diamond VT)_b-(R_{lb}\diamond VT)_a) \Xi_{abij}   , 
    \end{align*}
    where we recall $VT_i:=T_{k,ki}$. More compactly, we can rewrite the above as 
 \begin{align*}
        \pi^2_{\mathfrak{m}}(\Lambda\nabla S_l)_{ij} 
        &= \nabla_l(\pi_{\mathfrak{m}}^2(d VT)_{ij}) + R*T+\nabla T *T . 
    \end{align*}
\end{lemma}
\begin{proof}
With $S=\mathcal L_{VT}g$, that is, $S_{ij}=\nabla_iVT_j+\nabla_jVT_i =\nabla_iT_{k,kj}+\nabla_jT_{k,ki}$, we compute and use $\pi^2_{\mathfrak m}$ in the first displayed line:
\begin{align*}
     \pi_{\mathfrak{m}}(\Lambda\nabla S_l)_{ij} = \ &c_{\H}(\Lambda\nabla S_l)_{ij} + \tilde{c}_{\H}(\Lambda\nabla S_l)_{ab}\Xi_{abij}\\
     =\ &c_{\H}(\nabla_i\nabla_jT_{k,kl}+\nabla_i\nabla_lT_{k,kj}-\nabla_j\nabla_i T_{k,kl}-\nabla_j\nabla_l T_{k,ki})\\ &+\tilde{c}_{\H}(\nabla_a\nabla_bT_{k,kl}+\nabla_a\nabla_lT_{k,kb}-\nabla_b\nabla_a T_{k,kl}-\nabla_b\nabla_l T_{k,ka})\Xi_{abij}\\
     =\ &(c_{\H}(\nabla_i\nabla_j-\nabla_j\nabla_i)T_{k,kl}+\tilde{c}_{\H}(\nabla_a\nabla_b-\nabla_b\nabla_a)T_{k,kl}\Xi_{abij})+c_{\H}(\nabla_i \nabla_l T_{k,kj}-\nabla_j \nabla_l T_{k,ki})\\ &+\tilde{c}_{\H}(\nabla_a \nabla_l T_{k,kb}-\nabla_b\nabla_l T_{k,ka})\Xi_{abij}\\
     =\ &c_{\H} (R_{ji}\diamond VT)_l+\tilde{c}_{\H}(R_{ba}\diamond VT)_l\Xi_{abij}+c_{\H}(\nabla_l \nabla_i VT_j-\nabla_l \nabla_j VT_{i}+(R_{li}\diamond VT)_j-(R_{lj}\diamond VT)_i) \\ &+\tilde{c}_{\H}(\nabla_l \nabla_a VT_{b}-\nabla_l\nabla_b VT_{a}+(R_{la}\diamond VT)_b-(R_{lb}\diamond VT)_a)\Xi_{abij}\\
     =\ &\nabla_l \pi^2_{\fm}(dVT)_{ij} - \tilde{c}_{\H}(dVT)_{ab}\nabla_l \Xi_{abij}+ ((c_{\H} R_{ji}+\tilde{c}_{\H} R_{ba}\Xi_{abij})\diamond VT)_l\\ &+c_{\H}((R_{li}\diamond VT)_j-(R_{lj}\diamond VT)_i)
     +\tilde{c}_{\H}((R_{la}\diamond VT)_b-(R_{lb}\diamond VT)_a)\Xi_{abij},
\end{align*}
where in the fourth equality we used the Ricci identity.
\end{proof}

\begin{corollary}\label{cor: main result of section 3}
    For the Ricci-harmonic $\H$-flow \eqref{equ: Ricci-harmonic}, and after adding arbitrary lower-order terms quadratic in $T$, the torsion $T$ evolves by the heat-type equation:
    \begin{equation}
        \frac{\partial}{\partial t} T_{l,ij} = \Delta T_{l,ij}  + R*T+\nabla T*T+T*T*T.
    \end{equation}
    On the other hand, for the negative gradient $\H$-flow of the form \eqref{equ: negative grad flow most cases}, we have
    \begin{equation}
        \frac{\partial}{\partial t} T_{l,ij} = \Delta T_{l,ij}  +
        (2\mu c_{\H}-1) \pi^2_{\fm}(\Lambda \nabla \mathrm{Ric}_l)_{ij} + \mu \nabla_l\pi^2_{\fm}(dVT)_{ij} 
        + R*T+\nabla T*T+T*T*T. \label{equ: rr1}
    \end{equation}
    In the $\mathrm{U}(m)$ case \eqref{equ: negative grad flow U(m)} and $\mathrm{SU}(m)$ case \eqref{equ: negative grad flow SU(m)}, we have
    \begin{equation}
        \frac{\partial}{\partial t} T_{l,ij} = \Delta T_{l,ij}  +
        (\mu-1) \pi^2_{\fm}(\Lambda \nabla \mathrm{Ric}_l)_{ij}  + \mu\lambda\pi^2_{\fm}(\Lambda \nabla \mathrm{sym}(\mathrm{Ric}^*)) + \mu \nabla_l\pi^2_{\fm}(dVT)_{ij}
        + R*T+\nabla T*T+T*T*T,
    \end{equation}
    where $\lambda=\frac{1}{2}$ and $\frac{m-2}{2m}$, respectively.
\end{corollary}
\begin{proof}
    For the Ricci-harmonic $\H$-flow, the result follows directly from Theorem \ref{thm: evolution of torsion for div T} by setting $S=-\mathrm{Ric}$. For the negative gradient flow, the result follows from also using Proposition \ref{prop: comparing negative gradient flows} and Lemma \ref{lem: evolution for lie derivative term}.
\end{proof}
\begin{remark}\label{rem: choice of mu}
    Ideally one would like $T$ to evolve by a heat-type flow, for instance, in order to control its higher derivatives (Shi-type estimates). When 
    $\Xi$ is totally skew-symmetric, from \eqref{equ: rr1} there are 2 bad terms: $\pi^2_{\fm}(\Lambda \nabla \mathrm{Ric}_l)$ and $\pi^2_{\fm}(dVT)$.
    Corollary \ref{cor: main result of section 3} suggests that one should take $\mu=(2c_{\H})^{-1}$ to eliminate the first bad term. Unfortunately, the second bad term remains. Furthermore, for $\mathrm{U}(m)$ and $\mathrm{SU}(m)$, there is yet another bad term coming from $\mathrm{sym}(\mathrm{Ric}^*)$. 
    
    An alternative way to deal with the bad term $VT$ could be to consider an initial condition which is preserved as long as the flow exists. This can be illustrated in the case $\H=\rG_2$ by the Laplacian flow of closed $\rG_2$-structures, where $VT=0$ \cites{Bryant2011,lotay-wei-gafa}, or the flow for large-volume heterotic $\rG_2$-systems, where $VT=d\phi$ for some $\phi\in C^\infty(M)$ \cite{Garcia-Fernandez2025}.
\end{remark}

%%%%%%%%%%%%%%%%%%%%%%%%%%%%%%%%%%%%%%%%%%%%%

\section{Short-time existence and uniqueness}

We now proceed to the short-time theory for the non-isometric flows considered above.  In contrast to the (isometric) harmonic $\H$-flow, which is a heat flow for sections of the fixed bundle $\Fr(M,g)/\H\to M$, the flows studied here generally evolve the metric induced by the $\H$-structure itself.  Their leading-order terms therefore combine the Ricci tensor of the evolving metric with torsion terms coming from the first variation of the Dirichlet energy.  As usual for geometric flows with diffeomorphism invariance, the resulting operators are not strictly parabolic unless a gauge is imposed.

The purpose of this section is to select a DeTurck-type gauge for these flows and to establish short-time existence and uniqueness in the cases covered by Table \ref{table_1}.  The argument is organised as follows.  We first isolate the universal principal symbols of the quantities that appear in the Ricci-harmonic $\H$-flow and in the negative gradient $\H$-flow.  These symbols are computed under a general variation
\begin{equation*}
    \frac{\partial}{\partial t}\xi=A\diamond \xi,\qquad A=S+C\in \Gamma(\End(TM))=\Sigma^2(M)\oplus \Omega^2(M),
\end{equation*}
where the skew-symmetric component is understood modulo the stabiliser algebra, as in the general framework of \S \ref{sec: the general setup}.  The curvature symbols follow from Lemma \ref{lemma: evolution curvature}, while the torsion symbols follow from Corollary \ref{cor: evolution of derivatives of torsion}, together with the projection formula \eqref{equ: projection map pi2m}.

Once these symbols are available, the parabolicity estimates are case-dependent. The first family consists of those cases for which the tensor $\Xi$ entering \eqref{equ: projection map pi2m} is a $4$-form, namely, $\H=\{1\},\SU(2),\mathrm{G}_2$ and $\mathrm{Spin}(7)$ in Table \ref{table_1}.  In these cases the DeTurck correction is obtained by combining the usual DeTurck vector field with the vector field $VT$, and the resulting modified systems are strictly parabolic for the relevant choices of the weights.  The remaining almost Hermitian cases require additional symbols involving the $J$-linear components of torsion.  These are treated separately below, under the convention \eqref{eq: convention J eta covector}.

We denote throughout by $D\mathcal P$ the linearisation of an operator $\mathcal P$ along a general variation $A\diamond\xi$, and by
\[
    \sigma(D\mathcal P)(x,\chi)
\]
its principal symbol at a covector $0\neq\chi\in T_x^*M$.  We use the letter $\chi$ for the symbol covector in order to avoid conflict with the defining tensor $\xi$ of the $\H$-structure.

\begin{proposition}
\label{prop: principal symbols} 
We record the principal symbols of the relevant quantities at a point $(x,\chi)\in T^*M$.
Under the general variation \eqref{equ: GF}, we have:
    \begin{enumerate}[label=(\arabic*)]
        \item $\sigma(DX_{DT})(x,\chi)(A)_{j}=\chi_kS_{kj}-\frac{1}{2}\chi_jS_{kk}.$
        \item $\sigma(D\mathrm{Ric})(x,\chi)(A)_{ij}=-\Big(|\chi|^2S_{ij}-\chi_i\chi_kS_{kj}-\chi_j\chi_kS_{ki}+\chi_i\chi_jS_{kk}\Big).$
        \item $\sigma(DT)(x,\chi)(A)_{l,ij}=\chi_l C_{ij}-c_{\H}(\chi_iS_{jl}-\chi_jS_{il})-\tilde{c}_{\H}(\chi_aS_{bl}-\chi_bS_{al})\Xi_{abij}.$
        \item $\sigma(D(VT))(x,\chi)(A)_{j}=\chi_i C_{ij}-c_{\H}(\chi_iS_{ji}-\chi_jS_{ii})-\tilde{c}_{\H}(\chi_aS_{bi}-\chi_bS_{ai})\Xi_{abij}.$
        \item $\sigma(D(\Div T))(x,\chi)(A)_{ij}=|\chi|^2 C_{ij}-c_{\H}(\chi_k\chi_iS_{jk}-\chi_k\chi_jS_{ik})-\tilde{c}_{\H}(\chi_k\chi_aS_{bk}-\chi_k\chi_bS_{ak})\Xi_{abij}.$
    \end{enumerate}
\end{proposition}
\begin{proof}
    Items $(1)$ and $(2)$ are well-known, cf. \cites{chow2011, topping2006}. Item $(3)$ follows easily from Corollary \ref{cor: evolution of derivatives of torsion}, with $k=0$, using the expression \eqref{equ: projection map pi2m} for $\pi^2_{\mathfrak{m}}$. Finally, $(4)$ and $(5)$ are obtained by suitable contractions of $(3)$. 
\end{proof}

\begin{theorem}[Short-time existence for the Ricci-harmonic $\H$-flow]
\label{thm: STE Ricci harmonic H}
Let $\H\subset\SO(n)$ be a closed subgroup in the tensorial setting of
\S\ref{sec: the general setup}, and let $\xi_0$ be a smooth $\H$-structure
on a closed manifold $M$.  Consider the flow
\begin{equation}
\label{eq: Ricci harmonic H with lower-order terms}
    \frac{\partial}{\partial t}\xi
    =\big(-\Ric(g_\xi)+\Div T_\xi+\mathcal Q(\xi)\big)\diamond\xi,
\end{equation}
where $\mathcal Q(\xi)\in\Gamma(\Sigma^2\oplus\Lambda^2_{\mathfrak m})$ is any
natural term whose linearisation has zero second-order principal symbol.  Then
there exist $\varepsilon>0$ and a unique smooth solution of
\eqref{eq: Ricci harmonic H with lower-order terms}, with initial condition
$\xi(0)=\xi_0$, defined for $t\in[0,\varepsilon)$.  In particular, this applies
to the Ricci-harmonic flow \eqref{equ: Ricci-harmonic} and to the modified
Ricci-harmonic flow \eqref{eq: modified Ricci-harmonic H flow}.
\end{theorem}

\begin{proof}
We only need to discuss the highest-order terms, since $\mathcal Q$ has zero
second-order principal symbol.  Let $X_{DT}$ be the DeTurck vector field from
Proposition~\ref{prop: principal symbols}.  We modify the flow by adding the
Lie derivative  $\mathcal L_{2X_{DT}}\xi$.  Up to first-order terms in the
operator, this replaces the tensor defining the flow by
\begin{equation*}
    -\Ric+\Div T+\mathcal L_{X_{DT}}g+dX_{DT}.
\end{equation*}
Indeed, for a tensor defining an $\H$-structure one has, at the level of
principal symbols,
\begin{equation*}
    \mathcal L_{2X_{DT}}\xi
    =\big(\mathcal L_{X_{DT}}g+dX_{DT}\big)\diamond\xi+\mathrm{l.o.t.}
\end{equation*}
The Ricci--DeTurck part has the usual scalar symbol
\begin{equation*}
    \sigma\big(D(-\Ric+\mathcal L_{X_{DT}}g)\big)(x,\chi)(A)_{ij}
    =|\chi|^2S_{ij}.
\end{equation*}
For the skew-symmetric component, Proposition~\ref{prop: principal symbols}
gives
\begin{equation*}
    \sigma(D(\Div T))(x,\chi)(A)
    =|\chi|^2C-\pi^2_{\mathfrak m}\big(\chi\wedge S\chi\big),
\end{equation*}
while
\begin{equation*}
    \sigma(D(dX_{DT}))(x,\chi)(A)=\chi\wedge S\chi.
\end{equation*}
Since only the $\Lambda^2_{\mathfrak m}$-projection of a two-form acts on
$\xi$, the $S$-dependent terms cancel after  the diamond action. Hence
the modified operator has principal symbol
\begin{equation*}
    \sigma(D\mathcal P)(x,\chi)(A)=|\chi|^2A,
    \qquad A=S+C\in\Sigma^2\oplus\Lambda^2_{\mathfrak m},
\end{equation*}
and therefore the DeTurck-modified system is strictly parabolic.  Standard quasilinear
parabolic theory gives a unique smooth short-time solution to the modified system.  Pulling back by the flow generated by $-2X_{DT}$ gives a solution
of \eqref{eq: Ricci harmonic H with lower-order terms}.  Conversely, any
solution of the unmodified flow is transformed by the corresponding DeTurck
diffeomorphisms into a solution of the modified strictly parabolic system, and
uniqueness follows from uniqueness for the modified system.
\end{proof}

\subsection{The skew-symmetric 
\texorpdfstring{$\Xi$}{Xi} cases: 
\texorpdfstring{$\{1\}$}{1}, 
\texorpdfstring{$\mathrm{SU}(2)$}{SU(2)}, 
\texorpdfstring{$\mathrm{G}_2$}{G2}, \texorpdfstring{$\mathrm{Spin}(7)$}{Spin(7)}.}
The main result of this subsection is the following short-time existence theorem.
\begin{theorem}
\label{thm: STE skew case}
    Assume that the projection $\pi^2_{\fm}$ is of the form \eqref{equ: projection map pi2m}, with $\Xi$ a $4$-form and $c_{\H}\in(0,1]$. Given an $\H$-structure defined by a smooth tensor $\xi(0)$ on a closed manifold $M$, there exists $\varepsilon>0$ such that the negative gradient $\H$-flow \eqref{equ: general gradient flow}, with weight $\mu=(2c_{\H})^{-1}$, admits a unique smooth solution for $t\in[0,\varepsilon)$. In particular, by Table \ref{table_1}, this includes $\H=\{1\}$, $\mathrm{SU}(2)$, $\mathrm{G}_2$ and $\mathrm{Spin}(7)$. The Ricci-harmonic $\H$-flow is covered, without any assumption on $\Xi$, by Theorem~\ref{thm: STE Ricci harmonic H}.
\end{theorem}
Recall that, when $\Xi$ is a $4$-form, we can express \eqref{equ: general gradient flow} equivalently as \eqref{equ: negative grad flow most cases}. Furthermore, in view of Remark \ref{rem: choice of mu}, we set $\mu=(2c_{\H})^{-1}$.
To highest order, the negative gradient $\H$-flow is then given by
\begin{equation}
    \frac{\partial}{\partial t} \xi = \Big(-\mathrm{Ric}(g)-\frac{1}{2c_{\H}}\mathcal{L}_{VT}g+\Div T\Big)\diamond \xi .\label{equ: unmodified flow}
\end{equation}
Throughout this section we shall omit all lower-order terms since these will not play a role in the subsequent proof. 
In order to break the diffeomorphism invariance of the flow, we define the vector field
\begin{equation*}
    X:=2X_{DT}+\frac{1}{c_{\H}}VT,
\end{equation*}
and consider the modified flow:
\begin{align}
    \frac{\partial}{\partial t} \xi &= \Big(-\mathrm{Ric}(g)-\frac{1}{2c_{\H}}\mathcal{L}_{VT}g+\Div T\Big)\diamond \xi  + \mathcal{L}_X \xi \nonumber\\
    &=\Big(-\mathrm{Ric}(g)-\frac{1}{2c_{\H}}\mathcal{L}_{VT}g+\Div T +\frac{1}{2}\mathcal{L}_X g+\frac{1}{2} dX \Big)\diamond \xi  \nonumber\\
    &=\Big(-\mathrm{Ric}(g)+\mathcal{L}_{X_{DT}}g+\Div T + dX_{DT} +\frac{1}{2c_{\H}} dVT  \Big)\diamond \xi =:\mathcal{P}(\xi)\diamond\xi,\label{equ: modified flow}
\end{align}
where in the second equality we used the identity:
\begin{equation*}
    \mathcal{L}_Y \alpha = \nabla_Y \alpha + \frac{1}{2}( \mathcal{L}_Yg + dY^\flat) \diamond \alpha,
\end{equation*}
for a general differential form $\alpha$ and vector field $Y$, and we omitted the lower-order term $\nabla_Y \alpha$.

Note that since $\Xi$ is a $4$-form and $S$ is symmetric, Proposition \ref{prop: principal symbols}--(4) simplifies to:
\[
\sigma(D(VT))(x,\chi)(A)_{j}=\chi_k C_{kj}-c_{\H}(\chi_kS_{jk}-\chi_jS_{kk}).
\]
Using Proposition \ref{prop: principal symbols}, we now compute
\begin{align*}
    \sigma(D\mathcal{P})(x,\chi)(A)=\ &\Big( |\chi|^2S_{ij}+|\chi|^2C_{ij}-c_{\H}(\chi_k\chi_iS_{jk}-\chi_k\chi_jS_{ik})-\tilde{c}_{\H}(\chi_k\chi_aS_{bk}-\chi_k\chi_bS_{ak})\Xi_{abij}\\
    &+(\chi_i\chi_kS_{kj}-\chi_j\chi_kS_{ki})+\frac{1}{2c_{\H}}(\chi_i\chi_k C_{kj}-\chi_j\chi_k C_{ki}-c_{\H}(\chi_i\chi_kS_{jk}-\chi_j\chi_kS_{ik}))\Big)\diamond \xi,\\
    =\ &\Big( |\chi|^2S_{ij}+|\chi|^2C_{ij}+\frac{1}{2c_{\H}}(\chi_{i}\chi_kC_{kj}-\chi_{j}\chi_kC_{ki})-\frac{1}{2}(\chi_i\chi_kS_{kj}-\chi_j\chi_kS_{ki})\Big)\diamond \xi,
\end{align*}
where we used the fact that for any $2$-form $\alpha$, one has $\alpha \diamond \xi=\pi^2_{\mathfrak{m}}(\alpha)\diamond \xi$ and applied it to the term $\alpha=\sigma(D(\Lambda\nabla S_l))$.
Thus, to prove parabolicity, we need to show that there exists $\lambda>0$ such that
\[LHS:=
\Big( |\chi|^2S_{ij}+|\chi|^2C_{ij}+\frac{1}{2c_{\H}}(\chi_{i}\chi_kC_{kj}-\chi_{j}\chi_kC_{ki})-\frac{1}{2}(\chi_i\chi_kS_{kj}-\chi_j\chi_kS_{ki})\Big)  A_{ij}\geq \lambda|\chi|^2|A|^2.
\]
First observe that the third term in $LHS$ can be expressed as
\begin{align*}
    \frac{1}{2c_{\H}}(\chi_{i}\chi_kC_{kj}-\chi_{j}\chi_kC_{ki})  A_{ij}=\frac{1}{2c_{\H}}(\chi_{i}\chi_kC_{kj}-\chi_{j}\chi_kC_{ki})  C_{ij} &= \frac{1}{c_{\H}}\chi_{i}\chi_kC_{kj}  C_{ij}\\
    &=\frac{1}{c_{\H}}| C(\chi)|^2 \geq0.
\end{align*}
Similarly, for the fourth term,
\begin{align*}
    (\chi_i\chi_kS_{kj}-\chi_j\chi_kS_{ki})  A_{ij}=(\chi_i\chi_kS_{kj}-\chi_j\chi_kS_{ki}) C_{ij} &= 2\chi_{i}\chi_kS_{kj}  C_{ij}\\
    &=2\langle S(\chi),C(\chi) \rangle\\
    &\geq -2|S(\chi)| |C(\chi)|\\
    &\geq-(\epsilon|S(\chi)|^2+\epsilon^{-1}|C(\chi)|^2),
\end{align*}
where $\epsilon>0$ and we used Cauchy-Schwarz and Young's inequalities. Combining all the above,
\begin{align*}
    LHS &\geq |\chi|^2|S|^2+|\chi|^2|C|^2+({c_{\H}}^{-1}-(2\epsilon)^{-1})|C(\chi)|^2-\frac{\epsilon}{2} |S(\chi)|^2\\
    &\geq \Big(1-\frac{\epsilon}{2}\Big)|\chi|^2|S|^2+
    \big(1-|({c_{\H}}^{-1}-(2\epsilon)^{-1})|\big)|\chi|^2|C|^2.
\end{align*}
If we set $c:=1-\frac{\epsilon}{2}$, then the condition for parabolicity becomes:
\begin{equation}
    0<c<1\qquad\text{and}\qquad 1-|c_{\H}^{-1}-(4(1-c))^{-1}|>0. \label{equ: inequalities for parabolicity}
\end{equation}
The region defined by these inequalities is illustrated in Figure \ref{fig:figure1}.
\begin{figure}
\centering
\includegraphics[width=.5\textwidth]{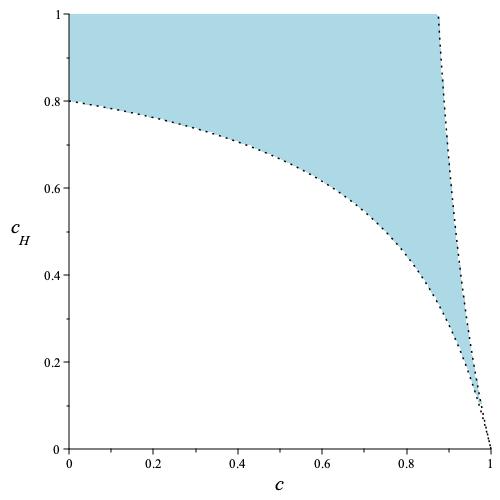}
\caption{{Region defined by \eqref{equ: inequalities for parabolicity}}}
\label{fig:figure1}
\end{figure}
For any $0<c_{\H}\leq 1$, choose $c=1-\frac{c_{\H}}{4}\iff\epsilon=\frac{c_{\H}}{2}$.  Then $c\in(0,1)$ and
\[
    c_{\H}^{-1}-(4(1-c))^{-1}=0.
\]
Thus \eqref{equ: inequalities for parabolicity} holds, and the symbol is bounded below by a positive multiple of $|\chi|^2|A|^2$.  In particular, this applies to the $\{1\}$, $\mathrm{SU}(2)$, $\mathrm{G}_2$ and $\mathrm{Spin}(7)$ cases in Table \ref{table_1}. This shows that the modified negative gradient flow \eqref{equ: modified flow} is strictly parabolic, and the standard short-time existence theorem for quasilinear parabolic systems gives a unique smooth solution $\{\xi(t)\}_{t\in[0,\varepsilon)}$, with $\xi(0)=\xi_0$.

To obtain a solution to the (unmodified) flow \eqref{equ: unmodified flow}, we consider  the diffeomorphisms $\{f_t\}$ defined by
\begin{equation}
    Y_t \circ f_t(p) = \frac{d}{dt} f_t(p),
    \qwithq
    f_0=\mathrm{Id},
    \label{eq: flow generated by vector field}
\end{equation}
where $Y_t:=-2X_{DT}-c_{\H}^{-1}VT$ as above. If $\{\xi(t)\}$ denotes the solution to the modified flow \eqref{equ: modified flow} with $\xi(0)=\xi_0$, then we define the pullback $\hat{\xi}(t):=f_t^*(\xi(t))$.
A simple computation shows that 
\begin{align*}
    \frac{\partial}{\partial t}\hat{\xi}(t) &=f_t^*\big(\frac{\partial}{\partial t}\xi(t)+\mathcal{L}_{Y_t}\xi(t)\big)\\ 
    &= f^*_t\Big(\big(-\mathrm{Ric}(g)-\frac{1}{2c_{\H}}\mathcal{L}_{VT}g+\Div T\big)\diamond \xi \Big)\\
    &=\big(-\mathrm{Ric}(\hat{g})-\frac{1}{2c_{\H}}\mathcal{L}_{\widehat{VT}}\hat{g}+\Div \hat{T}\big)\diamond \hat{\xi} 
\end{align*}
i.e., $\{\hat{\xi}(t)\}$ is a solution to \eqref{equ: unmodified flow}, with $\hat{\xi}(0)=\xi_0$. 
To prove that it is also unique, one can apply the same argument as in the proof of \cite{udhav2024}*{Theorem 5.4} or \cite{Dwivedi2023}*{Theorem 6.76}. The rough strategy is to produce two diffeomorphisms: one associated to the vector field $VT$ and one using the harmonic heat map flow, see \cite{chow2011}*{Lemma 3.27}. Note that the latter diffeomorphism is unique for given initial data since it is a parabolic flow. The composition of these diffeomorphisms pullback solutions of the unmodified flow to solutions of the modified flow, but again since the latter is parabolic the solutions are unique for given smooth initial data. This shows uniqueness of $\{\hat{\xi}(t)\}$, with $\hat{\xi}(0)=\xi_0$. Hence we have recovered in generality the corresponding results in \cites{ Dwivedi2025gradient,Dwivedi2023, udhav2024}.

\begin{remark}\label{rm: short-time_Ricci_DivT}
    The Ricci-harmonic part of the discussion is not restricted to the
    $\Xi$-skew cases: it is precisely Theorem~\ref{thm: STE Ricci harmonic H}.
    The point of Theorem~\ref{thm: STE skew case} is instead that, under the
    additional algebraic hypothesis that $\Xi$ be a $4$-form, the unrestricted
    negative gradient flow can also be brought into a strictly parabolic
    DeTurck gauge.
\end{remark}

\subsection{The \texorpdfstring{$\U(m)$}{} symbol obstruction and the negative gradient flow of \texorpdfstring{$\SU(m)$}{}-structures}
\label{sec: ste for sun}

Let us analyse the remaining almost Hermitian cases from Table \ref{table_1}, namely $\H=\U(m)$ and $\H=\SU(m)$. The symbol computations below are carried out for both groups, since they identify the extra $\Ric^*$ contribution absent from the $\Xi$-skew cases. The DeTurck argument is completed here for the $\SU(m)$ negative gradient flow, while the $\U(m)$ endpoint is shown in Proposition~\ref{prop: U(m) endpoint degeneracy} to carry a genuine principal-symbol obstruction. In these cases, the tensor $\Xi$ is given by
\begin{align}\label{eq: Xi_U(m)_SU(m)}
    \Xi_{ijkl}= J_{ik}J_{jl}-\frac{\lambda_{\H}}{m}J_{ij}J_{kl} \qforq \lambda_{\U(m)}=0, \qandq \lambda_{\SU(m)}=1.
\end{align}
Following \cite{Fadel2022}*{Section 1.2}, a $\fgl(n,\bR)$-deformation of an $\SU(m)$-structure $(g,J,\Upsilon)$ is given by
\begin{align}\label{eq: general SU(m) flow}
    \frac{\partial}{\partial t}J=A\diamond J, \quad \frac{\partial}{\partial t}\Upsilon^\pm=A\diamond \Upsilon^\pm \qandq \frac{\partial}{\partial t}g=S\diamond g=2S,
\end{align}
where $A=S+C\in \Sigma^2(M)\oplus \Omega^2_{\fm}(M)$. Note that under the flow of the metric \eqref{eq: general SU(m) flow}, the inverse metric and Christoffel symbols evolve as in \eqref{eq: standard metric variation}, and
\begin{equation*}
    \frac{\partial}{\partial t}\mathrm{vol}=S_{ii}\mathrm{vol}.
\end{equation*}

First, we need the expression for the principal symbol of $\Ric^*$, as well as specialisations of some of the formulae from Proposition \ref{prop: principal symbols}:
\begin{lemma}
\label{lem: symbols_SU(m)}
    Let $\H$ be either $\U(m)$ or $\SU(m)$, with $n=2m$, and $(x,\chi)\in T^*M$. Then, the
    principal symbols of $\Ric^*$, $\nabla VT$, $\Div T$ and, in the
    $\SU(m)$ case, $\nabla J\eta$ are:
    \begin{align}
    \begin{split}
        \sigma\left(D(\Ric^*)\right)(x,\chi)(A)_{ij}=&
        -(SJ\chi)_i(J\chi)_j+\chi_i(JSJ\chi)_j,\\
        \sigma(D(\nabla VT))(x,\chi)(A)_{ij}=&
        \chi_i(C\chi)_j
        -\frac{1}{2}\chi_i((S\chi)_j-\chi_j S_{kk})
        +\frac{m-2\lambda_{\H}}{2m}\chi_i(JSJ\chi)_j, \\
        \sigma\left(D(\Div T)\right)(x,\chi)(A)_{ij}=&
        |\chi|^2C_{ij}
        -\frac12\left(\chi_i(S\chi)_j-\chi_j(S\chi)_i\right)\\
        &+\frac12\left((J\chi)_i(JS\chi)_j-(J\chi)_j(JS\chi)_i\right)
        -\frac{\lambda_{\H}}{m}\langle J\chi,S\chi\rangle J_{ij},\\
        \sigma(D(\nabla J\eta))(x,\chi)(A)_{ij}=&
        -2^{m-1}\langle C,J\rangle \chi_i(J\chi)_j
        + 2^{m}\chi_i(JSJ\chi)_j.
    \end{split}
    \end{align}
\end{lemma}

\begin{proof}
    Using the expression \eqref{eq:Ricci*} for $\Ric^*$ and the evolutions of $R$ (see Lemma \ref{lemma: evolution curvature}) and $J$, we have
    \begin{align*}
        \partial_t\Ric^*_{ab}=&-\partial_t\left({R_{iaj}}^kJ_k^iJ_b^j\right)\\
        =&-g^{kl}\left(\nabla_i\nabla_aS_{jl}+\nabla_i\nabla_jS_{al}-\nabla_i\nabla_lS_{aj}-\nabla_a\nabla_iS_{jl}-\nabla_a\nabla_jS_{il}+\nabla_a\nabla_lS_{ij}\right)J_k^iJ_b^j\\
        &-{R_{iaj}}^k(C\diamond J)_k^iJ_b^j-{R_{iaj}}^kJ_k^i(C\diamond J)_b^j.
    \end{align*}
Thus, at $(x,\chi)\in T^*M$, we have
    \begin{align*}
        \sigma\left(D(\Ric^*)\right)(x,\chi)(A)_{ab}=&
        -g^{kl}\left(\chi_i\chi_aS_{jl}+\chi_i\chi_jS_{al}
        -\chi_i\chi_lS_{aj}-\chi_a\chi_iS_{jl}
        -\chi_a\chi_jS_{il}+\chi_a\chi_lS_{ij}\right)J_k^iJ_b^j\\
        =&-\chi_a(JSJ\chi)_b-(SJ\chi)_a(J\chi)_b+(J\chi)_k\chi_lg^{kl}(SJ)_{ba}+\chi_a(JSJ\chi)_b\\
        &+\chi_a(J\chi)_b\langle S,J\rangle+\chi_a(JSJ\chi)_b,
    \end{align*}
     where $\langle J,S\rangle=0$ and $\langle J\chi,\chi\rangle=0$. Now, the expression for $\sigma(D(\nabla VT))$ and $\sigma(D(\Div T))$ follow by substituting \eqref{eq: Xi_U(m)_SU(m)} into Proposition \ref{prop: principal symbols}--(4)\&(5), respectively. Finally, since
\begin{align*}
    (J\eta)_l = -2^{m-1}T_{k,ij}J_{ij}J_{lk}
    =2^{m-1}T_{k,ij}J_{ij}J_{kl},
\end{align*}
we have
\[
\sigma(D( J\eta))(x,\chi)(A)_l
= -2^{m-1}J_{ij}J_{lk}\,\sigma(D T)(x,\chi)(A)_{k,ij}.
\]
Now, applying Proposition \ref{prop: principal symbols}--(3) to the tensor $\Xi$ on \eqref{eq: Xi_U(m)_SU(m)} for $\lambda_{\H}=1$, we have
\begin{align*}
     \sigma(D( J\eta))(x,\chi)(A)_l =-2^{m-1}\langle C,J\rangle (J\chi)_l+2^{m}(JSJ\chi)_l,
\end{align*}
and the expression for $\sigma(D(\nabla J\eta))(x,\chi)(A)_{ij}$ follows. 
\end{proof}

Now, we consider the $\U(m)$ or $\SU(m)$ negative gradient flow \eqref{equ: general energy functional}, with $n=2m$, given by 
    \begin{equation}
        \frac{\partial}{\partial t} \xi = \Big(-\mathrm{Ric}(g)+\frac{m-2\lambda_{\H}}{2m}\mathrm{sym}(\mathrm{Ric}^*)-\mathcal{L}_{VT}g+\hat{Q}+\Div T\Big)\diamond \xi=:\mathcal{P}(\xi)\diamond\xi.\label{equ: negative grad flow U(m)_SU(m)}
    \end{equation}
    where $\lambda_{\U(m)}=0$ and  $\lambda_{\SU(m)}=1$. The main result of this section is the following

    \begin{theorem}
    \label{thm: ste SUn}
        Let $\xi_0=(g_0,J_0,\Upsilon_0)$ be an $\SU(m)$-structure on $M^{2m}$.  Then there exist $\epsilon>0$ and a unique smooth family of $\SU(m)$-structures $\{\xi_t=(g_t,J_t,\Upsilon_t)\}_{t\in [0,\epsilon)}$, with $\xi(0)=\xi_0$, solving \eqref{equ: negative grad flow U(m)_SU(m)} for $\lambda_{\H}=1$. 
    \end{theorem}

    \begin{proof}
    Using Proposition \ref{prop: principal symbols} and Lemma \ref{lem: symbols_SU(m)}, the principal symbol of the operator $\cP(\xi)$ from \eqref{equ: negative grad flow U(m)_SU(m)} is
    \begin{align*}
        \begin{split}
            \langle \sigma(D\cP(\xi))(x,\chi)(A),A\rangle=&|A|^2|\chi|^2-|S\chi|^2-\frac{m-2\lambda_{\H}}{m}|SJ\chi|^2
    -3\langle C\chi,S\chi\rangle+\langle CJ\chi,JS\chi\rangle\\
    &-\frac{\lambda_{\H}}{m}\langle J\chi,S\chi\rangle\langle C,J\rangle.
        \end{split}
    \end{align*}
    Writing $C=\frac{\lambda_{\H}}{2m}\langle C,J\rangle J+\bar{C}$, where $\bar{C}J=-J\bar{C}$, we obtain
    \begin{align*}
        \begin{split}
            \langle \sigma(D\cP(\xi))(x,\chi)(A),A\rangle=&|A|^2|\chi|^2-|S\chi|^2-\frac{m-2\lambda_{\H}}{m}|SJ\chi|^2
    -4\langle C\chi,S\chi\rangle-\frac{2\lambda_{\H}}{m}\langle J\chi,S\chi\rangle\langle C,J\rangle.
        \end{split}
    \end{align*}
Now, defining the vector field $X:=\Div S+3\Div C$, 
\begin{align*}
    \langle\sigma\left(\nabla X\right)(x,\chi)(A),A\rangle
    =&\langle A\chi,S\chi+3C\chi\rangle\\
    =& |S\chi|^2+3|C\chi|^2 +4\langle S\chi,\bar{C}\chi\rangle+\frac{2\lambda_{\H}}{m}\langle C,J\rangle\langle S\chi,J\chi\rangle.
\end{align*}

In order to break the diffeomorphism invariance of \eqref{equ: negative grad flow U(m)_SU(m)}, consider the modified flow:
\begin{align}\label{eq: dT_gradient_flow_SU(m)}
\begin{split}
    \frac{\partial}{\partial t} \xi &=\mathcal{Q}(\xi)\diamond\xi=:\mathcal{P}(\xi)\diamond\xi + \mathcal{L}_X \xi=\left(\mathcal{P}(\xi)+\nabla X+X\lrcorner T\right)\diamond\xi,
    \end{split}
\end{align}
whose principal symbol is
\begin{align*}
        \begin{split}
            \langle \sigma(D\cQ(\xi))(x,\chi)(A),A\rangle=&|A|^2|\chi|^2+3|C\chi|^2-\frac{m-2\lambda_{\H}}{m}|SJ\chi|^2.
        \end{split}
    \end{align*}
Hence, for $\lambda_{\H}=1$, 
\begin{align*}
        \begin{split}
            \langle \sigma(D\cQ(\xi))(x,\chi)(A),A\rangle\geq\frac{2}{m}|A|^2|\chi|^2.
        \end{split}
    \end{align*}
    Thus, the DeTurck-modified operator is strictly parabolic. Standard quasilinear parabolic theory gives a unique smooth solution to the modified system on a short-time interval.  Pulling this solution back by the flow generated by the negative of the gauge vector field $X$ gives a solution to the original negative gradient flow.  Uniqueness follows by the usual DeTurck argument, as in the proof of Theorem~\ref{thm: STE skew case}.
       \end{proof}
\begin{proposition}[A principal-symbol obstruction for the $\U(m)$ negative gradient flow]
\label{prop: U(m) endpoint degeneracy}
Let $\H=\U(m)$, so that $\lambda_{\H}=0$ in
\eqref{equ: negative grad flow U(m)_SU(m)}.  For every non-zero covector $\chi$,
the principal quadratic form of the DeTurck-modified operator
\eqref{eq: dT_gradient_flow_SU(m)} has a non-zero null direction that is
transverse to the principal symbol of the infinitesimal diffeomorphism action.
More precisely, there exists a non-zero symmetric variation $S$, with $C=0$, such
that
\begin{equation*}
    \big\langle\sigma(D\mathcal Q)(x,\chi)(S),S\big\rangle=0,
\end{equation*}
provided $S$ is not in the image, at covector $\chi$, of the infinitesimal
diffeomorphism action on $\U(m)$-structures.  Moreover, if $X=X(\xi)$ is any
first-order natural vector field in the $\U(m)$-structure, then the principal
quadratic contribution of the gauge term $\mathcal L_X\xi$ also vanishes on this
same variation.  Consequently, no first-order DeTurck modification can make the
natural unrestricted negative gradient $\U(m)$-flow strictly parabolic along the 
principal quadratic-form.
\end{proposition}

\begin{proof}
By \eqref{eq: dT_gradient_flow_SU(m)}, when $\lambda_{\H}=0$ the standard DeTurck-modified symbol is
\begin{equation*}
    \big\langle \sigma(D\mathcal Q)(x,\chi)(A),A\big\rangle
    =|A|^2|\chi|^2 +3|C\chi|^2-|SJ\chi|^2.
\end{equation*}
  Normalising $|\chi|=1$ and setting $\theta:=J\chi$, we have
$|\theta|=1$ and $\langle\theta,\chi\rangle=0$.  Let
$S=\theta\otimes\theta$ and $C=0$.  Then $S\ne0$,
\begin{equation*}
    S\chi=\theta\,\langle\theta,\chi\rangle=0,
    \qquad
    SJ\chi=S\theta=\theta,
\end{equation*}
and hence $|S|^2=1=|SJ\chi|^2$.  Therefore the displayed quadratic form vanishes
on $A=S$.

We next check that this null direction is not a diffeomorphism direction.  At
the level of metric variations, the symmetric part of the principal symbol of
the infinitesimal diffeomorphism action is of the form $\chi\odot\alpha$, for
some one-form $\alpha$.  If $S=\chi\odot\alpha$, then evaluating both sides on
$(\theta^\sharp,\theta^\sharp)$ gives
\begin{equation*}
    1=S(\theta^\sharp,\theta^\sharp)
    = (\chi\odot\alpha)(\theta^\sharp,\theta^\sharp)=0,
\end{equation*}
because $\chi(\theta^\sharp)=\langle\chi,\theta\rangle=0$.  This contradiction
shows that the null direction constructed above is transverse to the principal
symbol of the diffeomorphism orbit.

Finally, let $X=X(\xi)$ be any first-order natural vector field.  In the
principal symbol of $\mathcal L_X\xi$, only the $\nabla X$ contribution is 
second-order; the $X\lrcorner T$ term is lower-order for the present purpose.  If
$V=\sigma(DX)(x,\chi)(A)$, then the corresponding quadratic contribution is
\begin{equation*}
    \langle \chi\otimes V,A\rangle=\langle V,A\chi\rangle.
\end{equation*}
For the variation above, $A\chi=S\chi=0$.  Therefore the principal contribution
of every first-order diffeomorphism gauge vanishes on this same null direction.
This proves that the loss of parabolicity is not only the usual degeneracy along
diffeomorphism orbits, and cannot be removed by a first-order DeTurck
modification.
\end{proof}

\begin{remark}[A symbolic interpolation]
\label{rem: symbolic lambda interpolation Um SUm}
The estimate in the proof of Theorem~\ref{thm: ste SUn} shows more generally
that a modified symbol of the following form satisfies
\begin{equation*}
    |A|^2|\chi|^2+3|C\chi|^2-\frac{m-2\lambda}{m}|SJ\chi|^2
    \geq
    \min\{1,2\lambda/m\}|A|^2|\chi|^2>0,\quad\forall\lambda>0.
\end{equation*}
In this sense the $\U(m)$ symbol is the
critical endpoint $\lambda=0$, where Proposition~\ref{prop: U(m) endpoint degeneracy}
shows that the estimate becomes sharp and loses strict positivity.  The
intermediate values should be regarded here as a principal-symbol observation.
Equivalently, they can be viewed as the symbol-level shadow of the weighted
$\SU(m)$ torsion energies obtained by assigning weight $\lambda$ to the
$\mathbb R J$-component in the splitting
\begin{equation*}
    \mathfrak{su}(m)^\perp=\mathfrak u(m)^\perp\oplus \mathbb R J.
\end{equation*}
Thus $\lambda=1$ gives the canonical $\SU(m)$ symbol, while $\lambda=0$ gives the
$\U(m)$ symbol and is precisely the endpoint at which the same algebraic
positivity mechanism degenerates.
\end{remark}

%%%%%%%%%%%%%%%%%%%%%%%%%%%%%%%%%%%%%%%%%%%%%%%%%%%%

\section{Modified Ricci-harmonic flow of \texorpdfstring{$\H$}{H}-structures}\label{sec: Ricci-harmonic}

In this section we study the modified Ricci-harmonic flow of $\H$-structures introduced in \S\ref{sec: modifying Ricci-harmonic}.  By Theorem~\ref{thm: STE Ricci harmonic H}, the Ricci-harmonic $\H$-flow is locally well posed.  Combining this with the heat-type torsion equation of Corollary~\ref{cor: main result of section 3}, we derive global derivative estimates for the Riemann curvature tensor $\rR\rmm$ and the intrinsic torsion $T$.  These are the analogues, in the present $\H$-structure setting, of the Shi estimates for Ricci flow.  As an application, we obtain a finite-time continuation criterion.

The argument follows the maximum-principle strategy used by Lotay--Wei for the Laplacian flow of closed $\rG_2$-structures~\cite{lotay-wei-gafa}, and is compatible with Chen's corresponding estimates for $\mathrm G_2$-flows~\cite{Chen2018} and with the Ricci-harmonic estimates of Dwivedi~\cite{dwivedi2026ricci}.  In the special case $\H=\mathrm{Spin}(7)$, the results in this section overlap with Duthie's recent theory of reasonable $\mathrm{Spin}(7)$-flows~\cite{Duthie2026reasonableSpin7}: his hypotheses isolate a broad $\mathrm{Spin}(7)$-specific class for which the same curvature--torsion quantity $$\Lambda=(|\rR\rmm|^2+|\nabla T|^2+|T|^4)^{1/2}$$ 
controls all higher derivatives, and for which he further develops compactness and singularity blow-up results.  Our emphasis is complementary: we fix the modified Ricci-harmonic evolution and use the tensorial $\H$-structure formalism to obtain the Shi-type estimates and continuation criterion uniformly for arbitrary closed $\H\subset\SO(n)$.

\subsection{Shi-type estimates}

Let $\{\xi(t)\}_{t\in [0,\delta)}$ be a time-dependent family of $\H$-structures on $M$, solving the modified Ricci-harmonic $\H$-flow \eqref{eq: modified Ricci-harmonic H flow},  
\[
    \frac{\partial}{\partial t} \xi 
    = \Big(-{\rm{Ric}}(g) +(\lambda-n\hat{\lambda}) (T\star T)+ \hat{\lambda}|T|^2g + {\rm{div}}\ T\Big)\diamond \xi, 
\]
where $\lambda$ and $\hat\lambda$ are fixed constants.  When the stationary-point statement from \S\ref{sec: modifying Ricci-harmonic} is invoked, $\lambda$ is chosen in one of the ranges specified there; for the estimates below, only the schematic fact that the additional metric terms are quadratic in $T$ is used. In particular, the metric evolves by
\begin{equation*}
    \frac{\partial}{\partial t}g=-2{\rm{Ric}}(g)+ 2(\lambda-n\hat{\lambda}) (T\star T)+ 2\hat{\lambda}|T|^2g.
\end{equation*}
The previous evolution equation can be rewritten schematically as
\begin{equation}\label{eq: ddt g}
    \frac{\partial}{\partial t}g=-2{\rm{Ric}}(g)+T*T.
\end{equation}
Here and throughout this section, $*$ denotes a contraction using the metric $g(t)$ and the algebraic tensors naturally determined by the $\H$-structure, such as the projection tensors entering the definition of intrinsic torsion.  Covariant derivatives of these algebraic tensors contribute only additional torsion factors, since they are parallel for the canonical $\H$-connection and hence have Levi--Civita derivative schematically of the form $T*\mathcal A$, where $\mathcal A$ denotes such an algebraic tensor.  With this convention, Corollary~\ref{cor: main result of section 3} gives the schematic torsion evolution
\begin{equation}\label{eq: ddt T}
        \frac{\partial}{\partial t} T = \Delta T + \rR\rmm*T+\nabla T*T+T*T*T.
    \end{equation}
As a consequence, for some universal constant $C>0$, we have 
    \begin{align*}
        \frac{\partial}{\partial t}|T|^2=& \ \frac{\partial}{\partial t}\left(T_{i,ab}T_{j,pq}g^{ij}g^{ap}g^{bq}\right)=T*T*\Ric+2\langle T,\frac{\partial}{\partial t}T\rangle\\
        \leq & \ \Delta|T|^2-2|\nabla T|^2+C|\rR\rmm||T|^2+C|\nabla T||T|^2+C|T|^4.
    \end{align*}    

Using the Ricci identity 
\begin{align}
\label{eq: Ricci_identity}
    \nabla^r\Delta A=\Delta\nabla^rA+\sum_{s=0}^r\nabla^s\rR\rmm*\nabla^{r-s} A,
\end{align}
as well as 
\begin{align}\label{eq: Leibniz_rule_*}
    \frac{\partial}{\partial t}\nabla A-\nabla\frac{\partial}{\partial t}A=\nabla\frac{\partial}{\partial t}g*A, \qandq \nabla(A*B)=\nabla A*B+A*\nabla B,
\end{align}
we have
\begin{align}
\label{eq: ddt_nabla_eta}
\begin{split}
        \frac{\partial}{\partial t}\nabla T=& \ \Delta\nabla T+\nabla\rR \rmm *T+\rR \rmm *\nabla T+\rR \rmm *T*T\\
        &+\nabla^2T*T+\nabla T*\nabla T+\nabla T*T*T.
\end{split}
    \end{align}
    Moreover, as with the Ricci-flow, the evolution of $\rR\rmm$ is \cite{topping2006}*{(2.5.2)}
    \begin{align}\label{eq: ddt Rm}
    \begin{split}
         \frac{\partial}{\partial t}\rR\rmm=& \ \Delta\rR\rmm+\rR\rmm*(\rR\rmm+T*T)+\nabla^2 T*T+\nabla T*\nabla T.
    \end{split}
    \end{align}
Hence, the evolution of $|\rR\rmm|^2$, $|\nabla T|^2$ and $|T|^4$ satisfy
 \begin{align*}
        \frac{\partial}{\partial t}|T|^4
        \leq & \ \Delta|T|^4+C|\rR\rmm||T|^4+C|\nabla T||T|^4+C|T|^6\\
    \frac{\partial}{\partial t}|\rR\rmm|^2\leq& \ \Delta|\rR\rmm|^2-2|\nabla \rR\rmm|^2+C|\rR\rmm|^3+C|T|^2|\rR\rmm|^2+C|\rR\rmm|\left(|\nabla^2T||T|+|\nabla T|^2\right)\\
    \frac{\partial}{\partial t}|\nabla T|^2\leq & \ \Delta|\nabla T|^2-2|\nabla^2 T|^2+C|\nabla\rR\rmm||\nabla T||T|+C|\rR\rmm||\nabla T|^2+C|\nabla T||\nabla^2T||T|+C|\nabla T|^3+|\nabla T|^2|T|^2
\end{align*}

In terms of the quantity
\begin{align}
\label{eq: Lambda_quantity}
    \Lambda(x,t)=\left(|\rR\rmm|^2+|\nabla T|^2+|T|^4\right)^{\frac12},
\end{align}
we have
\begin{align*}
     \frac{\partial}{\partial t}|T|^4
        \leq & \ \Delta|T|^4+C\Lambda(x,t)^3\\
    \frac{\partial}{\partial t}|\rR\rmm|^2\leq& \ \Delta|\rR\rmm|^2-2|\nabla \rR\rmm|^2+C|\rR\rmm||\nabla^2T||T|+C\Lambda(x,t)^3\\
    \frac{\partial}{\partial t}|\nabla T|^2\leq & \ \Delta|\nabla T|^2-2|\nabla^2 T|^2+C|\nabla\rR\rmm||\nabla T||T|+C|\nabla T||T||\nabla^2T|+C\Lambda(x,t)^3
\end{align*}
Using Young's inequality 
\begin{align*}
    |\nabla \rR\rmm||T||\nabla T|\leq & \ \frac{\varepsilon}{2}|\nabla \rR\rmm|^2+\frac{1}{2\varepsilon}|T|^2|\nabla T|^2\\
    |\nabla^2T||T||\nabla T|\leq & \ \frac{\varepsilon}{2}|\nabla^2T|^2+\frac{1}{2\varepsilon}|T|^2|\nabla T|^2,\\
   |\rR\rmm||\nabla^2 T||T|\leq & \ \frac{\epsilon}{2}|\nabla^2 T|^2+\frac{1}{2\epsilon}|\rR\rmm|^2|T|^2,
\end{align*}
for any $\epsilon>0$. Bounding the terms $|\rR\rmm|,|\nabla T|,|T|^2$ above by $\Lambda(x,t)$, we obtain
\begin{align*}
     \frac{\partial}{\partial t}|T|^4
        \leq & \ \Delta|T|^4+C\Lambda(x,t)^3,\\
    \frac{\partial}{\partial t}|\rR\rmm|^2\leq& \ \Delta|\rR\rmm|^2-2|\nabla \rR\rmm|^2+\frac{\epsilon C}{2}|\nabla^2T|^2+C\Lambda(x,t)^3\\
    \frac{\partial}{\partial t}|\nabla T|^2\leq & \ \Delta|\nabla T|^2-2|\nabla^2T|^2+\frac{\epsilon C}{2}|\nabla\rR\rmm|^2+\frac{\epsilon C}{2}|\nabla^2T|^2+C\Lambda(x,t)^3.
\end{align*}
Combining the above, 
\begin{align*}
    \frac{\partial}{\partial t}\Lambda(x,t)^2\leq \ \Delta\Lambda(x,t)^2+\left(\epsilon C-2\right)\left(|\nabla\rR\rmm|^2+|\nabla^2T|^2\right)+C\Lambda(x,t)^3.
\end{align*}
Choosing $\epsilon>0$ such that $\epsilon C<1$,
\begin{align}
\label{eq: ddt Lambda2}
    \frac{\partial}{\partial t}\Lambda(x,t)^2\leq \ \Delta\Lambda(x,t)^2-\left(|\nabla\rR\rmm|^2+|\nabla^2 T|^2\right)+C\Lambda(x,t)^3.
\end{align}
Therefore, for the quantity
\begin{equation}
\label{eq: lambda_t}
    \Lambda(t)=\sup_M\Lambda(x,t),
\end{equation}
we have the following evolution estimate:

\begin{proposition}[Doubling-time estimate]
\label{prop: doub_time_est}
Let $\{\xi(t)\}_{t\in[0,\delta)}$ be a solution of the modified Ricci-harmonic flow \eqref{eq: modified Ricci-harmonic H flow} on a compact $n$-manifold. There exists a constant $C$ such that 
$$\Lambda(t)\leq 2\Lambda(0), \quad\forall t\in [0, \min\{\delta,(C\Lambda(0))^{-1}\}).
$$   
\end{proposition}

\begin{proof}
   From \eqref{eq: ddt Lambda2}, 
   $$
      \frac{\partial}{\partial t}\Lambda(x,t)\leq  \Delta\Lambda(x,t)+\langle\nabla(\log(\Lambda(x,t)),\nabla\Lambda(x,t)\rangle+\frac{C}{2}\Lambda(x,t)^2.
   $$
   Applying this inequality first to $(\Lambda(x,t)^2+\varepsilon)^{1/2}$
   and then letting $\varepsilon\downarrow0$, it follows from the maximum principle \cite{mantegazza2011lecture}*{Theorem 2.1.1} that the function $\Lambda(t)$ is locally Lipschitz and satisfies
   $$
   \frac{d}{dt}\Lambda(t)\leq \frac{C}{2}\Lambda(t)^2, 
   \quad\forall t\in [0,\delta).
   $$
   Then
   $$
   \Lambda(t)\leq \frac{\Lambda(0)}{1-\frac{C}{2}\Lambda(0)t},
   $$
   as long as $t<\min\{\delta,2(C\Lambda(0))^{-1}\}$, which yields the claim.
\end{proof}

\begin{theorem}[Shi-type estimates]
\label{thm: Shi-estimates} Let $K>0$ and $\xi(t)$ be a solution of the modified Ricci-harmonic flow \eqref{eq: modified Ricci-harmonic H flow} on a compact manifold $M^{n}$ with $t\in [0,\frac{1}{K}]$. For each $k\in \N$, there exists a constant $C_k$ such that, if $\Lambda(x,t)\leq K$ on $M\times [0,\frac{1}{K}]$, then 
\begin{equation}
\label{eq: Shi-estimates}
    |\nabla^k\rR\rmm| +|\nabla^{k+1}T|\leq \ C_kt^{-\frac{k}{2}}K \qforq t\in (0,\frac{1}{K}].
\end{equation}  
\end{theorem}

\begin{proof}
The proof goes by induction on $k$. We define a suitable function $f_k(x,t)$ satisfying a parabolic differential inequality manageable by the maximum principle. For $k=1$, consider
\begin{equation}
\label{eq: function_f}
    f(x,t)=t\left(|\nabla\rR\rmm|^2+|\nabla^2T|^2\right)+\alpha\left(|\rR\rmm|^2+|\nabla T|^2+|T|^4\right),
\end{equation}
where $\alpha$ will be determined later. First, we compute the evolution of $\nabla\rR\rmm$ and $\nabla^2T$. Using \eqref{eq: ddt g}, \eqref{eq: ddt Rm}, \eqref{eq: Ricci_identity} and \eqref{eq: Leibniz_rule_*}, we have
    \begin{align*}
        \frac{\partial}{\partial t}\nabla\rR\rmm=& \ \Delta\nabla\rR\rmm+\rR\rmm*\nabla\rR\rmm+\nabla\rR\rmm*T*T+\nabla^3T*T+\nabla^2T*\nabla T+\rR\rmm*\nabla T*T,
    \end{align*}
    thus
    \begin{align}
    \label{eq: ddt nabla_Rm}
    \begin{split}
        \frac{\partial}{\partial t}|\nabla\rR\rmm|^2\leq& \ \Delta|\nabla\rR\rmm|^2-2|\nabla^2\rR\rmm|^2+C|\rR\rmm||\nabla\rR\rmm|^2+C|T|^2|\nabla\rR\rmm|^2\\
        &\ +C|\nabla^3T||T||\nabla\rR\rmm|+C|\nabla^2T||\nabla T||\nabla\rR\rmm|+C|\nabla T||T||\rR\rmm||\nabla\rR\rmm|.
        \end{split}
    \end{align}
    Similarly, using \eqref{eq: Ricci_identity}, \eqref{eq: Leibniz_rule_*} and \eqref{eq: ddt_nabla_eta}, we obtain
    \begin{align*}
        \frac{\partial}{\partial t}\nabla^2T =& \ \Delta\nabla^2T+\nabla^2\rR \rmm *T+\nabla\rR \rmm *\nabla T+\nabla\rR\rmm*T*T+\rR \rmm *\nabla^2T\\
        &+\rR\rmm*\nabla T*T+\rR\rmm*T*T*T+\nabla^3T*T+\nabla^2T*\nabla T\\
        &+\nabla^2T*T*T+\nabla T*\nabla T*T+\nabla T*T*T*T,
    \end{align*}
    thus
    \begin{align}
    \label{eq: ddt nabla^2_eta}
    \begin{split}
        \frac{\partial}{\partial t}|\nabla^2T|^2\leq & \ \Delta|\nabla^2T|^2-2|\nabla^3T|^2+C|\nabla^2\rR\rmm||\nabla^2T||T|+C|\nabla\rR\rmm||\nabla^2T||\nabla T|+C|\nabla\rR\rmm||\nabla^2T||T|^2\\
        &+C|\rR\rmm||\nabla^2T|^2+C|\rR\rmm||\nabla T||T||\nabla^2T|+C|\rR\rmm||T|^3|\nabla^2T|+C|\nabla^3T||T||\nabla^2T|\\
        &+C|\nabla^2T|^2|\nabla T|+C|\nabla^2T|^2|T|^2+C|\nabla T|^2|T||\nabla^2T|+C|\nabla T||T|^3|\nabla^2T|.
    \end{split}
    \end{align}
  By hypothesis  $\Lambda(x,t)\leq K$, so the inequalities \eqref{eq: ddt nabla_Rm} and \eqref{eq: ddt nabla^2_eta} become
   \begin{align}\label{ineq: ddt nabla_Rm_2}
    \begin{split}
        \frac{\partial}{\partial t}|\nabla\rR\rmm|^2\leq& \ \Delta|\nabla\rR\rmm|^2-2|\nabla^2\rR\rmm|^2+CK|\nabla\rR\rmm|^2\\
        &\ +CK^{1/2}|\nabla^3T||\nabla\rR\rmm|+CK|\nabla^2T||\nabla\rR\rmm|+CK^{5/2}|\nabla\rR\rmm|
        \end{split}
    \end{align}
    and
    \begin{align}\label{ineq: ddt nabla^2_eta}
    \begin{split}
        \frac{\partial}{\partial t}|\nabla^2T|^2\leq & \ \Delta|\nabla^2T|^2-2|\nabla^3T|^2+CK^{1/2}|\nabla^2\rR\rmm||\nabla^2T|+CK|\nabla\rR\rmm||\nabla^2T|+CK|\nabla^2T|^2\\
        &+CK^{5/2}|\nabla^2T|+CK^{1/2}|\nabla^3T||\nabla^2T|,
    \end{split}
    \end{align}
    respectively. Using Young's inequality, for $\epsilon>0$,
    \begin{align*}
       K^{1/2}|\nabla^3T||\nabla\rR\rmm|\leq & \ \frac{\epsilon}{2}|\nabla^3T|^2+\frac{K}{2\epsilon}|\nabla\rR\rmm|^2 &
       K|\nabla^2T||\nabla\rR\rmm|\leq & \frac{K}{2}|\nabla^2T|^2+\frac{K}{2}|\nabla\rR\rmm|^2 \\
       K^{5/2}|\nabla\rR\rmm|\leq & \ \frac{K^4}{2}+\frac{K}{2}|\nabla\rR\rmm|^2 & K^{1/2}|\nabla^2\rR\rmm||\nabla^2T|\leq & \ \frac{\epsilon}{2}|\nabla^2\rR\rmm|^2+\frac{K}{2\epsilon}|\nabla^2T|^2\\
        K^{5/2}|\nabla^2T|\leq & \ \frac{K^4}{2}+\frac{K}{2}|\nabla^2T|^2 &
        K^{1/2}|\nabla^3T||\nabla^2T|\leq & \ \frac{\epsilon}{2}|\nabla^3T|^2+\frac{K}{2\epsilon}|\nabla^2T|^2.
    \end{align*}
Using the previous inequalities for suitably chosen small $\epsilon>0$, into \eqref{ineq: ddt nabla_Rm_2} and \eqref{ineq: ddt nabla^2_eta},  
\begin{align}\label{eq: ddt nabla(Rm+nablaT)}
\begin{split}
    \frac{\partial}{\partial t}\left(|\nabla\rR\rmm|^2+|\nabla^2T|^2\right)\leq & \ \Delta\left(|\nabla\rR\rmm|^2+|\nabla^2T|^2\right)-\left(|\nabla^2\rR\rmm|^2+|\nabla^3T|^2\right)\\
    &+CK\left(|\nabla\rR\rmm|^2+|\nabla^2T|^2\right)+CK^4.
    \end{split}
\end{align}
Thus, from \eqref{eq: ddt Lambda2} and \eqref{eq: ddt nabla(Rm+nablaT)}, we obtain
\begin{align*}
    \frac{\partial}{\partial t}f\leq & \ \Delta f+\left(CKt-\alpha\right)\left(|\nabla\rR\rmm|^2+|\nabla^2T|^2\right)+(Kt+\alpha)CK^3.
\end{align*}
Since $Kt\leq 1$, and finally choosing $\alpha\geq C$, we have
\begin{align}
\label{eq: heat_ineq_f}
    \frac{\partial}{\partial t}f\leq & \ \Delta f + (1+\alpha)CK^3.
\end{align}
Notice that $f(x,0)=\alpha\Lambda(x,0)^2\leq \alpha K^2$, so that applying the maximum principle to \eqref{eq: heat_ineq_f} yields 
\begin{align*}
    \sup_{x\in M}f(x,t)\leq \alpha K^2+C(\alpha+1)tK^3\leq CK^2.
\end{align*}
Then, by the definition \eqref{eq: function_f} of $f$, we obtain
\begin{align*}
    |\nabla\rR\rmm|+|\nabla^2T|\leq CKt^{-1/2},
\end{align*}
this proves \eqref{eq: Shi-estimates}, for $k=1$. 

Now, suppose $k\geq 2$. By iteration of \eqref{eq: Leibniz_rule_*} for any time-dependent tensor $A(t)$, we have
\begin{align}
\label{eq: k_Leibniz_rule_*}
    \frac{\partial}{\partial t}\nabla^kA-\nabla^k\frac{\partial}{\partial t}A=\sum_{i=1}^k\nabla^{k-i}A*\nabla^i\frac{\partial}{\partial t}g.
\end{align}
Applying \eqref{eq: ddt g}, \eqref{eq: ddt Rm}, \eqref{eq: Ricci_identity} and \eqref{eq: k_Leibniz_rule_*},
\begin{align*}
    \frac{\partial}{\partial t}\nabla^k\rR\rmm=&\Delta\nabla^k\rR\rmm+\sum_{i=0}^k\nabla^{k-i}\rR\rmm*\nabla^i\rR\rmm+\sum_{i=0}^k\nabla^{k-i}\rR\rmm*\nabla^i(T*T)\\
    &+\sum_{i=0}^k\nabla^{k-i+1}T*\nabla^{i+1}T+\sum_{i=0}^k\nabla^{k-i+2}T*\nabla^{i}T,
\end{align*}
and then
\begin{align}\nonumber
    \frac{\partial}{\partial t}|\nabla^k\rR\rmm|^2\leq &\Delta|\nabla^k\rR\rmm|^2-2|\nabla^{k+1}\rR\rmm|^2+C\sum_{i=0}^k|\nabla^{k-i}\rR\rmm||\nabla^i\rR\rmm||\nabla^k\rR\rmm|+C\sum_{i=0}^k|\nabla^{k-i}\rR\rmm||\nabla^i(T*T)||\nabla^k\rR\rmm|\\ \nonumber
    &+C\sum_{i=0}^k|\nabla^{k-i+1}T||\nabla^{i+1}T||\nabla^k\rR\rmm|+C\sum_{i=0}^k|\nabla^{k-i+2}T||\nabla^{i}T||\nabla^k\rR\rmm|\\ \label{eq: ddt_norm_nabla^kRm}
    \leq &\Delta|\nabla^k\rR\rmm|^2-2|\nabla^{k+1}\rR\rmm|^2+\mathrm{(i)}+\mathrm{(ii)}+\mathrm{(iii)}+\mathrm{(iv)}.
\end{align}
For the first reaction term in \eqref{eq: ddt_norm_nabla^kRm}, we have
\begin{align*}
    \mathrm{(i)}
    &= CK|\nabla^k\rR\rmm|^2+C\sum_{j=1}^{k-1}|\nabla^{k-j}\rR\rmm||\nabla^j\rR\rmm||\nabla^k\rR\rmm|\\
    &\leq  CK|\nabla^k\rR\rmm|^2+C\sum_{j=1}^{k-1}K^2t^{-\frac{k-j}{2}}t^{-\frac{j}{2}}|\nabla^k\rR\rmm|\\
     &\leq  CK|\nabla^k\rR\rmm|^2+CK^2t^{-\frac{k}{2}}|\nabla^k\rR\rmm|
\end{align*}
Now, for the second reaction term in \eqref{eq: ddt_norm_nabla^kRm}, we obtain
\begin{align*}
    \mathrm{(ii)}&=CK |\nabla^k\rR\rmm|^2 +C\sum_{j=1}^k|\nabla^{k-j}\rR\rmm||\nabla^j(T*T)||\nabla^k\rR\rmm|\\
    &\leq CK |\nabla^k\rR\rmm|^2 +C\sum_{j=1}^k\sum_{l=0}^j|\nabla^{k-j}\rR\rmm||\nabla^{j-l}T||\nabla^lT||\nabla^k\rR\rmm|\\
    &\leq  CK|\nabla^k\rR\rmm|^2 +C\sum_{j=1}^k|\nabla^{k-j}\rR\rmm||T||\nabla^jT||\nabla^k\rR\rmm| +C\sum_{j=1}^k\sum_{l=1}^{j-1}|\nabla^{k-j}\rR\rmm||\nabla^{j-l}T||\nabla^lT||\nabla^k\rR\rmm|\\
    &\leq  CK|\nabla^k\rR\rmm|^2+CK^{5/2}t^{-\frac{k-1}{2}}|\nabla^k\rR\rmm|+CK^3t^{-\frac{k-2}{2}}|\nabla^k\rR\rmm|.
\end{align*}
Similarly, for the third reaction term in \eqref{eq: ddt_norm_nabla^kRm},
\begin{align*}
    \mathrm{(iii)}
    &= CK^{1/2}|\nabla^{k+1}T||\nabla^k\rR\rmm| +CK|\nabla^{k+1}T||\nabla^k\rR\rmm| +C\sum_{j=2}^k|\nabla^{k-j+2}T||\nabla^jT||\nabla^k\rR\rmm|\\
    &\leq  CK^{1/2}|\nabla^{k+1}T||\nabla^k\rR\rmm| +CK|\nabla^{k+1}T||\nabla^k\rR\rmm|+CK^2t^{-\frac{k}{2}}|\nabla^k\rR\rmm|.
\end{align*}
Finally, for the fourth reaction term in \eqref{eq: ddt_norm_nabla^kRm},
\begin{align*}
    \mathrm{(iv)}
    &= CK|\nabla^{k+1}T||\nabla^k\rR\rmm| +C\sum_{j=1}^{k-1}|\nabla^{k-j+1}T||\nabla^{j+1}T||\nabla^k\rR\rmm|\\
    &\leq  CK|\nabla^{k+1}T||\nabla^k\rR\rmm| +CK^2t^{-\frac{k}{2}}|\nabla^k\rR\rmm|.
\end{align*}
Then, \eqref{eq: ddt_norm_nabla^kRm} becomes
\begin{align}\nonumber
  \frac{\partial}{\partial t}|\nabla^k\rR\rmm|^2 \leq & \  \Delta|\nabla^k\rR\rmm|^2-2|\nabla^{k+1}\rR\rmm|^2+CK(|\nabla^k\rR\rmm|+|\nabla^{k+1}T|)|\nabla^k\rR\rmm|\\ \nonumber
  & \ +CK^{5/2}t^{-\frac{k-1}{2}}|\nabla^k\rR\rmm|+CK^3t^{-\frac{k-2}{2}}|\nabla^k\rR\rmm|+Ct^{-\frac{k}{2}}K^2|\nabla^k\rR\rmm|\\ \label{eq: nabla^kRm}
   \leq & \  \Delta|\nabla^k\rR\rmm|^2-2|\nabla^{k+1}\rR\rmm|^2+CK(|\nabla^{k+1}T|^2+|\nabla^k\rR\rmm|^2)+CK^3t^{-k}.
\end{align}
Here, we used Young's inequality and the hypothesis $Kt\leq 1$. Similarly, by \eqref{eq: ddt g}, \eqref{eq: ddt T}, \eqref{eq: Ricci_identity} and \eqref{eq: k_Leibniz_rule_*}, we get
\begin{align*}
    \frac{\partial}{\partial t}\nabla^{k+1}T=& \ \Delta\nabla^{k+1}T+
    \sum_{i=0}^{k+1}\nabla^{k+1-i}T*\nabla^i\left(\rR\rmm+T*T\right)\\
    &+\nabla^{k+1}\left(\rR\rmm*T\right)
    +\nabla^{k+1}\left(\nabla T*T+T*T*T\right).
\end{align*}
and then,
\begin{align}\label{eq: ddt_norm_nabla^k_eta}
\begin{split}
    \frac{\partial}{\partial t}|\nabla^{k+1}T|^2=& \ \Delta|\nabla^{k+1} T|^2-2|\nabla^{k+2}T|^2\\
    & +\sum_{i=0}^{k+1}\nabla^{k+1-i}T*\nabla^i\left(\rR\rmm+T*T\right)*\nabla^{k+1}T\\
    &+\nabla^{k+1}\left(\rR\rmm*T\right)*\nabla^{k+1}T\\
    &+\nabla^{k+1}\left(\nabla T*T+T*T*T\right)*\nabla^{k+1}T\\
     =& \ \Delta|\nabla^{k+1}T|^2-2|\nabla^{k+2}T|^2+(\rI)+(\rI\rI)+(\rI\rI\rI).
     \end{split}
\end{align}

Using \eqref{eq: Shi-estimates} for $i\leq k$, the second line of \eqref{eq: ddt_norm_nabla^k_eta} can be estimated as
\begin{align}\label{eq: I}
    \begin{split}
        (\rI)\leq & \ C\sum_{i=0}^{k+1}|\nabla^{k+1-i}T||\nabla^i\left(\rR\rmm+T*T\right)||\nabla^{k+1}T|\\
        \leq & \ CK\left(K^{-1/2}|\nabla^{k+1}\rR\rmm|+|\nabla^k\rR\rmm|+|\nabla^{k+1}T|\right)|\nabla^{k+1} T| +CK^2t^{-\frac{k}{2}} |\nabla^{k+1}T|.
    \end{split}
\end{align}
Similarly, for the third line of \eqref{eq: ddt_norm_nabla^k_eta} we have
\begin{align}\label{eq: II}
    \begin{split}
       (\rI\rI) \leq & \ C|\nabla^{k+1}\left(\rR\rmm*T\right)||\nabla^{k+1}T|\\
       \leq & \ CK\left(K^{-1/2}|\nabla^{k+1}\rR\rmm|+|\nabla^k\rR\rmm|+|\nabla^{k+1}T|\right)|\nabla^{k+1} T|\\
       &+CK^2t^{-\frac{k}{2}}|\nabla^{k+1}T|.
    \end{split}
\end{align}
Finally, for the fourth line of  \eqref{eq: ddt_norm_nabla^k_eta}, we have
\begin{align}\label{eq: III}
    \begin{split}
        \mathrm{(III)}\leq & \ C|\nabla^{k+1}\left(\nabla T*T+T*T*T\right)||\nabla^{k+1}T|\\
        \leq & \ CK\left(K^{-1/2}|\nabla^{k+2}T|+|\nabla^{k+1}T|\right)|\nabla^{k+1}T|+CK^2t^{-\frac{k}{2}}|\nabla^{k+1}T|.
    \end{split}
\end{align}
Using \eqref{eq: I}, \eqref{eq: II} and \eqref{eq: III}, we can estimate \eqref{eq: ddt_norm_nabla^k_eta} by
\begin{align}\label{eq: ddt_nabla^k+1_eta_2}
    \begin{split}
          \frac{\partial}{\partial t}|\nabla^{k+1}T|^2\leq & \ \Delta|\nabla^{k+1}T|^2-2|\nabla^{k+2}T|^2+CK\left(|\nabla^k\rR\rmm|+|\nabla^{k+1}T|\right)|\nabla^{k+1}T|\\
          & \ +CK^{\frac12}\left(|\nabla^{k+1}\rR\rmm|+|\nabla^{k+2}T|\right)|\nabla^{k+1}T|+CK^2\left(t^{-\frac{k}{2}}+K^{\frac12}t^{-\frac{k-1}{2}}+Kt^{-\frac{k-2}{2}}\right)|\nabla^{k+1}T|.
    \end{split}
\end{align}
By the assumption $tK\leq 1$ and using Young's inequality, \eqref{eq: ddt_nabla^k+1_eta_2} becomes
\begin{align}\label{eq: ddt_norm_eta^k_ineq}
    \begin{split}
         \frac{\partial}{\partial t}|\nabla^{k+1}T|^2\leq & \ \Delta|\nabla^{k+1}T|^2-2|\nabla^{k+2}T|^2+CK\left(|\nabla^k\rR\rmm|^2+|\nabla^{k+1}T|^2\right)\\
          & \ +CK^{\frac12}\left(|\nabla^{k+1}\rR\rmm|+|\nabla^{k+2}T|\right)|\nabla^{k+1}T| +CK^3t^{-k}.
    \end{split}
\end{align}
Thus, from the estimates \eqref{eq: nabla^kRm} and  \eqref{eq: ddt_norm_eta^k_ineq}, we obtain
\begin{align}\label{eq: ddt_norm_nabla^k}
    \begin{split}
        \frac{\partial}{\partial t}\left(|\nabla^k\rR\rmm|^2+|\nabla^{k+1}T|^2\right)\leq & \ \Delta\left(|\nabla^k\rR\rmm|^2+|\nabla^{k+1}T|^2\right) -2\left(|\nabla^{k+1}\rR\rmm|^2+|\nabla^{k+2}T|^2\right) \\
        & \  + CK^{\frac12}|\nabla^{k+1}\rR\rmm||\nabla^{k+1}T|+CK\left(|\nabla^k\rR\rmm|^2+|\nabla^{k+1}T|^2\right)+CK^3t^{-k}.
    \end{split}
\end{align}
Using Young's inequality once again, for any $\epsilon>0$,
\begin{align*}
    K^{\frac12}|\nabla^{k+1}\rR\rmm||\nabla^{k+1}T|\leq & \ \frac{K}{2\epsilon}|\nabla^{k+1}T|^2+\frac{\epsilon}{2}|\nabla^{k+1}\rR\rmm|^2\\
    K^{\frac12}|\nabla^{k+2}T||\nabla^{k+1}T|\leq & \ \frac{K}{2\epsilon}|\nabla^{k+1}T|^2+\frac{\epsilon}{2}|\nabla^{k+2}T|^2.
\end{align*}
Using the estimates above and choosing $\epsilon<1/C$,  inequality \eqref{eq: ddt_norm_nabla^k} becomes
\begin{align}\label{eq: ddt nabla^k_Rm_eta_J}
    \begin{split}
        \frac{\partial}{\partial t}\left(|\nabla^k\rR\rmm|^2+|\nabla^{k+1}T|^2\right)\leq & \ \Delta\left(|\nabla^k\rR\rmm|^2+|\nabla^{k+1}T|^2\right)+CK^3t^{-k}\\
        & \ -\left(|\nabla^{k+1}\rR\rmm|^2+|\nabla^{k+2}T|^2\right)+CK\left(|\nabla^k\rR\rmm|^2+|\nabla^{k+1}T|^2\right).
    \end{split}
\end{align}

Now, for $k\geq 2$ we define the function
\begin{align}
\label{eq: f_k}
    \begin{split}
        f_k(t,x)=&t^k\left(|\nabla^k\rR\rmm|^2+|\nabla^{k+1}T|^2\right) +\beta_k\sum_{i=1}^k\alpha_i^kt^{k-i}\left(|\nabla^{k-i}\rR\rmm|^2+|\nabla^{k+1-i}T|^2\right),
    \end{split}
\end{align}
where $\beta_k$ and $\alpha^k_i$ are constants to be determined later. By the induction hypothesis, for any $1\leq i< k$ a similar computation to \eqref{eq: ddt nabla^k_Rm_eta_J} yields
\begin{align}
\label{eq: ddt nabla^k-i_Rm_eta_J}
    \begin{split}
        \frac{\partial}{\partial t}\left(|\nabla^{k-i}\rR\rmm|^2+|\nabla^{k+1-i}T|^2\right)\leq & \ \Delta\left(|\nabla^{k-i}\rR\rmm|^2+|\nabla^{k+1-i}T|^2\right)+CK^3t^{-(k-i)}\\
        & \ -\left(|\nabla^{k+1-i}\rR\rmm|^2+|\nabla^{k+2-i}T|^2\right).
    \end{split}
\end{align}
Notice that we used the induction hypothesis for the term
$$
  CK\left(|\nabla^{k-i}\rR\rmm|^2+|\nabla^{k+1-i}T|^2\right)\leq CK^3t^{-(k-i)},
$$
which corresponds to the last term of \eqref{eq: ddt nabla^k_Rm_eta_J}. Thus, from \eqref{eq: ddt nabla^k_Rm_eta_J}, \eqref{eq: f_k} and \eqref{eq: ddt nabla^k-i_Rm_eta_J} we obtain
\begin{align*}
    \frac{\partial}{\partial t}f_k\leq & \ t^k\Delta(\left(|\nabla^{k}\rR\rmm|^2+|\nabla^{k+1}T|^2\right)-t^k\left(|\nabla^{k+1}\rR\rmm|^2+|\nabla^{k+2}T|^2\right)\\
    & \ +CK^3+\left(kt^{k-1}+CKt^k\right)\left(|\nabla^{k}\rR\rmm|^2+|\nabla^{k+1}T|^2\right)\\
    & \ +\beta_k\sum_{i=1}^k\left(\alpha_i^kt^{k-i}\Delta\left(|\nabla^{k-i}\rR\rmm|^2+|\nabla^{k+1-i}T|^2\right)+\alpha_i^kCK^3\right.\\
    & \ \left. -\alpha_i^kt^{k-i}\left(|\nabla^{k+1-i}\rR\rmm|^2+|\nabla^{k+2-i}T|^2\right)\right)\\
    & \ +\beta_k\sum_{i=1}^k(k-i)\alpha_i^kt^{k-i-1}\left(|\nabla^{k-i}\rR\rmm|^2+|\nabla^{k+1-i}T|^2\right).
\end{align*}
Collecting terms, we see that
\begin{align}\label{eq: ddt f_k}
    \begin{split}
         \frac{\partial}{\partial t}f_k\leq & \ \Delta f_k+\left(kt^{k-1}+CKt^k-\beta_k\alpha_1^kt^{k-1}\right)\left(|\nabla^{k}\rR\rmm|^2+|\nabla^{k+1}T|^2\right)\\
    & \ +\beta_k\sum_{i=1}^{k-1}\left((k-i)\alpha_i^k-\alpha_{i+1}^k\right)t^{k-i}\left(|\nabla^{k-i}\rR\rmm|^2+|\nabla^{k+1-i}T|^2\right)+\left(C+C\beta_k\sum_{i=1}^k\alpha_i^k\right)K^3.
    \end{split}
\end{align}
Now, take $\alpha_i^k=\frac{(k-1)!}{(k-i)!}$ and $\beta_k>k+C$. Notice that $(k-i)\alpha_i^k-\alpha_{i+1}^k=0$, and since $Kt\leq 1$ the inequality \eqref{eq: ddt f_k} becomes
\begin{align}
\label{eq: ddt f_k_ineq}
     \frac{\partial}{\partial t}f_k\leq & \ \Delta f_k+CK^3.
\end{align}
Since $f_k(0,x)=\beta_k\alpha_k^k\left(|\rR\rmm|^2+|\nabla T|^2\right)\leq \beta_k\alpha_k^kK^2$, by the maximum principle \eqref{eq: ddt f_k_ineq} implies
\begin{align*}
    \sup_{x\in M}f_k(x,t)\leq \beta_k\alpha_k^kK^2+CtK^3\leq CK^2.
\end{align*}
Finally, from the definition of $f_k$ in  \eqref{eq: f_k}, we obtain
$$
|\nabla^k\rR\rmm|+|\nabla^{k+1}T|\leq CKt^{-\frac{k}{2}},
$$
and the result follows. 
\end{proof}

The previous proof can be adapted to obtain a local version of Theorem \ref{thm: Shi-estimates}

\begin{corollary}[Local derivative estimates] Let $K>0$ and $r>0$. Let $(M,g)$ be a Riemannian manifold, fix $p\in M$ and let $\{\xi(t)\}$ be a solution of the modified Ricci-harmonic flow \eqref{eq: modified Ricci-harmonic H flow} on $B_r(p)\times [0,\frac{1}{K}]$, where $B_r(p)$ denotes the ball of radius $r$ with respect to $g=g(0)$. If $$\Lambda(x,t)\leq K, \quad\forall (x,t)\in B_r(p)\times [0,\frac{1}{K}],$$ 
then for each $k\in \N$, there exists $C_k(K,r)$ such that
\begin{equation}
    |\nabla^k\rR\rmm|+|\nabla^{k+1}T|\leq C_k(K,r)t^{-\frac{k}{2}}\quad \text{on} \quad B_{r/2}(p)\times \left[0,\frac{1}{K}\right].
\end{equation}   
\end{corollary}

\subsection{Long-time existence}

Using the Shi-type estimates and following the ideas from \cite{lotay-wei-gafa}, we have the following result.

\begin{theorem}
\label{thm: Lambda quantity for LTE}
    If $\{\xi(t)\}_{[0,T_0)}$ is a solution of the modified Ricci-harmonic flow \eqref{eq: modified Ricci-harmonic H flow} on a compact manifold $M^{n}$ on a maximal time interval  with $T_0<\infty$, then $\Lambda(t)$ defined in \eqref{eq: lambda_t} satisfies
    \begin{align}\label{eq: lim_Lambda}
        \lim_{t\nearrow T_0}\Lambda(t)=\infty.
    \end{align}
    Moreover, for some constant $C>0$, the quantity $\Lambda(t)$ has the lower bound 
    \begin{align}
    \label{eq: Lambda_bound_below}
        \Lambda(t)\geq \frac{C}{T_0-t}.
    \end{align}
\end{theorem}

\begin{proof}
    Let  $\{\xi(t)\}_{[0,T_0)}$ be a maximal solution of the modified Ricci-harmonic flow \eqref{eq: modified Ricci-harmonic H flow}. First, we claim that
    \begin{align}
    \label{eq: limsup_Lambda}
        \limsup_{t\nearrow T_0}\Lambda(t)=\infty.
    \end{align}
    By contradiction, suppose that \eqref{eq: limsup_Lambda} does not hold, then there exists $K>0$ such that
    \begin{align}\label{eq: Lambda_bound_contradiction}
\sup_{M\times[0,T_0)}\Lambda(x,t)=\sup_{M\times[0,T_0)}\left(|\rR\rmm|_{g(t)}^2+|\nabla T|_{g(t)}^2+|T|_{g(t)}^4\right)^{\frac12}\leq K.
    \end{align}
    In particular, the curvature and the torsion have uniform bounds
    \begin{align*}
        \sup_{M\times[0,T_0)}|\rR\rmm|_{g(t)}\leq K \qandq \sup_{M\times[0,T_0)}|T|_{g(t)}^2\leq K,
    \end{align*}
    and therefore 
    \begin{align} \label{ineq: sup g}
        \sup_{M\times[0,T_0)}\left|\frac{\partial}{\partial t}g_{ij}\right|_{g(t)}=\sup_{M\times[0,T_0)}\left|-2{\rm{Ric}}(g)+ 2(\lambda-n\hat{\lambda}) (T\star T)+ 2\hat{\lambda}|T|^2g\right|_{g(t)}\leq CK.
    \end{align}
    By \cite{HamiltonRic}*{Lemma 14.2},  the metrics $g(t)$ for all times are equivalent and as $t\to T_0$ they converge to a positive-definite metric $g(T_0)$ which is continuous and also equivalent. Notice that the metric compatibility of the $\H$-structure implies that $\xi(t)$ has constant (time-independent) length with respect to $g(t)$. Indeed,
    $$
      d||\xi||^2=2\langle T\diamond\xi,\xi\rangle=0,
    $$
    and from the general form of the $\H$-flow \eqref{equ: GF} we have
    \begin{equation*}
        \begin{split}
            \frac{\partial}{\partial t}||\xi||^2=2\langle \frac{\partial}{\partial t}\xi,\xi\rangle-2\langle S\diamond \xi,\xi\rangle=0.
        \end{split}
    \end{equation*}
    Now, from  \eqref{eq: Lambda_bound_contradiction} and \eqref{ineq: sup g}, we obtain
    \begin{align}\label{eq: bounds_g_t}
        \begin{split}
            \left|\frac{\partial}{\partial t}\xi\right|_{g(t)}=&\left|\left(-\Ric+(\lambda-n\hat{\lambda}) (T\star T)+ \hat{\lambda}|T|^2g+\Div T\right)\diamond \xi\right|_{g(t)}\leq CK,
        \end{split}
    \end{align}
    for some uniform positive constant $C>0$. Besides, fix a background metric $\bar{g}=g(0)$, then from the equivalence between $g(t)$ and $\bar{g}$, the estimate \eqref{eq: bounds_g_t} becomes
    \begin{align}\label{eq: bounds_bar_g}
            \left|\frac{\partial}{\partial t}\xi\right|_{\bar{g}}\leq CK.
    \end{align}
    For any $0 <t_1<t_2< T_0$,
    \begin{align}
    \label{eq: uniform_cont}
        \begin{split}
            |\xi(t_2)-\xi(t_1)|_{\bar{g}}&\leq \int^{t_2}_{t_1} \left|\frac{\partial}{\partial t}\xi\right|_{\bar{g}}dt\leq CK(t_2-t_1).
        \end{split}
    \end{align}
    Then $\xi(t)\to \xi(T_0)$ continuously as $t\to T_0$. Notice that, for any $t\in [0,T_0)$ and $x\in M$, we have:
    \begin{align}\label{eq: SU_condition_t}
        \xi(t)(x)=\sigma_t(x)^*\xi_o 
    \end{align}
    where $\sigma_t\in \Gamma\left(\rF\rr(M)/\H \right)$ and $\xi_\circ\in (\bR^n)^{\otimes p}\otimes((\bR^n)^*)^{\otimes q}$ is the Euclidean $\H$--structure model. Since the defining compatibility relations for the model tensor are algebraic and closed, and since the limiting metric $g(T_0)$ is positive definite, the pointwise limit $\xi(T_0)$ still lies in the same $\mathrm{GL}^+(n,\mathbb R)$-orbit of the model tensor. Hence $\xi(T_0)$ defines a continuous $\H$-structure. 
    
    Now, we will show that $\xi(T_0)$ is in fact smooth. First, by \cite{chow2011}*{Proposition 6.48} for each $r\in \mathbb{N}$ there exists a constant $C_r$, such that 
    \begin{align}\label{eq: estimate_nabla_r_g}
        \sup_{M\times[0,T_0)}\left|\bar{\nabla}^r g(t)\right|_{\bar{g}}\leq C_r,
    \end{align}
    where $C_r$ only depends on $r,n,T_0$ and $K$. Similarly, we prove that
   for each $r\in \mathbb{N}$ there exists a constant $C_r$ such that
    \begin{align}\label{eq: estimate_nabla_r_J}
        \sup_{M\times[0,T_0)}\left|\bar{\nabla}^r\xi(t)\right|_{\bar{g}}\leq C_r.
    \end{align}
    We will establish \eqref{eq: estimate_nabla_r_J} by induction on $r$. Consider $r=1$, for any $(x,t)\in M\times[0,T_0)$ we have
    \begin{align*}
        \frac{\partial}{\partial t}\bar{\nabla}\xi=&\bar{\nabla}\frac{\partial}{\partial t}\xi=\bar{\nabla}\left(\left(-\Ric+T*T+\Div T\right)\diamond \xi\right)\\
       =& \nabla\left(\left(-\Ric+T*T+\Div T\right)\diamond \xi\right)+B*\left(\left(-\Ric+T*T+\Div T\right)\diamond \xi\right),
    \end{align*}
    where $B=\bar{\nabla}-\nabla\in \Omega^1(M,\End(TM))$ is the difference of two connections. Then, by \eqref{eq: standard metric variation} we obtain
    \begin{align*}
        \frac{\partial}{\partial t}B=g^{-1}*\nabla(-\Ric+ T*T).
    \end{align*}
    Integrating in time $t<T_0$, using the equivalence metric,  the estimate \eqref{eq: Lambda_bound_contradiction} and Theorem \ref{thm: Shi-estimates} we have
    \begin{align*}
        |B(t)|_{\bar{g}}\leq &  |B(0)|_{\bar{g}}+\int_0^t \left|\frac{\partial}{\partial s}B\right|_{\bar{g}}ds\\
        \leq &  |B(0)|_{\bar{g}}+C\int_0^t \left|\frac{\partial}{\partial s}B\right|_{g(s)}ds\\
        \leq &  |B(0)|_{\bar{g}}+C\left(|\nabla\Ric|+ |\nabla T||T|\right)t\leq C.
    \end{align*}
    Moreover, using \eqref{eq: estimate_nabla_r_g},
    \begin{align}
    \label{eq: nabla^kB}
        |\bar{\nabla}^kB(t)|_{\bar{g}}\leq C \qforq 0\leq k\leq r-1.
    \end{align}
    Thus, from \eqref{eq: Lambda_bound_contradiction}, \eqref{ineq: sup g}, Theorem \ref{thm: Shi-estimates} and \eqref{eq: nabla^kB}, we have
    $
    \left|\frac{\partial}{\partial t}\bar{\nabla}\xi\right|_{\bar{g}}\leq C.
    $
    Then 
    \begin{align*}
        \left|\bar{\nabla}\xi(t)\right|_{\bar{g}}\leq \left|\bar{\nabla}\xi(0)\right|_{\bar{g}}+\int_0^t\left|\frac{\partial}{\partial s}\bar{\nabla}\xi(s)\right|_{\bar{g}}ds\leq \left|\bar{\nabla}\xi(0)\right|_{\bar{g}}+CT_0,
    \end{align*}
    and we have proved \eqref{eq: estimate_nabla_r_J} for $r=1$. 
    
    For $r=2$, it is easy to get
    \begin{align*}
        \frac{\partial}{\partial t}\bar{\nabla}^2\xi
      =& \nabla^2\left(\left(-\Ric+T*T+\Div T\right)\diamond \xi\right)+B*\nabla\left(\left(-\Ric+T*T+\Div T\right)\diamond \xi\right)\\
      &+\bar{\nabla}B*\left(\left(-\Ric+T*T+\Div T\right)\diamond \xi\right)+B*B*\left(\left(-\Ric+T*T+\Div T\right)\diamond \xi\right)
    \end{align*}
    and similarly, for $r=3$,
    \begin{align*}
         \frac{\partial}{\partial t}\bar{\nabla}^3\xi
      =& \nabla^3\left(\left(-\Ric+T*T+\Div T\right)\diamond \xi\right)+B*\nabla^2\left(\left(-\Ric+T*T+\Div T\right)\diamond \xi\right)\\
      &+\bar{\nabla}B*\nabla\left(\left(-\Ric+T*T+\Div T\right)\diamond \xi\right)+B*B*\nabla\left(\left(-\Ric+T*T+\Div T\right)\diamond \xi\right)\\
      &+\bar{\nabla}^2B*\left(\left(-\Ric+T*T+\Div T\right)\diamond \xi\right)+\bar{\nabla}B*B*\left(\left(-\Ric+T*T+\Div T\right)\diamond \xi\right)\\
      &+B*B*B*\left(\left(-\Ric+T*T+\Div T\right)\diamond \xi\right).
    \end{align*}
    Thus, for $r\geq 2$, we have that $\frac{\partial}{\partial t}\bar{\nabla}^r\xi$ is given by the terms 
    $$
    \nabla^r\left(\left(-\Ric+T*T+\Div T\right)\diamond \xi\right), \quad B^r*\left(\left(-\Ric+T*T+\Div T\right)\diamond \xi\right)
    $$
    and several contractions of terms of the form 
    $$
    B^i, \quad \bar{\nabla}^jB, \quad \nabla^k((\Div T)\diamond \xi),
    \qforq i,j,k<r. 
    $$
    By \eqref{ineq: sup g}, \eqref{eq: Shi-estimates} and \eqref{eq: nabla^kB}, we have
    \begin{align}
    \label{eq: ddt nabla^rJ}
        \left|\frac{\partial}{\partial t}\bar{\nabla}^r\xi\right|_{\bar{g}}\leq C,
    \end{align}
    which we integrate to obtain
    \begin{align*}
        \left|\bar{\nabla}^r\xi(t)\right|_{\bar{g}}\leq &  \left|\bar{\nabla}^r\xi(0)\right|_{\bar{g}}+CT_0
    \end{align*}
    and deduce \eqref{eq: estimate_nabla_r_J}. We have proved that there is a continuous limiting  $\H$-structure $\xi(T_0)$, and we fix a local coordinate system $(U,x^1,...,x^{n})$ such that
    \begin{align}\label{eq: J(T_0)}
        \begin{split}
            \xi^{i_1\cdots i_p}_{j_1\cdots j_q}(T_0)=&\xi^{i_1\cdots i_p}_{j_1\cdots j_q}(t)+\int_t^{T_0}\left(\left(-\Ric+T*T+\Div T\right)\diamond\xi(s)\right)^{i_1\cdots i_p}_{j_1\cdots j_q}ds.
        \end{split}
    \end{align}
    From \eqref{eq: estimate_nabla_r_g}, \eqref{eq: estimate_nabla_r_J} and \eqref{eq: ddt nabla^rJ}, we see that
    \begin{align}
    \label{eq: _deri_esti_J_Upsilon}
        \begin{split}
            \frac{\partial^r}{\partial x^\alpha}&\xi^{i_1\cdots i_p}_{j_1\cdots j_q}  \qandq
           \frac{\partial^r}{\partial x^\alpha} \left(\left(-\Ric+T*T+\Div T\right)\diamond\xi\right)^{i_1\cdots i_p}_{j_1\cdots j_q}
        \end{split}
    \end{align}
   are uniformly bounded on $U\times[0,T_0)$, where $\alpha=(a_1,\ldots,a_r)$. Hence, from \eqref{eq: J(T_0)}, each  $\frac{\partial^r}{\partial x^\alpha}\xi^{i_1\cdots i_p}_{j_1\cdots j_q}$ is uniformly bounded on $U$, and the component functions converge smoothly to a smooth tensor $\xi(T_0)$. Since the model orbit defining $\H$-structures is a closed algebraic submanifold of the tensor bundle and $\xi(t)$ lies in this orbit for every $t<T_0$, the limit $\xi(T_0)$ is again a smooth $\H$-structure. Moreover,
   \begin{align*}
       \begin{split}
            \left|\frac{\partial^r}{\partial x^\alpha}\xi^{i_1\cdots i_p}_{j_1\cdots j_q}(T_0)-\frac{\partial^r}{\partial x^\alpha}\xi^{i_1\cdots i_p}_{j_1\cdots j_q}(t)\right|\leq & C(T_0-t),
       \end{split}
   \end{align*}
   thus, $\xi(t)\to \xi(T_0)$ smoothly as $t\to T_0$. 
   
   Now, by Theorem \ref{thm: STE Ricci harmonic H}, there exists a solution $\{\bar{\xi}(t)\}_{t\in [0,\epsilon)}$ of \eqref{eq: modified Ricci-harmonic H flow}, with $\bar{\xi}(0)=\xi(T_0)$. Then the family of smooth $\H$-structures defined by 
   \begin{align*}
       \tilde{\xi}(t)=\begin{cases}
           \xi(t), & 0\leq t < T_0\\
           \bar{\xi}(t-T_0), & T_0\leq t < T_0+\epsilon
       \end{cases},
       \qforq 
       t\in [0,T_0+\epsilon),
   \end{align*}
   solves \eqref{eq: modified Ricci-harmonic H flow}, with $\tilde{\xi}(0)=\xi(0)$, which is a contradiction to the maximality of $T_0$ and thus proves  \eqref{eq: limsup_Lambda}. 
   
   Now, we will prove \eqref{eq: lim_Lambda} by contradiction. Suppose there is a sequence $t_i\nearrow T_0$ such that $\Lambda(t_i)\leq K_0$, for some constant $K_0$. Using Proposition \ref{prop: doub_time_est}, 
    \begin{equation}
        \Lambda(t) \leq 2\Lambda(t_i)\leq 2K_0, \quad \forall  t\in [t_i,\min\{T_0,t_i+1/(CK_0)\}).
    \end{equation}
    Since $t_i\to T_0$, for  $i\gg 1$ we have $t_i+\frac{1}{CK_0}\geq T_0$, and so
    $$
    \sup _{M\times[t_i,T_0)}\Lambda(x,t)\leq 2K_0.
    $$
    However, we have proved that the last inequality produces a contradiction to the maximality of $T_0$ hence \eqref{eq: lim_Lambda} is true. 
    
    Finally, using the maximum principle in \eqref{eq: ddt Lambda2}, we have
           \begin{equation}\label{eq: ddt LAmbda inv}
               \frac{d}{dt}\Lambda(t)^{-1}\geq -\frac{C}{2},
           \end{equation}
           also, by \eqref{eq: lim_Lambda} we have $\lim_{t\to T_0}\Lambda(t)^{-1}=0$, then integrating \eqref{eq: ddt LAmbda inv} from $t$ to $t'\in (t,T_0)$ and taking $t'\to T_0$, we get \eqref{eq: Lambda_bound_below}.
    \end{proof}

%%%%%%%%%%%%%%%%%%%%%%%%%%%%%%%%%%%%%%%%%%%%%

\newpage

\appendix

\section{Differential invariants of \texorpdfstring{$\SU(3)$}{}-structures}\label{sec: SU3-torsion forms - appendix}

Throughout this paper, the intrinsic torsion of an $\SU(3)$-structure has been encoded by the tensor $T$ defined in \eqref{eq: nabla J_ nabla Upsilon}, and the quantities $\eta$, $VT$, $\Div T$, and $\Ric^*$ have been introduced as natural contractions in the general $\H$-structure formalism. This notation is well suited to the uniform derivation of the evolution equations and principal-symbol formulae above, but it is less directly comparable with the standard description of $\SU(3)$-geometry in terms of torsion forms. The first goal of this section is therefore to translate the quantities used in the previous sections into the usual torsion-form notation for $\SU(3)$-structures. These formulae are useful, in particular, for working with explicit examples such as Lie groups, homogeneous spaces, and cohomogeneity-one manifolds. The second goal is to record the representation-theoretic input needed to describe second-order quasilinear $\SU(3)$-flows to highest order; this leads to Theorem \ref{thm: general second order SU(3) flow}.

We recall from \cite{bedulli2007} the definition of the $\SU(3)$ torsion forms $\{\pi_0,\sigma_0,\pi_1,\nu_1,\sigma_2,\pi_2,\nu_3\}$ given by
\begin{align*}
    d\om &= -\frac{3}{2}\sigma_0 \Upsilon^+ + \frac{3}{2}\pi_0\Upsilon^- + \nu_1 \w \om + \nu_3,\\
    d\Upsilon^+ &= \pi_0 \om^2 + \pi_1 \w \Upsilon^+ - \pi_2 \w \om, \\
    d\Upsilon^- &= \sigma_0 \om^2 + J\pi_1 \w \Upsilon^+ - \sigma_2 \w \om,
\end{align*}
where $\pi_0,\sigma_0 \in \Lm^0$, $\pi_1, \nu_1\in \Lm^1_6 \cong [\Lm^{1,0}]$, $\pi_2,\sigma_2 \in \Lm^2_8\cong [\![\Lm^{1,1}_0]\!] \cong \mathfrak{su}(3)$ and $\nu_3 \in \Lm^3_{12}\cong [\Lm^{2,1}_0]$. We use the differential form inner product, i.e., we divide by $k!$ on $k$-forms.

\subsection{Classification of first- and second-order \texorpdfstring{$\SU(3)$}{}-invariants}

The space of first-order $\SU(3)$-invariants is given by
\begin{equation*}
    V_1(\mathfrak{su}(3)) \cong 2\mathbb{R} \oplus 2 \Lm^1_6 \oplus 2 \Lm^2_8 \oplus \Lm^3_{12}
\end{equation*}
which is of course generated by the aforementioned torsion forms $\{\sigma_i,\pi_i,\nu_i\}$. Similarly, the space of second-order $\SU(3)$-invariants is given by
\begin{equation}\label{equ: V2su3}
    V_2(\mathfrak{su}(3)) \cong 3\mathbb{R} \oplus 4 \Lm^1_6 \oplus 5 \Lm^2_8 \oplus 4 \Lm^3_{12} \oplus 3 V_{2,1} \oplus V_{3,0} \oplus V_{2,2}.
\end{equation}
The latter was computed in \cite{bedulli2007}*{\S 2.6} following the ideas of \cite{Bryant2006}.
In Theorem \ref{thm: second-order invariants} below we give an explicit basis for summands of $V_2(\mathfrak{su}(3)) $ which can be identified with differential forms, since these are precisely the terms we shall use to define $\SU(3)$-flows. The space of Riemann curvature tensors is given by
\begin{align}\label{equ: decomposition of space of  curvature}
\begin{split}
    V_2(\mathfrak{so}(6)) &\cong \mathbb{R} \oplus S^2_0 \oplus \mathcal{W} \\
    &\cong \mathbb{R} \oplus \big(\Lm^2_8 \oplus \Lm^3_{12}\big) \oplus \big( \mathbb{R} \oplus \Lm^2_6 \oplus \Lm^2_8 \oplus \Lm^3_{12} \oplus V_{2,1} \oplus V_{2,2}  \big)
    \end{split}
\end{align}
where $\mathcal{W}$ denotes the space of Weyl curvature tensors. 
We recall the following identification of the complexified $\SU(3)$-modules in terms of their weight spaces $V_{i,j}$: 
\begin{align*}
    V_{1,0} \oplus V_{0,1} &\cong \Lm^1_6\otimes \mathbb{C},\\
    V_{2,0}\oplus V_{0,2} &\cong \Lm^3_{12}\otimes \mathbb{C},\\
    V_{1,1} &\cong \Lm^2_8\otimes \mathbb{C}.
\end{align*}
Note that the curvature of a Calabi-Yau $3$-fold lies in the (real) $27$-dimensional space $V_{2,2}\subset \mathcal{W}$. Quadratic terms in the intrinsic torsion lie in 
\begin{equation*}
    S^2(V_{1}(\mathfrak{su}(3))) \cong 11 \mathbb{R} \oplus 13 \Lm^1_6 \oplus 17 \Lm^2_8 \oplus 12 \Lm^3_{12} \oplus 3 V_{3,0} \oplus 4 V_{2,2} \oplus 9 V_{2,1} \oplus 2 V_{3,1}.
\end{equation*}
Before describing an explicit basis for differential forms in the latter space, we
first introduce the quadratic bilinear form $Q$ defined by
\begin{equation*}
    Q(\alpha,\beta):= (e_i \ip \alpha) \w (e_i \ip \beta), 
\end{equation*}
where $\alpha,\beta$ are arbitrary differential forms.

\begin{theorem}[Quadratic first-order invariants]\label{thm: quadratic first-order invariants}\ 

    \noindent The 11 copies of $\mathbb{R}$ in $S^2(V_{1}(\mathfrak{su}(3)))$ are generated by:
    \begin{itemize}
        \item $\pi_0^2$, $\sigma_0^2$, $\pi_0 \sigma_0$.
        \item $|\pi_1|^2$, $|\nu_1|^2$, $g(\pi_1,\nu_1)$, $g(J\pi_1,\nu_1)$.
        \item $|\pi_2|^2$, $|\sigma_2|^2$, $g(\pi_2,\sigma_2)$.
        \item $|\nu_3|^2$.
    \end{itemize}
    The 17 copies of $\Lm^2_8$ in $S^2(V_{1}(\mathfrak{su}(3)))$ are generated by the projections of:
    \begin{itemize}
        \item $\pi_0\pi_2$, $\sigma_0\sigma_2$, $\pi_0 \sigma_2$, $\sigma_0 \pi_2$.
        \item $\pi_1 \w \nu_1$, $J\pi_1 \w \nu_1$, $\pi_1 \w J\pi_1$, $\nu_1 \w J\nu_1$.
        \item $*(\nu_1 \w \nu_3)$, $*(\nu_1 \w J\nu_3)$, $*(\pi_1 \w \nu_3)$, $*(J\pi_1 \w \nu_3)$.
        \item $*(\pi_2\w \pi_2)$, $*(\sigma_2\w \sigma_2)$, $*(\pi_2\w \sigma_2)$, $Q(\pi_2,\sigma_2)$.
        \item $*Q(\nu_3,\nu_3)$.
    \end{itemize}
    The 12 copies of $\Lm^3_{12}$ in $S^2(V_{1}(\mathfrak{su}(3)))$ are generated by the projections of:
    \begin{itemize}
        \item $\pi_0\nu_3$, $\sigma_0\nu_3$.
        \item $\pi_2 \w \nu_1$, $\pi_2 \w \pi_1$, $\sigma_2 \w \nu_1$, $\sigma_2 \w \pi_1$.
        \item $\pi_1 \w *(\pi_1 \w \Upsilon^+)$, $\nu_1 \w *(\nu_1 \w \Upsilon^+)$, $\pi_1 \w *(\nu_1 \w \Upsilon^+)$.
        \item $Q(\pi_2,\nu_3)$, $Q(\sigma_2,\nu_3)$
        \item $\widehat{Q}(\nu_3,\nu_3):=\Upsilon^+_{ijk}(e_j \ip e_i \ip \nu_3) \w (e_k \ip \nu_3)$.
    \end{itemize} 
    The 13 copies of $\Lm^1_{6}$ in $S^2(V_{1}(\mathfrak{su}(3)))$ are generated by:
    \begin{itemize}
        \item $\pi_0\nu_1$, $\sigma_0\nu_1$, $\pi_0\pi_1$, $\sigma_0\pi_1$.
        \item $\nu_1^\sharp \ip \pi_2$, $\nu_1^\sharp \ip \sigma_2$, $\pi_1^\sharp \ip \pi_2$, $\pi_1^\sharp \ip \sigma_2$.
        \item $\nu_1^\sharp \ip \nu_3$, $\pi_1^\sharp \ip \nu_3$.
        \item $\pi^2_6(\pi_1 \w \nu_1)$
        \item $*(\pi_2 \w \nu_3)$, $*(\sigma_2 \w \nu_3)$.
    \end{itemize} 
\end{theorem}

\begin{theorem}[Second-order invariants]\label{thm: second-order invariants}\

    \noindent The 3 copies of $\mathbb{R}$ in $V_{2}(\mathfrak{su}(3))$ are generated by:
    \begin{itemize}
        \item $\delta(\nu_1)$, $\delta(\pi_1)$ and $\delta(J\pi_1)$. (or equivalently by the trace of $\mathcal{L}_{\nu_1}g$, $\mathcal{L}_{\pi_1}g$,  $\mathcal{L}_{J\pi_1}g$) 
    \end{itemize}
    The 5 copies of $\Lm^2_8$ in $V_{2}(\mathfrak{su}(3))$ are generated by the projections of:
    \begin{itemize}
        \item $\delta(\nu_3)$.  
        \item $d(J\nu_1)$, $d(J\pi_1)$. (or we can equivalently use  $\mathcal{L}_{\pi_1}g$ and  $\mathcal{L}_{\nu_1}g$, see lemma below)
        \item $d\pi_1$, $d\nu_1$. (or equivalently we can use $\mathcal{L}_{J\pi_1}g$ and  $\mathcal{L}_{J\nu_1}g$, see lemma below) 
    \end{itemize}
    The 4 copies of $\Lm^3_{12}$ in $V_{2}(\mathfrak{su}(3))$ are generated by the projections of:
    \begin{itemize}
        \item $d\pi_2$, $d\sigma_2$.
        \item $\delta(\pi_1 \w \Upsilon^+)$, $\delta(\nu_1 \w \Upsilon^+)$. (or equivalently we can use $\mathcal{L}_{J\pi_1}g$ and  $\mathcal{L}_{J\nu_1}g$, see lemma below) 
    \end{itemize}
    The 4 copies of $\Lm^1_{6}$ in $V_{2}(\mathfrak{su}(3))$ are generated by the projections of:
    \begin{itemize}
        \item $d\pi_0$, $d\sigma_0$.
        \item $d\pi_1$, $d\nu_1$. (or equivalently we can use $\mathcal{L}_{\pi_1}\om$ and  $\mathcal{L}_{\nu_1}\om$) 
    \end{itemize}
\end{theorem}
\begin{remark} 
    The reader might find it helpful to compare Theorems \ref{thm: quadratic first-order invariants} and \ref{thm: second-order invariants} with the analogous classification of first- and second-order invariants in the $\mathrm{G}_2$ case given in \cite{fino2025}*{Theorem 9.1 and 9.2} and also \cite{Dwivedi2023}*{Theorem 6.2}.
\end{remark}

\begin{lemma}
Let $X$ denote an arbitrary vector field on $(M^6,\om,\Upsilon^\pm)$ and $X^\flat$ its dual $1$-form. Then the following identities hold:
\begin{align*}
  (\mathcal{L}_Xg)_1=&-\frac{1}{4}\delta X^\flat\\                 (\mathcal{L}_Xg)^2_8=&-\pi^2_8\Big(2d(JX^\flat)-2*(X^\flat \w J\nu_3)-2X^\flat \w J\nu_1\Big)\\
  =& -\pi^2_8\Big(
    2*d*(X^\flat \w \om)-4 X^\flat \w J\nu_1
    \Big)\\                 (\mathcal{L}_Xg)^3_{12}=&-\pi^3_{12}\Big(2*d*(X^\flat \w \Upsilon^+) + 2 \sigma_2 \w X^\flat-2*(X^\flat \w *(\pi_1 \w \Upsilon^+))
    \Big)
\end{align*}
In particular, this applies to $X^\flat=\pi_1, J\pi_1,\nu_1,J\nu_1$.
\end{lemma}
\subsection{Useful curvature formulae}\label{sec: appendix curvature formulae}

We first introduce the symmetric tensor $\mathrm{Ric}_F$ given in coordinates by
$$
(\mathrm{Ric}_F)_{ab}=R_{ijkl}\Upsilon^+_{ija}\Upsilon^+_{klb}.
$$
It is called the $F$-Ricci curvature and its trace $s_F$ is called the $F$-scalar curvature.
\begin{proposition}\label{prop: ricci torsion forms 0}
    The Riemannian scalar curvature is given by
    \begin{equation*}
        s = \frac{15}{2} \pi_0^2 + \frac{15}{2} \sigma_0^2 + 2\delta(\pi_1)+2\delta(\nu_1) - |\nu_1|^2- \frac{1}{2}|\sigma_2|^2 - \frac{1}{2}|\pi_2|^2 -\frac{1}{2}|\nu_3|^2+4g(\pi_1,\nu_1).
    \end{equation*}
   The $*$-scalar curvature is given by
    \begin{equation*}
        s^* = \frac{3}{2} \pi_0^2 + \frac{3}{2} \sigma_0^2 + 2\delta(\pi_1)-2\delta(\nu_1) -5 |\nu_1|^2+ \frac{1}{2}|\sigma_2|^2 + \frac{1}{2}|\pi_2|^2 -\frac{1}{2}|\nu_3|^2+4g(\pi_1,\nu_1).
    \end{equation*}
    The $F$-scalar curvature is given by
    \begin{equation*}
        s_F = -24 \pi_0^2 - 24 \sigma_0^2 -16\delta(\nu_1) -16 |\nu_1|^2+ 4|\sigma_2|^2 + 4|\pi_2|^2.
    \end{equation*}
In particular, we see that
\[
s_F=4(s^*-s).
\]
\end{proposition}

Recall from \cite{bedulli2007} the space of symmetric 2-tensors decompose into $\SU(3)$ modules as:
\[
\Sigma^2 = \langle g \rangle \oplus \Sigma^2_+ \oplus \Sigma^2_-,
\]
where 
\begin{align*}
\Sigma^2_+&=\{h\in S^2(\Lm^1) \ | \ h(\cdot,\cdot)=+h(J\cdot,J\cdot),\ g(g,h)=0\},\\
\Sigma^2_-&=\{h\in S^2(\Lm^1) \ | \ h(\cdot,\cdot)=-h(J\cdot,J\cdot)\}.
\end{align*}
In what follows we shall identify the latter spaces with differential forms using the diamond operator $\diamond$ by:
\begin{align*}
    \Sigma^2_+ &\xrightarrow{\simeq} \Lm^2_{8}\\
    h &\mapsto h \diamond \omega:= h_{ij} e^i \w (e_j \ip \om)
\end{align*}
and
\begin{align*}
    \Sigma^2_- &\xrightarrow{\simeq} \Lm^3_{12}\\
    h &\mapsto h \diamond \Upsilon^+:= h_{ij} e^i \w (e_j \ip \Upsilon^+).
\end{align*}
Using the above identifications we can now express the various Ricci tensors in terms of the torsion forms:
\begin{proposition}\label{prop: ricci torsion forms 1}
    The traceless $\Lm^2_8$-component of Ricci curvature is given by
    \begin{align*}
        (\mathrm{Ric}(g))^2_8 = 2\pi^2_8\Big( \
        &\frac{1}{2} \delta(\nu_3)+\frac{1}{2}\delta(\nu_1 \w \om) + d(J\pi_1) - *(\nu_1 \w J\nu_3) + \frac{1}{4}*(\pi_2 \w \pi_2)\\
        &+\frac{1}{4}*(\sigma_2 \w \sigma_2) -\frac{1}{4} \pi_0 \pi_2 -\frac{1}{4} \sigma_0 \sigma_2
        \Big).\nonumber
    \end{align*}
   The traceless $\Lm^2_8$-component of $*$-Ricci curvature is given by
    \begin{align*}
        (\mathrm{Ric}^*)^2_8 = 2\pi^2_8\Big( \
        &\frac{1}{2}\delta(\nu_1 \w \om) + \frac{1}{2}d(J\pi_1) -\frac{3}{4} *(\nu_1 \w J\nu_3) + 
        \frac{1}{16}*(\pi_2 \w \pi_2)
        +\frac{1}{16}*(\sigma_2 \w \sigma_2)\label{equ: ricci Lm^2_8}\\ 
        &-\frac{1}{16}Q(\nu_3,\nu_3)
        -\frac{1}{4} \pi_0 \pi_2 -\frac{1}{4} \sigma_0 \sigma_2
        -\frac{1}{8}Q(\pi_2,\sigma_2) + \frac{3}{8}(\nu_1 \w J\nu_1)\nonumber
        \Big).
    \end{align*}
    The traceless $\Lm^2_8$-component of $F$-curvature is given by
    \begin{align*}
        (\mathrm{Ric}_F)^2_8 = 2\pi^2_8\Big( \
        &2\delta(\nu_3)-2\delta(\nu_1 \w \om)  +2 *(\nu_1 \w J\nu_3) + 
        \frac{1}{2}*(\pi_2 \w \pi_2)
        +\frac{1}{2}*(\sigma_2 \w \sigma_2)\\ 
        &+\frac{1}{2}Q(\nu_3,\nu_3)
        + \pi_0 \pi_2 +\sigma_0 \sigma_2
        +Q(\pi_2,\sigma_2) -3(\nu_1 \w J\nu_1)
        \Big).\nonumber
    \end{align*}
\end{proposition}

\begin{remark}\label{remaark: riccistar and weyl}
    If we compare the above formula for $(\mathrm{Ric}(g))^2_8$ with that given in \cite{bedulli2007}*{Theorem 3.6}, we note that the terms $\delta(\pi_0 \Upsilon^-)$ and $\delta(\sigma_0 \Upsilon^+)$ therein have been replaced by $\pi_0 \pi_2$ and $\sigma_0 \sigma_2$; thus making it apparent that these are in fact first-order terms, \textit{not} second-order. 
\end{remark}

As with the scalar curvatures, observe that there is a linear relation between the highest order term of $(\mathrm{Ric}(g))^2_8$, $(\mathrm{Ric}^*)^2_8$ and $(\mathrm{Ric}_F)^2_8 $ as expected from \eqref{equ: decomposition of space of  curvature}:
\[
4(\mathrm{Ric}(g))^2_8 - 8 (\mathrm{Ric}^*)^2_8 - (\mathrm{Ric}_F)^2_8 = \text{ lower-order terms}
\]
It follows that the curvature terms $(\mathrm{Ric}^*)^2_8$ and $(\mathrm{Ric}_F)^2_8$ are in fact, to highest order, just linear combinations of $\Ric(g)$ and the Weyl curvature term in $\Lm^2_8$, see \eqref{equ: decomposition of space of  curvature}.

\begin{proposition}\label{prop: ricci torsion forms 2}
    The traceless $\Lm^3_{12}$-component of Ricci curvature is given by
    \begin{align}\label{equ: ricci Lm^3_12}
    \begin{split}
        (\mathrm{Ric}(g))^3_{12} =\frac{1}{2} \pi^3_{12}\Big( 
        &2d\sigma_2 -2 *d\pi_2 + 4*d*(\nu_1 \w \Upsilon^+) - 2\sigma_0 \nu_3 + 2 \pi_0 *\nu_3 + 2 \sigma_2 \w \pi_1\\ &+ 4 *(\pi_2 \w \nu_1) - 2\pi_2 \w \pi_1 -4 *(\nu_1 \w *(\pi_1 \w \Upsilon^+)) \\ &+2 *(\nu_1 \w *(\nu_1 \w \Upsilon^+)) -\frac{1}{2}\widehat{Q}(\nu_3, \nu_3)
        \Big).
        \end{split}
    \end{align}
    The traceless $\Lm^3_{12}$-component of $F$-curvature is given by
    \begin{align*}
        (\mathrm{Ric}_F)^3_{12} = \frac{1}{2}\pi^3_{12}\Big( 
        &-8d\sigma_2 -8 *d\pi_2 - 8\sigma_0 \nu_3 - 8 \pi_0 *\nu_3 -8  \sigma_2 \w \nu_1+ 8 \sigma_2 \w \pi_1\\ &- 8 *(\pi_2 \w \nu_1) + 8*(\pi_2 \w \pi_1) -4Q(\pi_2,\nu_3)+4*Q(\sigma_2,\nu_3)
        \Big).
    \end{align*}
    In particular, if the almost complex structure $J$ is integrable, i.e., $\pi_i=\sigma_i=0$, then $(\mathrm{Ric}_F)^3_{12}=0$. 
\end{proposition}
Note that $\Ric^*$ is omitted from the latter proposition since $(\Ric^*)^3_{12}=0$, essentially by definition. 
\begin{remark}
    The expression (\ref{equ: ricci Lm^3_12}) is equivalent to that given in \cite{bedulli2007}*{Theorem 3.6}, although we have chosen a different set of generators for quadratic first-order invariants and second-order invariants.
\end{remark}

Recall from \eqref{eq: nabla J_ nabla Upsilon}, we defined the intrinsic torsion $T\in \Lm^1 \otimes \mathfrak{su}(3)^\perp\cong \Lm^1 \otimes (\langle \om\rangle \oplus \Lm^2_6)$ by
\[
\nabla_X (\om,\Upsilon^\pm)= T(X) \diamond (\om,\Upsilon^\pm),
\]
where $X$ denotes an arbitrary vector field. On the other hand, we also saw above that one can describe the intrinsic torsion in terms of the differential forms $\{\pi_0,\sigma_0,\pi_1,\nu_1,\sigma_2,\pi_2,\nu_3\}$. The link between these two definitions is given explicitly by the following:
\begin{proposition}
    The intrinsic torsion $T\in  \Lm^1 \otimes (\langle \om\rangle \oplus \Lm^2_6)$ can be expressed in terms of the torsion forms as
    \begin{align*}
        T(X) =&\ \Big(X \ip \frac{1}{3}J(\pi_1 - \nu_1)\Big) \om \\
        &+  \pi^2_6\Big(
        \frac{1}{2} X \ip (J\nu_1 \w \om + J\nu_3)+ 
        \frac{1}{4} X \ip (\pi_0 \Upsilon^+ + \sigma_0 \Upsilon^-)\\
        &\qquad\ \ +\frac{1}{4}\big( (X\ip \sigma_2)^\sharp \ip \Upsilon^+ -(X\ip \pi_2)^\sharp \ip \Upsilon^-
        \big)\Big).
    \end{align*}
    In particular, we see that $\eta = -8 J(\pi_1 - \nu_1)$.
   Similarly, we have $VT=-\frac13\left(\pi_1+2\nu_1\right)$.
\end{proposition}

\begin{proposition}
    The divergence of $T$,  defined by $\Div T_{jk} : = \nabla_i T_{i,jk}$, can be expressed as
    \begin{equation*}
        \Div T = -\frac{1}{3} \delta (J\pi_1) \omega + \pi^2_6\big(\Div T\big) \in \langle \om \rangle \oplus \Lm^2_6,
    \end{equation*}
where $\pi^2_6\big(\Div T)= \alpha^\sharp \ip \Upsilon^+$ and the $1$-form $\alpha$ is given by
\begin{align*}
    \alpha =\ &\frac{3}{2}d\pi_0 + \frac{1}{6}\pi_0 \nu_1+\frac{1}{3}\pi_0\pi_1+\frac{1}{6}\nu_1^\sharp \ip \sigma_2 + \frac{1}{3} \pi_1^\sharp \ip \sigma_2 \\
    &+ \frac{1}{3} *( (\nu_1^\sharp \ip \nu_3 ) \w \Upsilon^+)+ \frac{1}{6} *( (\pi_1^\sharp \ip \nu_3 ) \w \Upsilon^+)\\
    &+J\Big(
    \frac{3}{2}d\sigma_0 + \frac{1}{2}*(d\pi_1 \w \Upsilon^+) + \frac{1}{2}*(d\nu_1 \w \Upsilon^+)+\frac{1}{6}\sigma_0 \nu_1 + \frac{1}{3}\sigma_0 \pi_1 \\
    &-\frac{1}{6} \nu_1^\sharp \ip \pi_2 - \frac{1}{3} \pi_1^\sharp \ip \pi_2 + \frac{1}{6} *(\pi_1 \w \nu_1 \w \Upsilon^+)
    \Big).
\end{align*}
\end{proposition}
\begin{remark}
    The reader might find it helpful to compare this with the analogous formula for $\Div T$ in the $\mathrm{G}_2$ case derived in  \cite{fino2025}*{Proposition 8.7}.
\end{remark}

\begin{corollary}\label{cor: harmonic su3 cases}
The $\SU(3)$-structure is harmonic in the following cases:
\begin{enumerate}
    \item The torsion forms $\pi_1,\nu_1$ vanish and $d\pi_0=-Jd\sigma_0$. In particular, a half-flat $\SU(3)$-structure is harmonic if and only if $\sigma_0$ is constant (for instance, this includes nearly K\"ahler).
    \item It is locally conformally Calabi-Yau i.e., there exists local function $f$ such that $3\nu_1=2\pi_1=df$ and all other torsion forms vanish.
    \item It is K\"ahler (i.e. the only non-zero torsion form is $\pi_1$) and $\delta (J\pi_1)=0$.
\end{enumerate}
\end{corollary}
\begin{remark}
    Similar results hold when $\H=\mathrm{G}_2$,  $\mathrm{Spin}(7)$ and $\mathrm{Sp}(2)\mathrm{Sp}(1)$, see \cite{Grigorian2019}*{Theorem 4.3} and \cite{udhav2022quaternionic}*{Theorem 3.3 and Corollary 3.4}.
\end{remark}

\subsection{Classification of second-order \texorpdfstring{$\SU(3)$}{}-flows}

We are now in a position to write down the most general second-order quasilinear flow of $\SU(3)$-structures. We follow the strategy outlined in \cite{udhav2024}*{Section 6.1}. Recall that the space of endomorphisms preserving $(\om,\Upsilon^\pm)$ is given by
$$\End(\mathbb{R}^6)/{\mathfrak{su}(3)}\cong \langle g \rangle \oplus \Sigma^2_+ \oplus \Sigma^2_- \oplus \langle \omega \rangle \oplus \Lm^2_6.$$
Comparing with \eqref{equ: V2su3}, we find that there is a $(3 \times 2 + 4 \times 2 + 5 \times 1+ 4\times 2) =27$-parameter family of $\SU(3)$-flows to second-order
\footnote{It was erroneously stated in \cite{udhav2024} that there is a $(3 \times 2 + 4 \times 1 + 5 \times 1+ 4\times 1) =19$-parameter family of $\SU(3)$-flows to second-order. The error stemmed from the fact that there is an automorphism of $\Sigma^2_-$ and $\Lm^1_6$ given by $J$ that was not taken into account therein.}. The combination of the results \S\S 4.1--4.2 yields:
\begin{theorem}\label{thm: general second order SU(3) flow}
There is a $27$-parameter family of second-order quasilinear $\SU(3)$-flow to highest order, and it is explicitly given by
\begin{equation}
    \frac{\partial}{\partial t}(\om,\Upsilon^\pm) = A \diamond(\om,\Upsilon^\pm),
\end{equation}
where 
\[
A := f_0 g + S_+ + S_- + f_1\omega + C \in \langle g \rangle \oplus \Sigma^2_+ \oplus \Sigma^2_- \oplus \langle \omega \rangle \oplus \Lm^2_6
\]
and 
\begin{enumerate}
    \item $f_0,f_1$ are linear combinations of $(\mathcal{L}_{VT}g)_1$, $(\mathcal{L}_{\eta}g)_1$, $(\mathcal{L}_{J\eta}g)_1$;
    \item $S_+$ is a linear combination of 
    $(\Ric(g))_+$, 
    $(\mathcal{L}_{\eta}g)_+$,
    $(\mathcal{L}_{J\eta}g)_+$,
    $(\mathcal{L}_{VT}g)_+$,
    $(\mathcal{L}_{J(VT)}g)_+$;
    \item $S_-$ is a linear combination of 
    $(\Ric(g))_-$, 
    $(\Ric(g))_-(J\cdot,\cdot)$, 
    $(\Ric_F)_-$, 
    $(\Ric_F)_-(J\cdot,\cdot)$, 
    $(\mathcal{L}_{\eta}g)_-$,
    $(\mathcal{L}_{J\eta}g)_-$,
    $(\mathcal{L}_{VT}g)_-$,
    $(\mathcal{L}_{J(VT)}g)_-$;
    \item $C$ is a linear combination of $\pi^2_6(\Div T)$, $\pi^2_6(\Div T)(J\cdot,\cdot)$, $\pi^2_6(\mathcal{L}_{\eta}\omega)$, $\pi^2_6(\mathcal{L}_{J\eta}\omega)$, $\pi^2_6(\mathcal{L}_{VT}\omega)$, $\pi^2_6(\mathcal{L}_{J(VT)}\omega)$, $d\sigma_0 \ip \Upsilon^+$, $d\sigma_0 \ip \Upsilon^-$.
\end{enumerate}
\end{theorem}
In similar contexts, it was shown in \cite{Dwivedi2023} that there is a $6$-parameter family of $\mathrm{G}_2$-flows, and in \cite{udhav2024} that there is a $72$-parameter family of $\SU(2)$-flows; as usual we implicitly mean that these $\H$-flows are second-order quasilinear and we are ignoring the lower-order terms, which typically are chosen from the space $S^2(V_1(\mathfrak{h}))$ given explicitly by Theorem \ref{thm: quadratic first-order invariants}.

%%%%%%%%%%%%%%%%%%%%%%%%%%%%%%%%%%%%%%%%%%%%%%%%%%%%%%%%%%%%%%%%%%%%%

\bibliography{references}
%===============================================================================

\bigskip

\noindent
(DF) Institute of Mathematics and Computer Sciences (ICMC), University of S\~ao Paulo (USP), 13566-590 S\~ao Carlos - SP, Brazil\\
\href{mailto:daniel.fadel@icmc.usp.br}{daniel.fadel@icmc.usp.br}

\medskip

\noindent
(UF) Institute of Mathematics, University of Warsaw, 
Banacha 2, 02-097 Warszawa, Poland\\ 
\href{mailto:u.fowdar@uw.edu.pl}{u.fowdar@uw.edu.pl}

\medskip

\noindent
(EL) Univ Brest, LMBA UMR 6205, F-29870 Brest, France \\ \href{mailto:loubeau@univ-brest.fr}{loubeau@univ-brest.fr}

\medskip

\noindent
(AM) Institut Camille Jordan, Université Claude Bernard Lyon 1, 69100 Villeurbanne, France.\\ \href{mailto:amoreno@math.univ-lyon1.fr}{amoreno@math.univ-lyon1.fr}

\medskip

\noindent
(HSE) Institute of Mathematics, Statistics and Scientific Computing (IMECC), University of Campinas (Unicamp), 13083-859 Campinas - SP, Brazil\\ \href{mailto:henrique.saearp@ime.unicamp.br}{henrique.saearp@ime.unicamp.br}

\end{document}

%% file: preamble.tex
\newcommand{\rmm}{{\rm m}}

\newcommand{\rr}{{\rm r}}

\newcommand{\rF}{{\rm F}}
\newcommand{\rG}{{\rm G}}

\newcommand{\rI}{{\rm I}}

\newcommand{\rR}{{\rm R}}

\newcommand{\bN}{{\bf N}}
\def\cE{\mathcal{E}}

\def\cP{\mathcal{P}}
\def\cQ{\mathcal{Q}}

\def\cT{\mathcal{T}}

\def\cX{\mathcal{X}}

\newcommand{\sX}{\mathscr{X}}

\newcommand{\fg}{{\mathfrak g}}

\newcommand{\fh}{{\mathfrak h}}

\newcommand{\fm}{{\mathfrak m}}

\newcommand{\fu}{{\mathfrak u}}

\def\bR{\mathbb R}

\def\bC{\mathbb C}

\newcommand{\N}{\bN}

\def\fgl{\mathfrak{gl}}
\def\fso{\mathfrak{so}}
\def\fsu{\mathfrak{su}}

\renewcommand{\H}{\rm H}

\newcommand{\SO}{{\rm SO}}

\newcommand{\SU}{{\rm SU}}
\newcommand{\GL}{\mathrm{GL}}

\newcommand{\U}{{\rm U}}

\def\G2{\mathrm{G}_2}
\def\Spin7{\mathrm{Spin(7)}}

\DeclareMathOperator\Div{div}

\DeclareMathOperator{\Spa}{span}

\newcommand{\End}{{\mathrm{End}}}

\renewcommand{\epsilon}{\varepsilon}

\newcommand{\Ric}{{\rm Ric}}

\DeclareMathOperator\tr{tr}

\DeclareMathOperator\vol{vol}

\newcommand{\sym}{\mathrm{sym}}

\newcommand{\Fr}{{\rm Fr}}

\def\om{\omega}

\def\w{\wedge}
\def\Lm{\Lambda}

\def\ip{\raise1pt\hbox{\large$\lrcorner\ \!$}}

\def\pt{\partial}

\def\g2{\varphi}
\def\s7{\Phi}
\def\ddt{\frac{d}{dt}}

\newcommand{\qandq}{\quad\text{and}\quad}
\newcommand{\qwithq}{\quad\text{with}\quad}
\newcommand{\qforq}{\quad\text{for}\quad}

\def\<{\mathopen{}\left<}
\def\>{\right>\mathclose{}}
\def\({\mathopen{}\left(}
\def\){\right)\mathclose{}}

\usepackage{multicol, color}

\definecolor{gold}{rgb}{0.85,.66,0}
\definecolor{cherry}{rgb}{0.9,.1,.2}
\definecolor{burgundy}{rgb}{0.8,.2,.2}
\definecolor{orangered}{rgb}{0.85,.3,0}
\definecolor{orange}{rgb}{0.85,.4,0}
\definecolor{olive}{rgb}{.45,.4,0}
\definecolor{lime}{rgb}{.6,.9,0}
\definecolor{green}{rgb}{.2,.7,0}
\definecolor{grey}{rgb}{.4,.4,.2}
\definecolor{brown}{rgb}{.4,.3,.1}

\newtheorem{theorem}{Theorem}[section]
\newtheorem*{theorem*}{Theorem}
\newtheorem{corollary}[theorem]{Corollary}
\newtheorem{lemma}[theorem]{Lemma}
\newtheorem{proposition}[theorem]{Proposition}
\newtheorem*{proposition*}{Proposition}

\newtheorem{example}[theorem]{Example}

\theoremstyle{definition}
\newtheorem{definition}[theorem]{Definition}
\newtheorem{remark}[theorem]{Remark}

\numberwithin{equation}{section}
\newtheorem{thmx}{Theorem}

\newtheorem{propx}[thmx]{Proposition}